\documentclass[EJP]{ejpecp} % replace ECP by EJP if needed.
\usepackage{enumerate}  % uncomment to use this package
\usepackage{subcaption}
\usepackage{bbm}
\usepackage{float}
\SHORTTITLE{Cluster-sizes in long-range percolation}

\TITLE{Cluster-size decay in supercritical long-range percolation
  } % \thanks is optional. Insert line breaks with \\

\AUTHORS{%
  Joost Jorritsma\footnote{Centre for Mathematics \& Computer Science (CWI), the Netherlands; Eindhoven University of Technology. \BEMAIL{joost.jorritsma@cwi.nl}}
  \and %% remove this line and below if single author
  J\'ulia Komj\'athy\footnote{Delft University of Technology, The Netherlands. \BEMAIL{j.komjathy@tudelft.nl}, \url{https://fa.ewi.tudelft.nl/\~komjathy/}}
   \and
   Dieter Mitsche\footnote{Institut Camille Jordan, Univ. Jean Monnet, Univ. de Lyon and IMC,  Pontif\'icia Univ. Cat\'olica de Chile. \BEMAIL{dmitsche@gmail.com}}
    }%AUTHORS
\KEYWORDS{Long-range percolation; cluster-size distribution; second-largest component; spatial random graphs} % Separate items with ;

\AMSSUBJ{60C05; 60K35} % Edit. Separate items with ;
\SUBMITTED{March 6, 2023} % Edit.
\ACCEPTED{TBA} % Edit.

\ARXIVID{2303.00712} % Edit.
\VOLUME{0}
\YEAR{2024}
\PAPERNUM{0}
\DOI{TBA}

\ABSTRACT{ We study the cluster-size distribution of supercritical long-range percolation on $\Z^d$, where two vertices $x,y\in\Z^d$ are connected by an edge with probability $\mathrm{p}(\|x-y\|):=p\min(1,\beta\|x-y\|)^{-d\alpha}$ for parameters $p\in(0, 1]$, $\alpha>1$, and $\beta>0$. We show that when $\alpha>1+1/d$, and either $\beta$ or $p$ is sufficiently large, the probability that the origin is in a finite cluster of size at least $k$ decays as $\exp\big(-\Theta(k^{(d-1)/d})\big)$. This corresponds to classical results for nearest-neighbor Bernoulli percolation on $\Z^d$, but is in contrast to long-range percolation with $\alpha<1+1/d$, when the exponent of the stretched exponential decay changes to $2-\alpha$. This result, together with our accompanying paper, establishes the phase diagram of long-range percolation with respect to cluster-size decay. Our proofs rely on combinatorial methods that show that large delocalized components are unlikely to occur. As a side result we determine the asymptotic growth of the second-largest connected component when the graph is restricted to a finite box.
}

\theoremstyle{plain}
\newtheorem{statement}[theorem]{Statement}

\newcommand{\R}{{\mathbb R}}
\newcommand{\Z}{{\mathbb Z}}
\newcommand{\N}{{\mathbb N}}
\newcommand{\E}{\mathbb E}
\newcommand{\Prob}{\mathbb{P}}

\newcommand{\CA}{{\mathcal A}}
\newcommand{\CB}{{\mathcal B}}

\newcommand{\CC}{{\mathcal C}}

\newcommand{\CD}{{\mathcal D}}
\newcommand{\CE}{{\mathcal E}}

\newcommand{\CF}{{\mathcal F}}
\newcommand{\CG}{{\mathcal G}}
\newcommand{\CH}{{\mathcal H}}

\newcommand{\CL}{{\mathcal L}}
\newcommand{\CM}{{\mathcal M}}

\newcommand{\CS}{{\mathcal S}}
\newcommand{\CT}{{\mathcal T}}

\newcommand{\CZ}{{\mathcal Z}}
\newcommand{\re}{{\mathrm e}}

\newcommand{\ind}[1]{\mathbbm{1}_{\{#1\}}}

\newcommand{\rd}{\mathrm{d}}
\newcommand{\sss}[1]{{\scriptscriptstyle #1}}

\usepackage{todonotes}

\usepackage[symbols,nogroupskip,sort=none]{glossaries-extra}
\usepackage{glossary-mcols}

\glsxtrnewsymbol[description={closure of ($1$-connected) block $A$}]{closure}{\phantom{....}\hspace{.7pt}\ensuremath{\bar A}\phantom{..}}
\glsxtrnewsymbol[description={\phantom{...}$1$-disconnected elements}]{in1}{\phantom{..}\ensuremath{\in_1}}
\glsxtrnewsymbol[description={interior boundary w.r.t.\  $\Lambda_n$}]{intL}{\ensuremath{\partial_\mathrm{int}(\cdot)}}
\glsxtrnewsymbol[description={interior boundary w.r.t.\  $\Z^d$}]{intZ}{\ensuremath{\widetilde\partial_\mathrm{int}(\cdot)}}
\glsxtrnewsymbol[description={exterior boundary w.r.t.\  $\Lambda_n$}]{extL}{\ensuremath{\partial_\mathrm{ext}(\cdot)}}
\glsxtrnewsymbol[description={exterior boundary w.r.t.\  $\Z^d$}]{extZ}{\ensuremath{\widetilde\partial_\mathrm{ext}(\cdot)}}
\glsxtrnewsymbol[description={minimum of $a,b\in\R$}]{min}{\ensuremath{a\wedge b}}
\glsxtrnewsymbol[description={maximum of $a,b\in\R$}]{max}{\ensuremath{a\vee b}}
\glsxtrnewsymbol[description={edge-existence between $u,v$ in $G$}]{edge}{\ensuremath{u\sim_G v}}
\glsxtrnewsymbol[description={path-existence between $u,v$ in $G$}]{path}{\ensuremath{u\leftrightarrow_G v}}
\glsxtrnewsymbol[description={Euclidean 2-norm}]{norm}{\ensuremath{\|\cdot\|}}
\begin{document}

%%%%%%%%%%%%%%%%%%%%%%%%%%%%%%%%%%%%%%%%%%%%%%%%%%%%%%%%%%%%%%%%%%%
%%                                                               %%
%% No need for \maketitle.                                       %%
%%                                                               %%
%%%%%%%%%%%%%%%%%%%%%%%%%%%%%%%%%%%%%%%%%%%%%%%%%%%%%%%%%%%%%%%%%%%

%%%%%%%%%%%%%%%%%%%%%%%%%%%%%%%%%%%%%%%%%%%%%%%%%%%%%%%%%%%%%%%%%%%
%%                                                               %%
%% Please replace what follows by the body of your article       %%
%% (up to the bibliography):                                     %%
%%                                                               %%
%%%%%%%%%%%%%%%%%%%%%%%%%%%%%%%%%%%%%%%%%%%%%%%%%%%%%%%%%%%%%%%%%%%
\section{Introduction}
For nearest-neighbor Bernoulli percolation on $\Z^d$ \cite{broadbent1957percolation} it is well known \cite{aizenman1980lower, alexander1990wulff,cerf2000large, grimmett1990supercritical, kesten1990bondpercolation, kunz1978essential} that in the supercritical case, the distribution of the number of vertices in the cluster containing the origin follows subexponential decay. Let us write $|\CC(0)|$ for the number of vertices in the cluster containing the origin, and assume $p > p_c(\Z^d)$. Then it holds that
\begin{equation}\label{eq:intro-basic}
 \Prob(k\le |\CC(0)|< \infty)      =
 \exp\big(-\Theta(k^{\frac{d-1}{d}})\big).
\end{equation}
The decay rate in \eqref{eq:intro-basic} -- stretched exponential decay with exponent $(d-1)/d$ -- can be intuitively explained as follows: a cluster $\CC$ with at least $k$ vertices has at least $\Theta(k^{(d-1)/d})$ edges on its (outer) boundary. All these edges need to be absent. In other words, the tail decay in \eqref{eq:intro-basic} is driven by \emph{surface tension}.
Recently, this result was extended to supercritical Bernoulli percolation on certain classes of transitive graphs \cite{contreras2021supercritical,hutchcroft2022transience}.
Related works are also \cite{lichev2022bernoulli,penrose2022giant}, which determine the size of the second-largest component in a finite box for random geometric graphs, also known as continuum percolation, obtaining the same exponent $(d-1)/d$ for the cluster-size decay.

In our accompanying papers \cite{clusterI, clusterII}, we study the supercritical cluster-size decay in a large class of spatial random graph models where at least one of the \emph{degree distribution} and the \emph{edge-length distribution} obey \emph{heavy tails}: long-range percolation \cite{longrangeUnique1987,schulman_1983}, scale-free percolation on $\Z^d$ and in the continuum \cite{DeiHofHoo13, DepWut18};  geometric inhomogeneous random graphs \cite{BriKeuLen19}, hyperbolic random graphs \cite{krioukov2010hyperbolic}, the ultra-small scale-free geometric network \cite{yukich_2016}; the scale-free Gilbert model \cite{hirsch2017gilbertgraph}, the Poisson Boolean model with random radii \cite{gouere2008subcritical}, the age- and the weight-dependent random connection models \cite{gracar2019age, GraHeyMonMor19}.

In these models the tail in \eqref{eq:intro-basic} stays still stretched exponential, but is at least as light as the right-hand-side of \eqref{eq:intro-basic}. The new exponent -- say $\zeta$ -- is at least $(d-1)/d$, with its formula depending on the model parameters. Generally speaking, the accompanying papers \cite{clusterI, clusterII} treat cluster-sizes whenever the decay is \emph{strictly} lighter than the right-hand side of \eqref{eq:intro-basic}. There, we leave the part of the phase diagram \emph{open} where the model parameters are such that the conjectured exponent in \eqref{eq:intro-basic} stays $(d-1)/d$ and the tail decay is driven by surface tension as in nearest-neighbor percolation.

 In particular, the paper \cite{clusterII} leaves open this region for long-range percolation (LRP) \cite{longrangeUnique1987,schulman_1983} (see Theorem \ref{thm:complementary} below for the result). This missing region for LRP is the main focus in this paper. We write $x\wedge y:=\min(x,y)$. 
 \begin{figure}
     \centering
     \includegraphics[width=0.45\textwidth]{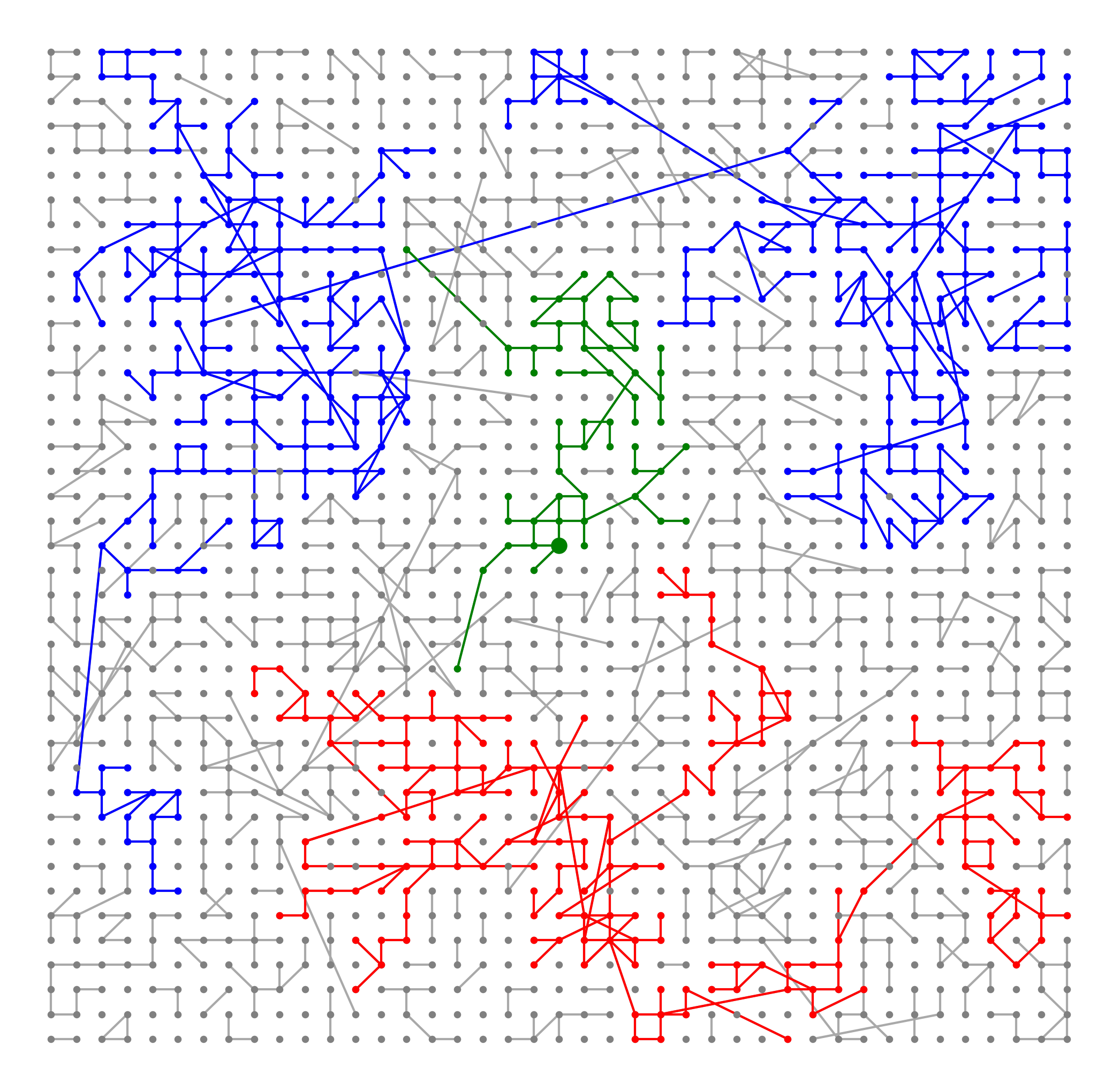}
     \caption{Simulation of long-range percolation in dimension $2$ restricted to a finite box. The largest connected component $\CC_n^\sss{(1)}$ is colored blue, the second-largest connected component $\CC_n^\sss{(2)}$ red, and the component containing the origin $\CC(0)$ green.}
     \label{fig:lrp}
 \end{figure} 
 \begin{definition}[Long-range percolation (LRP)]\label{def:lrp}
 \it
 Fix constants $d\in\mathbb{N}, \alpha>1$, $p\in(0,1]$, and $\beta>0$.
 We consider the random graph
 $\CG_\infty=(V(\CG_\infty), E(\CG_\infty))$ with $V(\CG_\infty)=\Z^d$ such that each edge $\{x, y\}$ is included in $E(\CG_\infty)$, independently of all the other edges, with probability
\begin{equation}\label{eq:connection-prob-gen}
 \mathrm{p}\big(\|x-y\|\big):=
  p\cdot \bigg(1 \wedge \frac{\beta}{\|x-y\|}\bigg)^{d \alpha},
  \end{equation}
 where $\|\cdot\|:=\|\cdot\|_2$ denotes the (Euclidean) $2$-norm.
 We set $\Lambda_n:=\Z^d\cap[-n^{-1/d}/2, n^{1/d}/2)^d$ and $E_n:=\{\{x,y\}\in E_\infty: \{x,y\}\subseteq \Lambda_n\}$, and write $\CG_n:=(\Lambda_n, E_n)$ for the induced subgraph of $\CG_\infty$ on $\Lambda_n$.
 We write $\CC(0)$ and $\CC_n(0)$ for the connected component containing the origin in $\CG_\infty$ and $\CG_n$, respectively.
\end{definition}
\noindent
Each edge with length at most $\beta$ is present with probability $p$, resembling spread-out percolation in this length-range. 
The long-range nature is apparent beyond radius $\beta$. Our parametrization with $d\alpha$ as the power of the polynomial decay matches the notation of the accompanying paper~\cite{clusterI}. Other common parametrizations of this power are $s$~\cite{biskup2004scaling, coppersmith2002diameter, crawford2012simple} and $d+\alpha$~\cite{hutchcroft2021power}.

Throughout the paper we will assume high edge density (that we make precise below). High edge density implies that the graph is supercritical, and
$$
 \Prob\big(0\leftrightarrow\infty\big)>0
$$
holds. Further, the infinite component is almost surely unique \cite{longrangeUnique1987, GanKeaNew92}.
We write $\CC_n^\sss{(i)}$ for the $i$-th largest component in $\CG_n$ with $i\in\{1,2\}$. If $\CC_n^{\sss{(1)}}$ contains all vertices of $\CG_n$ then we set $\CC_n^{\sss{(2)}}:=\emptyset$. For $n=\infty$, we write $\CC_\infty^\sss{(1)}$ for the unique infinite component in $\CG_\infty$. We refer to Figure~\ref{fig:lrp} for a visualization.
We will generally assume that $p\wedge \beta<1$, so that not all nearest-neighbor edges are present and hence the graph is not connected almost surely. 
We state our main result.
\begin{theorem}[Second-largest component and cluster-size decay]\label{thm:longrange}
 Consider supercritical long-range percolation on $\Z^d$ for $d\ge 2$ and  $\alpha>1+1/d$, and   $p\wedge \beta<1$. If either $\beta\ge 1$ and $\beta$ sufficiently large (depending on $p, \alpha, d$), or $\beta<1$ and  $p(1\wedge \beta)^{d\alpha}$ is sufficiently close to $1$, then there exist constants $A,\delta>0$ such that for all $n$ sufficiently large,
 \begin{equation}\label{eq:main-second-largest}
  \Prob\big(\tfrac{1}{A}(\log n)^{d/(d-1)}\le |\CC_{n}^\sss{(2)}| \le A(\log n)^{d/(d-1)}\big) \ge 1-n^{-\delta}.
 \end{equation}
 Under the same assumptions, for all $k$ sufficiently large, whenever $n(\log n)^{-2d/(d-1)}\ge k$ or $n=\infty$,
 \begin{equation}\label{eq:main-cluster-decay}
  -k^{-(d-1)/d}\log\big(\Prob\big(|\CC_{n}(0)| \ge k, 0\notin\CC_n^\sss{(1)}\big)\big) \in
  [1/A, A].
 \end{equation}
 Lastly, under the same assumptions, 
 \begin{equation}
     \frac{|\CC_{n}^\sss{(1)}|}{n}\overset{\Prob}\longrightarrow \Prob\big(0\leftrightarrow\infty\big),\qquad\mbox{as }n\to\infty.\label{eq:lln}
 \end{equation}
\end{theorem}
Note that \eqref{eq:main-cluster-decay} allows for $n=\infty$. In Remark~\ref{remark:general} we discuss a further generalization to more general connectivity functions $\mathrm{p}(\|x-y\|)$.
Theorem \ref{thm:longrange} complements the result of \cite{clusterII} applied to long-range percolation that we state here for completeness. 
\begin{theorem}[Complementary result for $\alpha<1+1/d$ {\cite{clusterII}}]\label{thm:complementary}
 Consider supercritical long-range percolation on $\Z^d$ for $d\ge 1$ and $\alpha<1+1/d$. There exist constants $A,\delta>0$ such that for all $n$ 
 \begin{equation}\label{eq:complem-second}
  \Prob\big(\tfrac{1}{A}(\log n)^{1/(2-\alpha)}\le |\CC_{n}^\sss{(2)}| \le A(\log n)^{1/(2-\alpha)}\big) \ge 1-n^{-\delta}.
 \end{equation}
 Moreover, for all $k$ sufficiently large, whenever $n\ge  Ak$ or $n=\infty$,
\begin{equation}\label{eq:complem-decay}
  -k^{-(2-\alpha)}\log\big(\Prob\big(|\CC_{n}(0)| \ge k, 0\notin\CC_n^\sss{(1)}\big)\big) \in
  [1/A, A].
 \end{equation}
 Lastly, under the same assumptions, 
 \begin{equation}
     \frac{|\CC_{n}^\sss{(1)}|}{n}\overset{\Prob}\longrightarrow \Prob\big(0\leftrightarrow\infty\big),\qquad\mbox{as }n\to\infty.\nonumber
 \end{equation}
\end{theorem}
In Theorem \ref{thm:complementary} we do not require $\beta$ or $p$ to be sufficiently large, and we also allow one-dimensional models: when $d=1$, LRP is supercritical when $\alpha\le2=1+1/d$ and when $p$, $\beta$ are sufficiently large~\cite{lrpRevisited20, schulman_1983}.
When $d=1$ and $\alpha>2$, LRP is subcritical for any $p,\beta>0$ such that $\mathrm{p}(1)<1$~\cite{schulman_1983}, so Theorems \ref{thm:longrange} and \ref{thm:complementary} together give a complete picture for the cluster-size decay for supercritical long-range percolation (under the additional assumption that $\beta$ or $p$ is sufficiently large when $\alpha>1+1/d$).
In \cite{clusterII}, we also study the phase boundary $\alpha=1+1/d$. In that case the lower bounds \eqref{eq:complem-second} and \eqref{eq:complem-decay} contain lower-order correction factors, which we conjecture to be sharp. We omit further details here.
To the extent of our knowledge, for LRP, the only related results regarding the distribution of smaller clusters in \emph{supercritical} LRP is an upper bound on the second-largest component with unidentified exponent by Crawford and Sly \cite{crawford2012simple} for $\alpha\in(1,2)$ in dimension $1$ and $\alpha\in(1, 1+2/d)$ in dimensions $2$ and higher. For \emph{(sub)critical} LRP with $\alpha\in(1,2)$, a polynomial upper bound on $\Prob(|\CC(0)|\ge n)$ is established in \cite{hutchcroft2021power}.

Before proceeding to the technical contributions, we remark that our results could be generalized to a more general class of random graph models on $\Z^d$. Theorem~\ref{thm:longrange} extends to random graph models on $\Z^d$ with independent edges for any connectivity function that has a lighter tail than $\mathrm{p}$ in Theorem~\ref{thm:longrange}, provided that the probability of `short-range' edges is still sufficiently high. For instance, our methods extend to long-range percolation models in which the connection probability decays superpolynomially, and also provide an alternative proof for the cluster size decay in spread out percolation, a special case in  \cite{contreras2021supercritical} (in spread-out percolation two vertices within distance $R$ connect independently by an edge with probability $p$. 
We refrain from proving the result in this generality, since it would require many technically involved changes in our already technical companion paper~\cite{clusterI}. We nevertheless formulate the following comment.
\begin{remark}\label{remark:general}\normalfont
Consider the percolation model on $\Z^d$ where each pair of vertices $x, y\in\Z^d$ is connected by an edge with probability $\mathrm{p}(\|x-y\|)$ for some function $\mathrm{p}:[0,\infty)\to[0,1)$, independently of other vertex pairs.
Let $J:[0,\infty)\to[0,1) $ be a function that satisfies $\sup_{r>0}J(r)<1$, and\begin{equation}
\int_{x: x\in \R^d} \|x\|J(\|x\|) \rd x<\infty.\label{eq:int-cond}
\end{equation}
Then we have the following two cases:
\begin{enumerate}
\item[1)] If the connectivity function $\mathrm{p}$ is of the form 
\begin{equation}
    \mathrm{p}(\|x\|)=J(\|x\|/\beta),\nonumber
\end{equation}
and there is an $\varepsilon>0$ such that $J(x)>\varepsilon$ whenever $x<\varepsilon$,
then \eqref{eq:main-second-largest}--\eqref{eq:lln} can be proven for all sufficiently large $\beta$ depending on $\varepsilon$. 
\item[2)] If the connectivity function is of the form 
\begin{equation}
    \mathrm{p}(\|x\|)=\begin{cases} p & \mbox{if } \|x\|=1,\\
    J(\|x\|) &\mbox{if } \|x\|>1,\end{cases}\nonumber
\end{equation}
then \eqref{eq:main-second-largest}--\eqref{eq:lln} can be proven for all $p$ sufficiently close to $1$.
\end{enumerate}
\end{remark}
The integral in the condition~\eqref{eq:int-cond} represents the order of the expected number of edges $\{x,y\}$ for which the line-segment $(x,y)$ crosses a fixed box of volume one. If this number is finite, Theorem~\ref{thm:longrange} holds in more generality. 
The connectivity function $\mathrm{p}$ from Definition~\ref{def:lrp} satisfies the integrability condition \eqref{eq:int-cond} if and only if $\alpha>1+1/d$.
We conjecture that the \emph{upper} bounds in \eqref{eq:main-second-largest}--\eqref{eq:main-cluster-decay} remain valid if~\eqref{eq:int-cond} is violated (but no longer match the lower bounds). However, this would require a different proof technique.

We state a proposition that contains the main technical contribution of this paper. Together with statements from our companion paper \cite{clusterI}, where we establish the relation between the second-largest component and the cluster-size decay for spatial random graph models more generally, this proposition will readily imply Theorem~\ref{thm:longrange}.
\begin{proposition}[Second-largest component, upper bound]\label{prop:second-largest}
 Consider supercritical long-range percolation on $\Z^d$ for $\alpha>1+1/d$, $d\ge 2$. If $\beta$ in \eqref{eq:connection-prob-gen} is sufficiently large (depending on $p, \alpha, d$), or if $p(1\wedge\beta)^{d\alpha}$ is sufficiently close to $1$, then there exists a constant $A>0$ such that for all $k$ sufficiently large and for all $n$ satisfying
 $n(\log n)^{-2d/(d-1)}\ge k$,
 $$
  \Prob\big(|\CC_{n}^\sss{(2)}| \ge k\big) \le (n \log n)\exp\big(-k^{(d-1)/d}\big).
 $$
\end{proposition}
We believe that Proposition \ref{prop:second-largest}, and hence Theorem \ref{thm:longrange}, should hold for any values of $\beta, p$ that lead to a supercritical graph: however, this would require non-trivial adaptations of our proof techniques. 

\subsection{Idea of proof}\label{sec:sketch} The proof of Proposition \ref{prop:second-largest} relies on a careful first-moment analysis in which we count all possible candidates of isolated components of size at least $k$. The starting point is the classic isoperimetric inequality which states that any set $\CS$ of at least $k$ vertices has an edge-boundary of size $|\partial \CS|=\Omega(k^{(d-1)/d})$. These edges need to be absent when $\CS$ is a connected component (or simply component below), i.e., detached from the rest of the graph. The combinatorial difficulty arises when we account for all possible candidate components $\CS$: the structure of $\CS$ is more complex than for nearest-neighbor bond percolation in $\mathbb{Z}^d$, since $\CS$ can be ``delocalized'' in space.
The second difficulty arises in the finite box $\Lambda_n\subseteq\mathbb{Z}^d$, where we need to take boundary effects into account caused by possibly shared boundaries of $\partial \CS $ and $\partial \Lambda_n$.

To resolve these two complications, we distinguish two types of components: the first type consists of several ``blocks'' connected by long edges: each block is a connected subset of $\mathbb{Z}^d$ (with respect to nearest-neighbor relation in $\Z^d$). We consider each possible combination of blocks  with fixed total outer edge-boundary size $m$, and give an upper bound on the probability that these blocks form a connected component by counting all possible spanning trees on these blocks.
We show that the combinatorial factor arising from counting all potential components with boundary $m$ is at most exponential in $m$.
We then use the large value of $\beta$ or $p$ in our favor to prove that the probability that such a component is formed and \emph{isolated} is sufficiently small.

The second type of potential component $\CS$ contains a large block that has a large overlap with the boundary of $\Lambda_n$, and consequently $|\partial \CS|$ inside $\Lambda_n$ may be small. Again, a simple enumeration of all such blocks would yield too large combinatorial factors. Instead, we use only certain holes of $\CS$ and an adapted isoperimetric inequality to still ensure that many edges need to be absent. Then we group potential large components, such that a large number of holes coincide for each potential large component in the same group. This makes combinatorial factors much smaller, and we obtain the right decay.
\subsection*{Organization}
In Section \ref{sec:main-thm}, we  derive an intermediate upper bound for $\Prob\big(|\CC_n^\sss{(2)}|\ge k\big)$, defining the two types of components formally. Then, we state two lemmas and show that  they imply Proposition \ref{prop:second-largest}. We prove the two lemmas in separate sections. In the last section we use the result of Proposition \ref{prop:second-largest} to prove Theorem \ref{thm:longrange}.
\subsection*{Notation}
Let $H=(V_H,E_H)$ be a graph.
For two sets $A, B\subseteq V_H$, we write $A\sim_H B$ if there exist $x \in A, y \in B$ such that $\{x,y\} \in E_H$, and $A\not\sim_H B$ if no such pair exists. We leave out the subscript $H$ if the graph is clear from the context.
For $A\subseteq V_H$, we write $H[A]$ for the induced subgraph of $H$ on vertices in $A$. Denote by $\Z^d_\infty$ the graph on the vertex set $\Z^d$ and an edge between $x, y\in\Z^d$ if and only if $\|x-y\|_\infty=1$. Similarly, let $\Z^d_1$ be the graph on the vertex set $\Z^d$ and an edge between $x, y\in\Z^d$ if and only if $\|x-y\|_1=1$. As already mentioned, we write $\|\cdot\|:=\|\cdot\|_2$. For two sets $A, B \subseteq \mathbb{Z}^d$, denote by $\|A-B\|_p = \min\{\|x-y\|_p \mid x\in A, y \in B\}$. For any graph structure, we say that a path $\pi=(v_1, v_2, v_3, \dots)$ is \emph{self-avoiding} if its vertices are all distinct. For $x,y\in \R$, we write $x\wedge y:=\min(x,y), x\vee y:=\max(x,y)$. 
We sometimes abuse notation and write $\cup_{i\ge1}A_i$ rather than $\bigcup_{i\ge 1}A_i$ for the union.
We also refer to Definition~\ref{def:boundary} below for the exterior boundary of $A\subseteq\Z^d$ with respect to either $\Lambda_n$ or $\Z^d$, denoted by $\partial_\mathrm{ext}A$, resp.\ $\widetilde\partial_\mathrm{ext}A$, and the interior boundaries $\partial_\mathrm{int}A$ and $\widetilde\partial_\mathrm{int}A$ with respect to $\Lambda_n$, resp.\ $\Z^d$.

\section{Preliminaries and setup}
\label{sec:main-thm}
Throughout the rest of the paper, we assume that $d\ge 2$, $\alpha>1+1/d$, and $n^{1/d}\in\N$. 
This latter assumption means that the box $\Lambda_n$ contains exactly $n$ vertices and avoids integer parts in our formulas. Adaptation to arbitrary $n$ is straightforward.

We now formalize the concepts from the proof outline in Section \ref{sec:sketch} that eventually lead to two lemmas, one for each of the two described types of components. To ensure that the upcoming definitions naturally follow each other, we will postpone the (sometimes standard) proofs of intermediate claims to the appendix.
We start with a definition to describe sets of $\Lambda_n$ that  form (subsets of) the second-largest component.
\begin{definition}[Connected sets and blocks]\label{def:blocks}
 We call a non-empty set $A\subseteq \Z^d$ of vertices \hypertarget{1conn}{\emph{$1$-connected}} or a \emph{block}, if the graph $\Z^d_1[A]$ consists of a single connected component. We similarly define $A$ being \emph{$\ast$-connected} if the graph  $\Z^d_\infty[A]$ consists of a single connected component.
 We write
 \begin{equation}\label{def:calA}
  \begin{aligned}
   \CA:=\big\{A\subseteq\Lambda_n \mid A\mbox{ is $1$-connected}\big\},\qquad
   \CA_\ast:=\big\{A\subseteq\Lambda_n \mid A\mbox{ is $\ast$-connected}\big\}.
  \end{aligned}
 \end{equation}
 A sequence of at least two sets $A_1,A_2,\ldots \subseteq\Z^d$ is \emph{$1$-disconnected} if $\|A_i-A_j\|_1>1$ for all $i\neq j$. We say that a set $A\subseteq \Lambda_n$ \emph{consists of blocks} $A_1,\ldots,A_b$ if $A_i$ is a (non-empty) block for $1\le i \le b$, if the sequence $(A_i)_{i\le b}$ is $1$-disconnected, and their union equals $A$. We write $A_1,\ldots, A_b\hypertarget{1el}{\in_1}\CA$ if the blocks $A_1,\ldots,A_b$ are $1$-disconnected.

 We say that a vertex $x$ is \emph{surrounded} by $A\in\CA$ if each infinite $1$-connected self-avoiding path starting from $x$ contains a vertex of $A$. We define for $A\in\CA$ its \emph{closure} $\bar{A}$  as
 \begin{equation}\label{eq:closure}
  \hypertarget{closure}{\bar{A}}=A\cup\{x\in\Z^d: x\mbox{ surrounded by }A\}.
 \end{equation}
 We call the maximal $1$-connected subsets of $\bar A\setminus A$ the holes of $A$, and write $\mathfrak{H}_A$ for the collection of holes.
\end{definition}
We make a few comments. The `vertices on the boundary' of $A$, that we shall shortly define, are not surrounded by $A$, but they belong to both $A$ and $\bar A$. 
The closures of blocks will be used for the first type of components described in Section \ref{sec:sketch}.
Take now a block $A\subseteq \Lambda_n$. Then $x$ can only be surrounded by $A$ if $x\in\Lambda_n$, hence $A\subseteq\Lambda_n$ implies that $\bar{A}\subseteq\Lambda_n$.
Due to the presence of long-range edges, a component in long-range percolation may consist of multiple $1$-disconnected blocks (some of them possibly consisting of a single vertex). We define the notion of a block graph.
\begin{definition}[Block graph]\label{def:blockgraph}
 Let $A_1,\ldots, A_b \in_1\CA$  be a sequence of $1$-disconnected blocks, and consider a graph $G$ on vertices $V_G \supseteq	 \cup_{i\le b} A_i$. The \emph{block graph} $\CH_G((A_i)_{i\le b})=(V_{\CH_G}, E_{\CH_G})$ of $G$ on blocks $A_1,\ldots, A_b$ is defined as
 \begin{equation}
  V_{\CH_G}:=\{1,\ldots, b\}, \qquad E_{\CH_G}:=\big\{\{i, j\}: A_i\sim_G A_j\big\}.\nonumber%
 \end{equation}
\end{definition}
In words, the vertices of each block are contracted to a single vertex in the block graph, and two corresponding vertices for the blocks $i$ and $j$ are connected in the block graph if and only if there exists an edge in the original graph $G$ between (some vertices in) the two blocks. We continue with a simple claim, with proof in the appendix on page~\pageref{sec:unique-decomp}.
\begin{claim}[Unique block-decomposition of components]\label{claim:unique}
Let $A$ be any subset of vertices in $\Lambda_n$. There exists $b\in\N$ such that $A$ can be uniquely partitioned into a $1$-disconnected sequence $(A_i)_{i\le b}$  of blocks up to permutation of the blocks. Further, if $A$ is the vertex set of a connected component $\CC$ of $\CG_n$, then the block graph $\CH_{\CG_n}((A_i)_{i\le b})$ is connected.
 \end{claim}

Later, we will  enumerate subsets $\CS\subseteq \Lambda_n$ of vertices  that potentially form a component of LRP in $\Lambda_n$.
To ensure that a subset is isolated from the rest of the graph, there must be no  edge from $\CS$ to its "surrounding" \emph{inside} $\Lambda_n$. This motivates the following definition of boundaries with respect to $\Lambda_n$.
\begin{definition}[Boundaries]\label{def:boundary}
 Let $A\subseteq\Lambda_n$. We define the \emph{exterior boundary} of $A$ with respect to $\Lambda_n$ and $\Z^d$, respectively, as
 \begin{equation}\label{eq:ext-boundary}
  \hypertarget{extL}{\partial_\mathrm{ext}} A := \{x\in\Lambda_n:  \|\{x\}-A\|_1=1\},
  \qquad
  \hypertarget{extZ}{\widetilde\partial_\mathrm{ext}} A := \{x\in\Z^d: \|\{x\}-A\|_1=1\}.
 \end{equation}
 We define the \emph{interior boundary} of $A$ with respect to $\Lambda_n$ and $\Z^d$, respectively, as
 \begin{equation}\label{eq:int-boundary}
  \hypertarget{intL}{\partial_\mathrm{int}}A:= \{x\in A: \|\{x\}-\partial_\mathrm{ext} A\|_1=1\},\qquad
  \hypertarget{intZ}{\widetilde\partial_\mathrm{int}}A:= \{x\in A: \|\{x\}-\widetilde\partial_\mathrm{ext} A\|_1=1\}.
 \end{equation}
 If, in words, we mention the exterior or interior boundary of $A$ then  -- unless explicitly specified differently -- we mean with respect to $\Lambda_n$.
\end{definition}
We mention that $\widetilde\partial_{\mathrm{int}}\Lambda_n$ is the `usual' vertex boundary of $\Lambda_n$.
The boundary $\widetilde\partial_\mathrm{ext} A$ may contain vertices outside $\Lambda_n$, and will be useful in the enumeration of subsets forming isolated components below. It may happen that a block $A$ contains (many) vertices of $\widetilde\partial_{\mathrm{int}}\Lambda_n$. On such regions, $A$ may not have exterior boundary vertices, implying that there $\partial_{\mathrm{int}}A$ is also empty. The next claim contains basic properties of blocks, their closures, and their boundaries, with proof in the Appendix on page \pageref{proof:containment}.

\begin{claim}[Blocks, their closures and their boundaries]\label{claim:containment}
 The following five statements hold:
 \begin{enumerate}
 \setlength\itemsep{0em}
  \item[(i)] For any block $B$, $\widetilde\partial_\mathrm{int}\bar B\subseteq\widetilde\partial_\mathrm{int} B$.
  \item[(ii)]
        Let $B_1, B_2$ be $1$-disconnected blocks such that $\bar B_1\cap\bar B_2\neq \emptyset$.
        Then either $\bar B_1\subseteq \bar B_2$ or $\bar B_2\subseteq \bar B_1$.
  \item[(iii)]
        Let $B_1, B_2$ be $1$-disconnected blocks such that $\bar B_1\cap\bar B_2 = \emptyset$. Then $\bar B_1, \bar B_2$ are also $1$-disconnected from each other. 
  \item[(iv)]For any block $B$, $\widetilde \partial_{\mathrm{int}}\bar B$ and  $\widetilde \partial_{\mathrm{ext}}\bar B$ are $*$-connected. 
  \item[(v)]For any hole $H$ of a block $B$, we have that $H=\bar H$, so $\widetilde\partial_{\mathrm{int}} H$ and $\widetilde\partial_{\mathrm{ext}} H$ are $\ast$-connected.
 \end{enumerate}
\end{claim}
The statements (ii)--(iii) say that if $B_1$ and $B_2$ are 1-disconnected, then either $B_1$ is inside a hole of $B_2$ or the other way round (ii), or their closures are also 1-disconnected (iii). Part (v) states that a hole $H$ of a 1-connected block $B$ cannot contain further holes. The fourth statement (iv) of the preceding claim is \cite[Lemma 2.1]{deuschel1996surface}.
The next claim shows that  the sizes of the boundaries with respect to $\Lambda_n$ and $\Z^d$ are of the same order, provided that the set has cardinality at most $3n/4$. Moreover, the claim contains an \emph{isoperimetric inequality} that we extensively use below. The proof is given in the appendix on page \pageref{proof:isoperimetry}.
\begin{claim}[Boundary bounds and isoperimetry]\label{claim:iso}
 There exists $\delta>0$ such that  for all $A\subseteq\Lambda_n$
 with $|A|\le 3n/4$ or $A\cap\widetilde\partial_\mathrm{int}\Lambda_n=\emptyset$,
 \begin{equation}
  | \partial_\mathrm{int} A|\ge \delta|\widetilde \partial_\mathrm{int}  A|\ {\buildrel (\star) \over \ge}\ \delta|A|^{(d-1)/d},
  \qquad
  |\partial_\mathrm{ext} A|\ge \delta|\widetilde \partial_\mathrm{ext}  A| \ {\buildrel (\star) \over \ge}\ \delta|A|^{(d-1)/d}. \label{eq:iso}
 \end{equation}
 The inequalities with $(\star)$ hold for any $A\subseteq \Lambda_n$ without  conditions on $A$.
\end{claim}

The next lemma is due to Peierls \cite{peierls1936ising}. We give a proof  in the appendix on page \pageref{proof:peierl}. More general results also appeared in \cite{babson1999cut, timar2007cutsets}. We will use this  claim to enumerate blocks that satisfy $A=\bar A$.
\begin{lemma}[Peierls' argument \cite{peierls1936ising}]\label{lemma:peierl}
 There exists a constant $c_\mathrm{pei}>0$ such that for all $x\in\Z^d$ and $m\in\N$,
 \begin{align}
  |\{A\in\CA:  A \ni x, A=\bar{A}, | \widetilde\partial_\mathrm{int}A|=m\}| & \le \exp(c_\mathrm{pei}m).\label{eq:peierl-2}
 \end{align}
\end{lemma}
The proof relies on the fact that $\widetilde\partial_\mathrm{int} A$ is $\ast$-connected \cite[Lemma 2.1]{deuschel1996surface} (which would not hold if $A$ contained holes, or may not hold if one  replaces $\widetilde\partial_\mathrm{int} A$ by $\partial_\mathrm{int} A$).

In \eqref{eq:iso}, we would like to replace $\partial_{\mathrm{int}}A$ by $\partial_\mathrm{int}\bar A$ from \eqref{eq:closure} (enumeration of sets that are equal to their closure would allow us to use Peierls' argument). However, then the isoperimetric inequality \eqref{eq:iso} may not hold anymore if the total boundary size of the holes in $A$ is too large compared to $\partial_\mathrm{int}\bar A$, which could happen if $|\bar A|>3n/4$.
We define two types of blocks, based on the size of the closures of the blocks (that is, whether Claim \ref{claim:iso} applies to $\bar A$ or not), i.e.,
\begin{equation}
 \begin{aligned}
  \CA_{\mathrm{small}} & :=\{A\in\CA:
  |\bar A|\le 3n/4,\ \mbox{ and }\ \bar{A}=A
  \},                                                                                       \\
  \CA_{\mathrm{large}} & := \big\{A\in\CA: |\bar A|>3n/4,\ \mbox{ and }\  |A|\le n/2\big\}.
 \end{aligned}
 \label{eq:a-sets}
\end{equation}
Recall the notation $A_1, \dots A_b\in_1 \CA$ in Definition \ref{def:blocks} meaning that the blocks $A_1, \dots, A_b$ are $1$-disconnected.  We extend this notation also to  subsets of $\CA$. For each of the sets and $k\in\N$ we define an event, namely
\begin{equation}\label{eq:def-event-1-spanning} 
 \CE_1(b, \CG_n)  := \left\{\exists\,(A_i)_{i\le b}\in_1\CA_{\mathrm{small}} \,\, \middle|\,\,
\begin{aligned}
  &(\cup_i \widetilde\partial_\mathrm{int}A_i)  \not\sim_{\CG_n} (\Lambda_n\!\setminus\! \cup_i A_i),\\
  &|\!\cup_i\!A_i|\ge k, \\
                                               &\CH_{\CG_n}((A_i)_{i\le b})\mbox{ connected}                       %\\
\end{aligned}
 \right\},                                      
\end{equation}
and
\begin{align}
 \CE_2(\CG_n)    & :=
 \{\exists\, A\in\CA_{\mathrm{large}}: A\not\sim_{\CG_n} \partial_\mathrm{ext}A\}.
 \label{eq:def-event-2-holes}
\end{align}
The following deterministic claim holds for any graph on vertices in $\Lambda_n$. It shows that the union of these events contains the event $\{|\CC_n^\sss{(2)} |\ge k\}$. In particular, the proof reveals why we could restrict to sets with $A=\bar A$ in the definition of $\CA_{\mathrm{small}}$ in \eqref{eq:a-sets}.
\begin{claim}\label{claim:second-events}
 Consider the graph $\CG_n$ from Definition \ref{def:lrp}, with $\CC_n^\sss{(2)}$ the second-largest component of $\CG_n$. Then
 \begin{equation}
  \big\{|\CC_n^\sss{(2)}|\ge k\big\} \subseteq \CE_2(\CG_n) \cup \Big(\bigcup_{b=1}^{\lfloor n/2\rfloor } \CE_1(b, \CG_n)\Big).\nonumber%
 \end{equation}
\end{claim}

\begin{proof}[Proof of Claim \ref{claim:second-events}]
 Clearly $\big\{|\CC_n^\sss{(2)}|\ge k\big\}\subseteq \{\exists \mbox{ a component $\CC$ of } \CG_n: |\CC| \in [k, \lfloor n/2\rfloor]  \}$.  The size restriction $n/2$ is possible since otherwise $\CC_n^{\sss{(2)}}$ would be the largest component. Take any such $\CC$.
 We use Claim \ref{claim:unique}  to first uniquely decompose $\CC$ into $1$-disconnected (hence disjoint) blocks $A_1, \dots, A_{b}$ for some $b\ge 1$.
 Since $k\le |\CC|\le n/2$, $|A_i|\le n/2$ also holds for all $i\le b$, and also $b\le n/2$, and $\sum_{i\le b} |A_i|\ge k$.

 We distinguish two cases. Either (1) there is at least one block that is in $\CA_{\mathrm{large}}$ or (2) all the blocks are in $\CA\setminus \CA_{\mathrm{large}}$. In the first case
 the event $\CE_2$ in \eqref{eq:def-event-2-holes} holds:  a block of $\CC$ satisfying $\CA_{\mathrm{large}}$ (say $A$) is by definition $1$-disconnected from the other blocks of $\CC$, and since $\CC$ is a component of $\CG_n$, $A$ is $\CG_n$-disconnected from its exterior boundary, hence $\CE_2(\CG_n)$ defined in \eqref{eq:def-event-2-holes} holds.

 In case (2) all the blocks $(A_i)_{i\le b}$ are in $\CA\setminus \CA_{\mathrm{large}}$, and, since they form the component $\CC$, the graph $\CG_n$ spanned on $\cup_{i\le b}A_i$ is  connected, while their union is disconnected in $\CG_n$ from the rest of the graph. Finally their disjointness and $|\CC|\in[k, n/2]$ ensure that $|\!\cup_{i\le b}\! A_i|\in[k,n/2]$. Formally we describe this event as (changing the sets to be denoted by $B_i$ to avoid clash of notation later):
 \begin{equation*}
  \widetilde{\CE}_1(b, \CG_n) \!:=\! \left\{\exists (B_i)_{i\le b}\in_1\CA\!\setminus\! \CA_{\mathrm{large}}\,\,\middle|\,\, 
   \begin{aligned}
   &\CG_n[\cup_{i\le b} B_i]\mbox{ connected},\\
   &|\!\cup_{i\le b}\! B_i|\in[k,n/2], \\                                                                           
                        &(\cup_{i\le b}B_i)\not\sim_{\CG_n} (\Lambda_n\!\setminus\! \cup_{i\le b} B_i)
                        \end{aligned}
  \right\}.\nonumber%\label{eq:def-event-1-spanning-adj} 
 \end{equation*}
 Taking a union over the number of blocks and combining the two cases, we arrive at
 \[ \big\{|\CC_n^\sss{(2)}|\ge k\big\}\subseteq \CE_2(\CG_n) \cup \big(\cup_{b=1}^{\lfloor n/2\rfloor } \widetilde \CE_1(b, \CG_n)\big). \]
Take now any graph $G_n$ on the vertex set $\Z\cap \Lambda_n$. Think of this as the realization of $\CG_n$. We will show the implication that
  \begin{equation}\label{eq:tilde-vs-nontilde}
 \Big\{ \bigcup_{b=1}^{\lfloor n/2 \rfloor} \widetilde\CE_1(b, G_n) \text{ holds for $G_n$} \Big\} \Longrightarrow \Big\{  \bigcup_{b'=1}^{\lfloor n/2 \rfloor} \CE_1(b', G_n) \text{ holds for $G_n$} \Big\}.
 \end{equation}
Given $G_n$ for which $\bigcup_{b=1}^{\lfloor n/2 \rfloor} \widetilde\CE_1(b, G_n)$ holds, take now $b$ and the $1$-disconnected blocks $(B_1, \dots, B_b)$ for which $\widetilde \CE_1(b, G_n)$ holds in the union.
 The conditions of Claim \ref{claim:containment} are satisfied for any pair $B_i, B_j$ with $i\neq j$, hence  for each pair, $\bar B_i$ and $ \bar B_j$ are either $1$-disconnected disjoint sets, or one contains fully the other one.  Choose now those sets in $\{\bar B_1,\ldots, \bar B_b\}$ that are not contained in any other set in the same list.
 We then obtain an integer $b'\le b$ and a $1$-disconnected subset $\{\bar B_{i_1},\ldots, \bar B_{i_{b'}}\}\subseteq\{\bar B_1,\ldots, \bar B_b\}$ such that
 \begin{equation}
  \bigcup_{i=1}^bB_i\, \subseteq\, \bigcup_{i=1}^b\bar B_i\,= \,\bigcup_{j=1}^{b'}\bar B_{i_j}.\label{eq:subset-blocks}
 \end{equation}
 Since  $(\bar B_1,\ldots, \bar B_{b'})$ is a $1$-disconnected sequence of blocks that are equal to their own closure, \eqref{eq:tilde-vs-nontilde} follows  if $\CE_1(b', G_n)$ holds. For this, we use that $B_1, \dots B_b$ satisfies $\widetilde \CE_1(b,G_n)$ by assumption, so when checking \eqref{eq:def-event-1-spanning} for $(\bar B_1,\ldots, \bar B_{b'})$, it is sufficient to  show that
 \begin{align}
  \{  G_n[\cup_{i\le b}B_i]\mbox{ connected}\}
   & \subseteq
  \{
  \CH_{G_n}(( \bar B_i)_{i\le b'})\mbox{ connected}
  \}, \label{eq:inc-conn} \\
  \{
  |\cup_{i\le b}B_i|\in[k, n/2]
  \}
   & \subseteq
  \{
  |\cup_{i\le b}\bar B_i|\ge k
  \}, \label{eq:inc-size} \\
  \{(\cup_{i\le b}B_i)\not\sim_{G_n}(\Lambda_n\setminus\cup_{i\le b}\bar B_i) \}
   & \subseteq
  \{(\cup_{i\le b'}\widetilde\partial_\mathrm{int}\bar B_i)\not\sim_{G_n}(\Lambda_n\setminus\cup_{i\le b'}\bar B_i) \},\label{eq:inc-disconn}
 \end{align}
 since then all conditions for $\CE_1(b', G_n)$ defined in \eqref{eq:def-event-1-spanning} are satisfied by setting $A_i=\bar B_i$ for all $i\le b'$.
For the first inclusion \eqref{eq:inc-conn} we observe that
 \begin{equation}
   \big\{G_n[\cup_{i\le b} B_i]  \mbox{ connected}\big\} \subseteq \big\{\CH_{G_n}((B_i)_{i\le b})\mbox{ connected}\big\}
   \subseteq\big\{\CH_{G_n}((\bar B_i)_{i\le b'})\mbox{ connected}\big\},
\nonumber%
 \end{equation}
 since the left-hand side was assumed in $\widetilde \CE_1(b, G_n)$, and the block graph being connected is a less demanding event than the actual spanned graph being connected,  and the second containment follows since each set of edges in $G_n$ that ensures that $\CH_{\CG_n}((B_i)_{i\le b})$ is connected, also ensures that the block graph on the closures of $(B_i)_{i\le b}$ is connected.

 The second inclusion \eqref{eq:inc-size} follows by using $\cup_{i\le b} B_i\subseteq \cup_{i\le b'} \bar B_i$ in \eqref{eq:subset-blocks}. For the third inclusion \eqref{eq:inc-disconn} we have to argue that the set of edges that is excluded on the right-hand side is contained in the set of excluded edges on the left-hand side. Clearly $(\Lambda_n\setminus\cup_{i\le b}B_i)\supseteq (\Lambda_n\setminus\cup_{i\le b'}\bar B_i)$, and by part (i) in Claim \ref{claim:containment} it follows that $\partial_\mathrm{int}\bar B_i\subseteq B_i$.
\end{proof}

We state two lemmas that together with Claim \ref{claim:second-events} prove Proposition \ref{prop:second-largest}. We prove the two lemmas in the following sections.
\begin{lemma}[Unlikely block graphs]\label{lemma:spanning-tree} Let $\CG_n$ be long-range percolation on $\Lambda_n$ as in Definition \ref{def:lrp} with $d\ge 2$, $\alpha> 1+1/d$. 
There exists a constant $c_{\ref{lemma:spanning-tree}}=c_{\ref{lemma:spanning-tree}}(d,\alpha)>0$ such that for all $k, n$ sufficiently large
   \begin{equation}\label{eq:unlikely-block-graphs}
  \Prob\Big(\bigcup_{b=1}^{\lfloor n/2\rfloor } \CE_1(b, \CG_n)\Big) 
  \le
      (n\log n)\cdot\Big(p\beta^{d\alpha}\cdot\big(1-p(1\wedge\beta)^{d\alpha}\big)^{(\beta\vee1) c_{\ref{lemma:spanning-tree}}}\Big)^{k^{(d-1)/d}}
      \end{equation}
 whenever $p$ and $\beta$ guarantee that the base of $k^{(d-1)/d}$ is sufficiently small.
\end{lemma}
We give a brief intuition on the powers occurring on the right-hand side of \eqref{eq:unlikely-block-graphs}. When $\beta\ge 1$ in \eqref{eq:unlikely-block-graphs}, the factor $(1-p)^\beta$ comes from the fact that any vertex within $\beta$ distance from the boundary of total size at least $k^{(d-1)/d}$ of the blocks $\cup_{i\le b}A_i$ in \eqref{eq:def-event-1-spanning} must be $\CG_n$-disconnected from $\cup_{i\le b}A_i$. An edge between two vertices within distance $\beta$ is absent with probability $1-p$ by \eqref{eq:connection-prob-gen}. In this case $p\beta^{d\alpha }\ge 1$ is possible; this factor arises from counting the number of possible spanning trees on the blocks. When $\beta<1$, we only exclude edges of length $1$ around the boundary of $\cup_{i\le b}A_i$, and the factor $p\beta^{d\alpha}$ is at most $1$.

We proceed to the lemma dealing with $\CE_2(\CG_n)$.
Define 
\begin{equation}\label{eq:fpbeta}
    f(p,\beta):=
    \begin{dcases}
        1-p(1\wedge\beta)^{d\alpha},&\text{if }\beta<2\sqrt{d},\\
        (1-p)^{(\log_2\beta)^{-2}\beta^{(d-2)/(d-1)}},&\text{if }\beta\ge 2\sqrt{d}.
    \end{dcases}
\end{equation}

The function $f(p,\beta)$ describes how large $p$ and $\beta$ should be in the next lemma. It is a technical artifact of the proof that we did not optimize. We write $x\vee y=\max(x,y)$.
\begin{lemma}[No large isolated component]\label{lemma:holes}
 Let $\CG_n$ be long-range percolation on $\Lambda_n$ as in Definition \ref{def:lrp} with $d\ge 2$, $\alpha> 1+1/d$.  There exists a constant $c_{\ref{lemma:holes}}=c_{\ref{lemma:holes}}(d)>0$ such that for all  $n$ sufficiently large
 \begin{equation}\label{eq:no-large-holes}
  \Prob\big(\CE_2(\CG_n)\big)\le \big(1-p(1\wedge\beta)^{d\alpha}\big)^{(\beta\vee 1)\cdot c_{\ref{lemma:holes}} n^{(d-1)/d}(\log n)^{-2}}
 \end{equation}
 whenever $f(p,\beta)$ is sufficiently small.
\end{lemma}
  To prove \eqref{eq:no-large-holes}, we use that any set $A$ with size below $n/2$ but closure $\bar A$ with size above $3n/4$ must have boundary at least $\Theta(n^{(d-1)/d})$, and large holes to which it cannot be connected by an edge in $\CG_n$. The $(\log n)^{-2}$ factor is a technical artifact of our proof based on the pigeon-hole principle. In both Lemma~\ref{lemma:spanning-tree} and~\ref{lemma:holes}, the challenge is that a simple union bound would yield too large combinatorial factors arising from the possibilities for the sets $(A_i)_{i\le b}$ and $A$, respectively. See also Section \ref{sec:sketch}.

\section{Spanning trees on block graphs}\label{sec:spanning}
We work towards proving Lemma \ref{lemma:spanning-tree}. Recall the event $\CE_1(b):=\CE_1(b, \CG_n)$ from \eqref{eq:def-event-1-spanning}, and $\CA_{\mathrm{small}}$ for the blocks in $\Lambda_n$ from \eqref{eq:a-sets} on which there should be a connected block graph $\CH_{\CG_n}((A_i)_{i\le b})$ (see Definition \ref{def:blockgraph}).
By a union bound over all possible $1$-disconnected sequences of blocks $(A_1,\ldots, A_b)\in_1\CA_\mathrm{small}$ whose total size is at least $k$, we obtain
\begin{equation*}
 \Prob\big(\CE_1(b)\big)
 \le
 \frac{1}{b!}\sum_{(A_i)_{i\le b}\in_1\CA_\mathrm{small}} \hspace{-15pt}\Prob\Big(
 \CH_{\CG_n}((A_i)_{i\le b})\mbox{ is connected},\ (\cup_{i\le b} \widetilde\partial_\mathrm{int}A_i)\not\sim_{\CG_n} (\Lambda_n\!\setminus\! \cup_{i\le b} A_i)
 \Big),
\end{equation*}
where the factor $1/b!$ corrects for the permutations of $(A_1,\ldots, A_b)$ yielding the same blocks, but ordered differently.
Using the independence of edges in long-range percolation in Definition \ref{def:lrp}, we obtain
\begin{equation}\label{eq:explain-factorial-1}
 \Prob\big(\CE_1(b)\big)
 \le
 \frac{1}{b!}\sum_{(A_i)_{i\le b}\in_1\CA_\mathrm{small}} \hspace{-15pt} \Prob\big(
 \CH_{\CG_n}((A_i)_{i\le b})\mbox{ is connected}\big)\cdot\Prob\big((\cup_{i\le b} \widetilde\partial_\mathrm{int}A_i)\not\sim_{\CG_n} (\Lambda_n\!\setminus\! \cup_{i\le b} A_i)
 \big).
\end{equation}
The block graph $\CH_{\CG_n}((A_i)_{i\le b})$ can only be connected if it contains a spanning tree on its blocks. To count these spanning trees, we introduce the \emph{rooted labeled $\mathbf{f}$-tree}. In the following definition, we use that each tree on $b$ vertices has $b-1$ edges.

\begin{definition}[$\mathbf{f}$-tree]\label{def:f-tree}
 Let $\CF_b$ be the set of vectors $\mathbf{f}=(f_1,\dots, f_{b})\in\N_0^{b}$ satisfying $\sum_{i\in[b]}f_i=b-1$, $f_1+\ldots+f_j\ge j$ for all $j\in[b-1]$. We call $\mathbf{f}$ the vector of \emph{forward degrees}.  A rooted labeled tree on $b$ vertices is an $\mathbf{f}$-tree if the root has label $1$, and it has an outgoing edge to each of the vertices with labels $2,\dots,f_1+1$, vertex $2$ has an outgoing edge to each of the vertices with labels  $f_1+2,\dots,f_1+f_2+1$, and so on, the vertex with label $j$ has an outgoing edge to each of the vertices with labels $2+\sum_{i=1}^{j-1}f_i, \dots, 1+ \sum_{i=1}^{j}f_i $. If $(i,j)$ is a directed edge in an $\mathbf{f}$-tree, then we say that $i$ is the parent of $j$ and $j$ is the child of $i$.
 We say that the labeled block graph $\CH_{\CG_n}((A_i)_{i\le b})$ is \emph{$\mathbf{f}$-connected}, if it contains an $\mathbf{f}$-tree on its vertices $(1,2,\dots, b)$.
\end{definition}
Given a  forward-degree vector $\mathbf{f}$ and a labeled set of vertices, the $\mathbf{f}$-tree is uniquely determined.
The construction ensures that the block with label $b$ must be a leaf, i.e., it has forward degree $f_b=0$, and its parent corresponds to the label of the last nonzero entry of $\mathbf{f}$. Further, given a tree $T$ with labeling $\mathbf f \in \CF_b$, upon removing the leaf with label $b$, we obtain a tree $T\setminus \{b\}$  on $\{1, \dots, b-1\}$ with a labeling in $\CF_{b-1}$.

Further, an $\mathbf f$-tree always has vertex $1$ as its root, and the forward neighbors of any vertex have consecutive labels. Hence, not all the $b!$ labelings of a tree $T$ are valid labelings, i.e., no vector $\mathbf f\in \CF_b$ can be associated to some labelings.
However, for a fixed tree $T$ on a connected block graph $\CH_{\CG_n}((A_i)_{i\le b})$, there is at least one permutation $\sigma$ of $(1,2,\ldots, b)$ with $\sigma(1)=1$ and a vector $\mathbf{f}\in\CF_b$ such that $\CH_{\CG_n}((A_{\sigma(i)})_{i\le b})$ is $\mathbf{f}$-connected. In other words, we can relabel the blocks so that the new labeling $(1, \sigma(2), \dots \sigma(b))$ is a proper labeling of $T$, for some $\mathbf f\in \CF_b$ in Definition \ref{def:f-tree}.  We denote the set of permutations of $(1,2,\dots, b)$ with $1$ a fixed point by $\CS_b^1$. Note that the choice of the spanning tree $T$ may not be unique if $\CH_{\CG_n}((A_i)_i)$ is connected. We obtain on the first factor inside the sum in  \eqref{eq:explain-factorial-1} that
\begin{equation}\label{eq:explain-factorial-2}
  \Prob\big(\CH_{\CG_n}((A_i)_{i\le b})\mbox{ connected}\big)  =
    \Prob\Big(\bigcup_{\mathbf{f} \in \CF_b}\bigcup_{\sigma \in \CS_b^1}   \big\{\CH_{\CG_n}((A_{\sigma(i)})_{i\le b})\mbox{ is $\mathbf{f}$-connected}\big\}\Big).
\end{equation}
If, for a given $(\mathbf{f}, \sigma)$ the block graph is $\CH_{\CG_n}((A_{\sigma(i)})_{i\le b})$ is $\mathbf{f}$-connected, then there are $\prod_i f_i!$ other pairs $(\mathbf{f}', \sigma')$ such that $\CH_{\CG_n}((A_{\sigma'(i)})_{i\le b})$ is $\mathbf{f}'$-connected, counting the isomorphisms of rooted trees: namely, the (consecutive) labels of the forward neighbors of any vertex $v$ may be permuted (yielding the factor $f_v!$ for each vertex), resulting in permuting the labels in the forward-subtrees of $v$ accordingly.  For any such $(\mathbf{f}, \sigma)$ and $(\mathbf{f}', \sigma')$, we then also have that
$\prod_{i=1}^{b}f_i!=\prod_{i=1}^{b}f_i'!$.  Hence, in the above union each rooted tree $T$ (with root fixed) on $\CH_{\CG_n}((A_{\sigma'(i)})_{i\le b})$ is counted $\prod_{i=1}f_i!$ times. Thus, we obtain from \eqref{eq:explain-factorial-2} that
\begin{equation}\nonumber%\label{eq:explain-factorial-3}
 \Prob\big(\CH_{\CG_n}((A_i)_{i\le b})\mbox{ connected}\big) \le \sum_{\mathbf{f}\in\CF_b}\bigg(\prod_{i=1}^b\frac{1}{f_i!}\bigg)\sum_{\sigma\in \CS_b^1}\Prob\big( \CH_{\CG_n}((A_{\sigma(i)})_{i\le b})\mbox{ is $\mathbf{f}$-connected}\big).%
\end{equation}
We substitute this into \eqref{eq:explain-factorial-1}, and use that the product $\prod_{i=1}^b\frac{1}{f_i!}$ is invariant under label permutations. So we arrive at
\begin{equation}
 \begin{aligned}
  \Prob\big(\CE_1(b)\big)
  \le
  \frac{1}b
  \sum_{\mathbf{f}\in\CF_b}\hspace{-2pt}\bigg(\prod_{i=1}^b\frac{1}{f_i!}\bigg)\sum_{A_1}\frac{1}{(b-1)!}\sum_{\sigma\in \CS_b^1}&\sum_{(A_i)_{2\le i\le b}\in_1\CA_\mathrm{small}}\hspace{-24pt}\Prob\big(\CH_{\CG_n}((A_{\sigma(i)})_{i\le b})\mbox{ is $\mathbf{f}$-connected}\big) \\
  &\quad\cdot                                                                  
                                \Prob\big((\cup_{i} \widetilde\partial_\mathrm{int}A_{\sigma(i)})\not\sim_{\CG_n} (\Lambda_n\!\setminus\! \cup_{i} A_{\sigma(i)})
  \big).
 \end{aligned}\nonumber%
\end{equation}
We now argue that the sum over the permutations and the factor $1/(b - 1)!$ cancel each other.
Given $A_1$, let $(B_2,\ldots, B_b)\in_1 \CA_{\mathrm{small}}$ be arbitrary blocks of total size at least $k-|A_1|$, also $1$-disconnected from $A_1$. Then, for any permutation $\sigma\in \CS_b^1$, in the summations over the blocks $A_2,\ldots,A_b$, there is precisely one combination of blocks such that $A_{\sigma(i)}=B_i$ for all $i\le b$.
Hence, when summing over all permutations $\sigma\in \CS_b^1$, we counted the case that the blocks are $A_1, B_2, \dots, B_b$ exactly $(b-1)!$ times.
This cancels the factor $1/(b-1)!$, and we arrive at
\begin{equation}
 \begin{aligned}
  \Prob\big(\CE_1(b)\big)
  \le
  \frac{1}b
  \sum_{\mathbf{f}\in\CF_b}\bigg(\prod_{i=1}^b\frac{1}{f_i!}\bigg)\sum_{(A_i)_{i\le b}\in_1\CA_\mathrm{small}}\Prob & \big( \CH_{\CG_n}((A_{i})_{i\le b})\mbox{ $\mathbf{f}$-connected}\big) \\  & \cdot\Prob\big((\cup_{i\le b} \widetilde\partial_\mathrm{int}A_i)\not\sim_{\CG_n} (\Lambda_n\!\setminus\! \cup_{i\le b} A_i)
  \big).
 \end{aligned}\nonumber%
\end{equation}
Lastly, we prescribe the sizes of the boundaries of the blocks.
We introduce the possible boundary-length vectors:
\begin{equation}\label{eq:boundary-vectors}
 \CM_b(k):=\left\{ \begin{aligned}\mathbf{m}\!=\!(m_1,\ldots, m_b)\in\N^b:  \exists (A_i)_{i\le b}\!\in_1\!\CA_\mathrm{small}: \ &|\widetilde\partial_\mathrm{int}A_i|=m_i\  \forall i\le b, \\ &|\cup_{i\le b}A_i|\ge k
 \end{aligned}\right\}.
\end{equation}
Then,
\begin{equation}
 \begin{aligned}
  \Prob\big(\CE_1(b)\big)\le \frac{1}{b} \sum_{\mathbf{m}\in\CM_b(k)}\sum_{\mathbf{f}\in\CF_b}\bigg(\prod_{i\in[b]}\frac{1}{f_i!}\bigg)&\sum_{\substack{(A_i)_{i\le b}\in_1\CA_\mathrm{small},\\ |\widetilde\partial_\mathrm{int}A_i|=m_i\, \forall i\le b}} \Prob \big(\CH_{\CG_n}((A_i)_i)\mbox{ $\mathbf{f}$-conn.}\big)                                                                                       \\
                                                                                                                                                                                                                            & \qquad \qquad\qquad\cdot \Prob\big(\cup_i \partial_\mathrm{int}A_i\not\sim_{\CG_n} \Lambda_n\setminus \cup_i A_i\big).\label{eq:spanning-before-stats}
 \end{aligned}
\end{equation}
The next two statements will imply Lemma \ref{lemma:spanning-tree}.
\vskip1em

\begin{statement}[Counting spanning trees]\label{stat:spanning} Let $\CG_n$ be long-range percolation on $\Lambda_n$ as in Definition \ref{def:lrp} with $d\ge 2$, $\alpha> 1+1/d$.
 There exists $c_{\ref{stat:spanning}}=c_{\ref{stat:spanning}}(d,\alpha)>0$ such that for all fixed $\mathbf{m}\in\CM_b(1)$,
 \begin{equation}\label{eq:stat-spanning}
  \sum_{\mathbf{f}\in\CF_b}
  \bigg(\prod_{i\in[b]}\frac1{f_i!}\bigg)\sum_{\substack{(A_i)_{i\le b}\in_1\CA_\mathrm{small},\\ |\widetilde\partial_\mathrm{int}A_i|=m_i\, \forall i\le b}}\hspace{-12pt}\Prob\big(\CH_{\CG_n}((A_i)_i)\mbox{ is $\mathbf{f}$-connected}\big) \le n\big(C_{\ref{stat:spanning}}\cdot p\beta^{d\alpha}\big)^{\sum_{i\in[b]}m_i}.
 \end{equation}
\end{statement}
 The bound given by the previous statement increases exponentially in $\sum_{i \in [b]}m_i$. While this might look harmful, the next statement will compensate for this, so that the two bounds combined, for appropriate choices of the constants $p$ and $\beta$, still gives a bound decaying exponentially in $\sum_{i \in [b]}m_i$.
\vskip1em
\begin{statement}[Isolation]\label{stat:disconn} Let $\CG_n$ be long-range percolation on $\Lambda_n$ as in Definition \ref{def:lrp} with $d\ge 2$, $\alpha> 1+1/d$. 
 There exists $c_{\ref{stat:disconn}}=c_{\ref{stat:disconn}}(d)>0$ such that for any $\mathbf{m}\in\CM_b(1)$, and any $1$-disconnected blocks $(A_i)_{i\le b}\in_1\CA_\mathrm{small}$ with $|\widetilde\partial_\mathrm{int}A_i|=m_i$ for all $i$,
 \begin{equation}\label{eq:stat-isolation}
  \Prob\Big((\cup_i \widetilde\partial_\mathrm{int}A_i)\not\sim_{\CG_n} (\Lambda_n\!\setminus\! \cup_i A_i)\Big) \le 
     \big(1-p(1\wedge\beta)^{d\alpha}\big)^{(\beta\vee1)c_{\ref{stat:disconn}}\cdot\sum_{i\in[b]}m_i}.
 \end{equation}
\end{statement}
We show first that Lemma \ref{lemma:spanning-tree} follows from these statements, and then prove the statements in the remainder of the section.
\begin{proof}[Proof of Lemma \ref{lemma:spanning-tree}, assuming Statements \ref{stat:spanning} and \ref{stat:disconn}]
 Substituting the bounds from Statements \ref{stat:spanning} and \ref{stat:disconn} into the right-hand side of \eqref{eq:spanning-before-stats} yields for $\beta\ge 1$ that
  \begin{equation}\label{eq:vector-m-1}
  \Prob\big(\CE_1(b)\big)\le \frac{1}{b}
  \sum_{\mathbf{m}\in\CM_b(k)}n\big(C_{\ref{stat:spanning}}p\beta^{d\alpha}\cdot(1-p)^{\beta\cdot c_{\ref{stat:disconn}}}\big)^{\sum_{i\in[b]}m_i}.
 \end{equation}
 In what follows we evaluate the summation over the vectors $\mathbf m\in \CM_b(k)$.
 We recall from \eqref{eq:boundary-vectors} that $\mathbf{m}$  represents the vector of interior boundary sizes of $1$-connected sets $(A_i)_{i\le b} \in_1 \CA_{\mathrm{small}}$ with total size at least $k$, and $A_i=\bar A_i\le 3n/4$  for all $i\le b$ by the definition of $\CA_{\mathrm{small}}$ in \eqref{eq:a-sets}. In \eqref{eq:boundary-vectors}, the boundary is taken with respect to $\Z^d$, i.e., \emph{not} with respect to $\Lambda_n$.
 By the isoperimetric inequality in Claim \ref{claim:iso},  for all blocks $(A_i)_{i\le b}$ it simultaneously holds that $|\widetilde\partial_\mathrm{int} A_i|\ge|A_i|^{(d-1)/d}$.
 Since the function $g(k)= k^{(d-1)/d}$ is concave and increasing, we obtain for all $\mathbf{m}\in\CM_b(k)$ and any $(A_i)_{i\le b}\in_1\CA_{\mathrm{small}}$ satisfying $\widetilde\partial_\mathrm{int}A_i=m_i$ for all $i\le b$ that
 \begin{equation}
  m_1+\ldots+m_b= |\widetilde\partial_\mathrm{int} A_1|+\ldots+|\widetilde\partial_\mathrm{int} A_b| \ge
  |A_1|^{(d-1)/d}+\ldots+|A_b|^{(d-1)/d}\ge
  k^{(d-1)/d}.\nonumber%
 \end{equation}
 We define the set $\CM_b(k, \ell):=\{\mathbf{m}\in\CM_b(k): m_1+\ldots+m_b=\ell\}$. By standard estimates (using that each summand is at least one), we bound
 $
  |\CM_b(k, \ell)|\le \binom{\ell+b}{b}\le\binom{2\ell}\ell\le 2^{2\ell}\le \re^{2\ell}.
 $ Hence, separating the summation in \eqref{eq:vector-m-1} according to the possible values of $\sum_{i}m_i=\ell\ge k^{(d-1)/d}$, we arrive at
  \begin{equation}
  \Prob\big(\CE_1(b)\big)\le \frac{n}b\sum_{\ell=k^{(d-1)/d}}^\infty \big(2\cdot C_{\ref{stat:spanning}}\cdot p\beta^{d\alpha}(1-p)^{\beta\cdot c_{\ref{stat:disconn}}}\big)^{\ell}.\label{eq:after-comb-trick}
 \end{equation}
 Since $b\le\lfloor n/2\rfloor$, we obtain by a union bound over the number of blocks that 
  \begin{equation}\label{eq:unlikely-proof-last}
  \Prob\Big(\bigcup_{b\le \lfloor n/2 \rfloor}\!\!\CE_1(b)\Big)
  \le
  \sum_{b=1}^{\lfloor n/2\rfloor}\frac{n}b\sum_{\ell=k^{(d-1)/d}}^\infty \big(2 C_{\ref{stat:spanning}}\cdot p\beta^{d\alpha}(1-p)^{\beta\cdot c_{\ref{stat:disconn}}}\big)^{\ell}.
 \end{equation}
 The two sums are independent of each other.  
 We bound the first sum from above by $n\log n$.
 We obtain
 \begin{equation}\nonumber
  \Prob\Big(\bigcup_{b\le \lfloor n/2 \rfloor}\!\!\CE_1(b)\Big)
  \le
 (n\log n)\sum_{\ell=k^{(d-1)/d}}^\infty \big(2 C_{\ref{stat:spanning}}\cdot p\beta^{d\alpha}(1-p)^{\beta\cdot c_{\ref{stat:disconn}}}\big)^{\ell}.
 \end{equation}The two constants $C_{\ref{stat:spanning}}$ and $c_{\ref{stat:disconn}}$ depend only on $d$ and $\alpha$. 
We may assume that $p$ and $\beta$ are such that $p\beta^{d\alpha}(1-p)^{\beta\cdot c_{\ref{stat:disconn}}}$ is smaller than $(2 C_{\ref{stat:spanning}})^{-2}$, and under this assumption, we obtain Lemma~\ref{lemma:spanning-tree} when $\beta\ge 1$. When $\beta<1$, we replace $(1-p)^\beta$ with $(1-p\beta^{d\alpha})$ in  \eqref{eq:vector-m-1} and conclude in the same way.
\end{proof}

\subsection{Proof of Statement \ref{stat:spanning}}
We start with a geometric claim.

\begin{claim}\label{claim:cross}
 There exists a constant $C_{\ref{claim:cross}}>0$ such that for each block $A\in\CA_\mathrm{small}$ in $\Lambda_n$ and all $r\in\N$,
 \begin{equation}\label{eq:cross-boundary}
  \big|\big\{(x,y)\in \Z^d\times \Z^d: x\in A, y\notin A, \|x-y\|\in (r,r+1]\big\}\big|\le C_{\ref{claim:cross}} r^d |\widetilde \partial_{\mathrm{int}} A|.
 \end{equation}
\end{claim}
\begin{proof}
 We start counting line-segments of the right length with endpoints in $\Z^d$ crossing a single unit square that will be centered later at some vertex in $\widetilde \partial_{\mathrm{int}} A$.

 Let $\mathrm B_{0}:=[-1/2,1/2]^d$. For two vertices $x,y\in \Z^d$, let $\CL_{x,y}$ denote the segment between $x$ and $y$ on the unique line connecting $x$ and $y$.
 Define
 \[ \mathrm{Cross}(r):=\{(x,y)\in \Z^d\times \Z^d: \|x-y\|_2 \in (r, r+1], \CL_{x,y}\cap \mathrm{B}_{0}\neq \emptyset\}.\]
 We will show that  there exists a constant $C_{\ref{claim:cross}}>0$ such that
 \begin{equation}\label{eq:cross}
  |\mathrm{Cross}(r)|\le C_{\ref{claim:cross}}r^d. \end{equation}
 Indeed, for each pair $(x,y) \in \mathrm{Cross}(r)$, at least one of the inequalities $\|x\|\ge r/2$ and $\|y\|\ge r/2$ is satisfied. Without loss of generality we may assume that $\|x\|\ge r/2$, and then also $\|x\|\le r+1$.
 Fix then such a vertex $x$.
 Let $\CS_{x}(r_0)$ denote the smallest spherical cone with apex at $x$ that completely contains $\mathrm{B}_{0}$, with $r_0$ being the radius of this cone. Let $\CS_{x}(r)$ denote a cone with apex $x$ that has the same boundary lines (and the same angle) as $\CS_{x}(r_0)$, but radius exactly $r$. Then, since $\|x\|\in [ r/2, r+1]$ by assumption,  $r_0 \in [ \|x\|+1/2, \|x\|+\sqrt{d}/2]$.
 Further, every $y\in \Z^d$ such that $(x,y)\in\mathrm{Cross}(r)$ must be  contained in $\CS_x(r+1)\setminus \CS_{x}(r)$, since all half-lines emanating from $x$ that cross $\mathrm{B_0}$ are contained in $\CS_x(\infty)$, and $\|x-y\|\in (r, r+1]$. Since the radius of $\CS_x(r+1)$ is at most by a factor two larger than the radius $r_0$ of $\CS_x=\CS_x(r_0)$ for all $r\ge 1$, by homothety of the cones,  $|(\CS_{x}(r+1)\setminus \CS_{x}(r))\cap \Z^d|$ is bounded from above by a dimension-dependent constant, and so for each $x$ with $\|x\|\in[r/2,r+1]$, the number of pairs $(x,y)\in \mathrm{Cross}(r)$ is bounded from above by a dimension-dependent constant.
 Summing over all the at most $O(r^d)$ many such $x$, we obtain \eqref{eq:cross} for some $C_{\ref{claim:cross}}>0$.

 To arrive to \eqref{eq:cross-boundary}, the block $A\in \CA_{\mathrm{small}}$ in \eqref{eq:a-sets} ensures that $A=\bar A$. Its interior boundary  $\widetilde\partial_\mathrm{int}A$ is then $\ast$-connected by Claim \ref{claim:containment} Part~(iv) (\cite[Lemma 2.1]{deuschel1996surface}). This implies that there exists a $\ast$-connected surface fully contained in $\widetilde\partial_\mathrm{int}A$ separating vertices in $A\setminus \widetilde\partial_\mathrm{int}A$ from vertices in $\Z^d\setminus A$. Hence, for each pair $x\in A$ and $y\notin A$,  there exists at least one vertex $z\in \widetilde\partial_\mathrm{int}A$ such that $\CL_{x,y}$ intersects the axis-parallel box $z+\mathrm{B}_0$. Here $x=z$ may occur. The statement of the claim now follows by \eqref{eq:cross} when summing the at most $C_{\ref{claim:cross}}r^d$ such pairs for each vertex $z$ on the boundary of $A$.
\end{proof}
We continue with a lemma treating the connectedness of the block graphs, i.e., the inner summation on the left-hand side of \eqref{eq:stat-spanning} in Statement \ref{stat:spanning}. We point out that in this lemma we only bound the event that the block graph $\CH_{\CG_n}$ is connected, not the event that the actual graph is connected. 

\begin{lemma}\label{claim:spanning} Let $\CG_n$ be long-range percolation on $\Lambda_n$ as in Definition \ref{def:lrp} with $d\ge 2$, $\alpha> 1+1/d$.
 There exists a constant $C_{\ref{claim:spanning}}=C_{\ref{claim:spanning}}(d, \alpha)>0$ such that for all $\mathbf{m}\in\CM_b(1)$, $\mathbf{f}\in\CF_b$,
 \begin{equation}
  \sum_{\substack{(A_i)_{i\le b}\in_1 \CA_{\mathrm{small}},\\ |\widetilde\partial_\mathrm{int}A_i|=m_i\, \forall i\le b}}\Prob\big(  \CH_{\CG_n}((A_i)_{i\in[b]})\mbox{ is $\mathbf{f}$-connected}\big)\le n(C_{\ref{claim:spanning}}p\beta^{d\alpha})^{b-1}
  \prod_{i\in[b]}\re^{c_\mathrm{pei}m_i}m_i^{f_i}.
  \label{eq-claim:spanning}
 \end{equation}
\end{lemma}
We comment that it is this lemma in the proof that crucially uses that  $\alpha>1+1/d$.
\begin{proof}
 We will prove the statement by induction on $b$.
 We first define the finite constant $C_{\ref{claim:spanning}}$, using that $\alpha>1+1/d$ as follows:
 \begin{equation}
  C_{\ref{claim:spanning}}:=C_{\ref{claim:cross}}\sum_{r=1}^\infty r^{-(\alpha-1)d}.\label{eq:induction-constant}
 \end{equation}
 We start with the initialization.
 Assume first that $b=1$, which corresponds to a tree on a single vertex (representing the block $A_1$), so its forward degree is $f_1=0$.
 We obtain 
 \begin{equation}\nonumber
 \begin{aligned}
  \sum_{A_1 \in \CA_{\mathrm{small}}:|\widetilde\partial_\mathrm{int}A_1|=m_1}\hspace{-20pt}\Prob\big(\CH_{\CG_n}((A_1))\mbox{ is $\mathbf{f}$-connected}\big)
  &\le
  |\{A_1 \in\CA_{\mathrm{small}}:|\widetilde\partial_\mathrm{int}A_1|=m_1\}|\\
  &\!\le\!\sum_{x\in\Lambda_n}|\{A_1 \in\CA_\mathrm{small} : A_1\ni x, |\widetilde\partial_\mathrm{int}A|=m_1\}|.
  \end{aligned}
 \end{equation}
 Since $A=\bar{A}$ for all $A \in\CA_\mathrm{small}$ by  definition in \eqref{eq:a-sets}, we can apply  Lemma \ref{lemma:peierl}, which yields, since $|\Lambda_n|=n$,
 \begin{equation}
  \sum_{A_1 \in\CA_{\mathrm{small}}:|\widetilde\partial_\mathrm{int}A|=m_1}\Prob\big(\CH_{\CG_n}((A_i)_i)\mbox{ is $\mathbf{f}$-connected}\big)
  \le
  \sum_{x\in\Lambda_n}\re^{c_\mathrm{pei} m_1}=n\re^{c_\mathrm{pei} m_1}.\nonumber%\label{eq:spanning-induction-base}
 \end{equation}
 Since $m_1^{f_1}=m_1^0=1$, this finishes the induction base for \eqref{eq-claim:spanning}. We now advance the induction. Assume \eqref{eq-claim:spanning} holds up to $b-1$.
 Let $\mathbf{f}\in\CF_b$ and consider the summation over the last block $A_b \in \CA_{\mathrm{small}}$ on the left-hand side in \eqref{eq-claim:spanning}. By construction of the $\mathbf{f}$-tree in Definition \ref{def:f-tree},  the $b$-th block is a leaf in the $\mathbf{f}$-tree, and $f_b=0$. Its parent in the $\mathbf{f}$-tree is the largest vertex-label $\ell$ in $\mathbf{f}$ that is nonzero, and the remaining labeled graph upon removing $b$ is a tree, with a labeling in $\CF_{b-1}$ (see the comment below Definition \ref{def:f-tree}). Then, the forward degrees of this new tree are given by
 $\mathbf{f}':=(f_1,\ldots,f_{\ell-1}, f_\ell-1,f_{\ell+1},\ldots,f_{b-1})\in\CF_{b-1}$, since the forward degree of the vertex $\ell$ decreased by one upon removing the leaf $b$. With this notation at hand,
 \[
  \{\CH_{\CG_n}((A_i)_{i\in[b]})\mbox{ is $\mathbf{f}$-connected}\}
  =
  \{\CH_{\CG_n}((A_i)_{i\in[b-1]})\mbox{ is $\mathbf{f}'$-connected}\}\cap\{A_\ell\sim_{\CG_n} A_b\}.
 \]
 Independence of edges in $\CG_n$ by Definition \ref{def:lrp} yields
  \begin{align}
  \sum_{(A_i)_{i\le b}\in_1 \CA_\mathrm{small}} \Prob\big(\CH_{\CG_n}((A_i)_{i\in[b]})\mbox{ $\mathbf{f}$-conn.}\big)\le
  \sum_{(A_i)_{i\le b-1}\in_1 \CA_\mathrm{small}}&\Prob\big(\CH_{\CG_n}((A_i)_{i\in[b-1]})\mbox{ $\mathbf{f'}$-conn.}\big)\nonumber\\
  &\cdot \sum_{A_b}\Prob\big(A_\ell\sim_{\CG_n} A_b\big),\label{eq:induction-sum-b}
   \end{align}
 where in the subscripts of the sums (and also in the remainder of the proof) we implicitly assume that the blocks $(A_i)_{i\le b}\in_1 \CA_\mathrm{small}$ are $1$-disconnected, and that
$|\widetilde\partial_{\mathrm{int}}A_i|=m_i$ for all $i\le b$.
 We focus on the summation over $A_b$. Here, $A_b \in\CA_{\mathrm{small}}$ and $A_b$ is $1$-disconnected from $A_{\ell}$.  We decompose the sum according to the length $r$ of an edge $(x,y)$ connecting $x\in A_\ell$ and $y\in A_b$, and use that $\|x-y\|> 1$ by the $1$-disconnectedness of $A_{\ell}, A_b$. By the connection probability \eqref{eq:connection-prob-gen}, $p(1\wedge \beta/r)^{d\alpha}< p\beta^{d\alpha}r^{-d\alpha}$ holds for all $r>0$. So, by a union bound, it follows that
 \begin{equation}\label{eq:al-ab}
  \begin{aligned}
   \sum_{A_b}\Prob\big(A_\ell\sim_{\CG_n} A_b\big)
    & \le
   \sum_{r=1}^\infty \sum_{x\in A_\ell,\, y\in\Z^d\setminus A_\ell}\ind{\|x-y\|\in(r, r+1]}\sum_{A_b\ni y}\Prob\big(x\sim_{\CG_n} y) \\
    & \le
   p\beta^{d\alpha}\sum_{r=1}^\infty r^{-d\alpha} \sum_{x\in A_\ell,\, y\in\Z^d\setminus A_\ell}\ind{\|x-y\|\in(r, r+1]}\sum_{A_b\ni y}1.
  \end{aligned}
 \end{equation}
 Using \eqref{eq:peierl-2} and that the summation in \eqref{eq-claim:spanning} requires that $A_b$ has boundary size $m_b$, we can bound the last sum over $A_b$ from above by $\exp(c_\mathrm{pei}m_b)$. Next, we can apply Claim \ref{claim:cross} to evaluate the summation over $x\in A_\ell, y\in \Z^d \setminus A_\ell$, since this sum equals the cardinality described in \eqref{eq:cross-boundary} with $A=A_{\ell}$. The conditions of the claim are satisfied since $A_\ell=\bar A_{\ell}$ by assuming $A_\ell \in \CA_{\mathrm{small}}$.
 Hence
 \[ \sum_{x\in A_\ell,\, y\in\Z^d\setminus A_\ell}\ind{\|x-y\|\in(r, r+1]}\sum_{A_b\ni y}1\le \mathrm{e}^{c_\mathrm{pei}m_b} C_{\ref{claim:cross}} r^d |\widetilde \partial_{\mathrm{int}} A_{\ell}|= \mathrm{e}^{c_\mathrm{pei}m_b} C_{\ref{claim:cross}} r^d m_\ell.\]
 Substituting this back into \eqref{eq:al-ab}  yields with the constant $C_{\ref{claim:spanning}}$ from \eqref{eq:induction-constant},
 \begin{align*}
  \sum_{A_b}\Prob\big(A_\ell\sim_{\CG_n} A_b\big)
  \le p\beta^{d\alpha} m_\ell\re^{c_\mathrm{pei}m_b}
  C_{\ref{claim:cross}}\sum_{r=1}^\infty r^{(1-\alpha)d}=C_{\ref{claim:spanning}}p\beta^{d\alpha} m_\ell\re^{c_\mathrm{pei}m_b}.
 \end{align*}
 We substitute this bound back into \eqref{eq:induction-sum-b}, and use the induction hypothesis:
 \begin{align*}
  \sum_{(A_i)_{i\le b}\in_1 \CA_{\mathrm{small}}}&\Prob\big(\CH_{\CG_n}((A_i)_{i\in[b]})\mbox{ $\mathbf{f}$-conn.}\big) \\& \le
  C_{\ref{claim:spanning}}p\beta^{d\alpha} m_\ell\re^{c_\mathrm{pei}m_b}\!\!\!\sum_{(A_i)_{i\le b-1}\in_1 \CA_{\mathrm{small}}}\!\!\!\!\Prob\big(\CH_{\CG_n}((A_i)_{i\in[b-1]})\mbox{ $\mathbf{f'}$-conn.}\big)              \\
    & \le
C_{\ref{claim:spanning}}p\beta^{d\alpha} m_\ell\re^{c_\mathrm{pei}m_b} \Bigg(n(C_{\ref{claim:spanning}}p\beta^{d\alpha})^{b-2}\prod_{i\in[b-1]}\re^{c_\mathrm{pei}m_i}m_i^{f_i'}\Bigg) \\                                                  & \le
  C_{\ref{claim:spanning}}p\beta^{d\alpha}m_\ell\re^{c_\mathrm{pei}m_b} n(C_{\ref{claim:spanning}}p\beta^{d\alpha})^{b-2}\frac{1}{m_\ell}\prod_{i\in[b-1]}\re^{c_\mathrm{pei}m_i}m_i^{f_i}.
 \end{align*}
 To obtain the fourth row we used that $f_i'=f_i$ for all $i\neq \ell, i\le b-1$, and $f_\ell'=f_{\ell}-1$ by construction, yielding the $1/m_{\ell}$ factor. We can rearrange the expression and obtain \eqref{eq-claim:spanning}, using that $f_b=0$ (the last block is a leaf). This finishes the proof.
\end{proof}
We are ready to prove Statement \ref{stat:spanning}.

\begin{proof}[Proof of Statement \ref{stat:spanning}]
 Using the result of Lemma \ref{claim:spanning} on the left-hand side of \eqref{eq:stat-spanning} of Statement \ref{stat:spanning}, we arrive at
 \begin{equation}\label{eq:proof-stat-1}
 \begin{aligned}
  \sum_{\mathbf{f}\in\CF_b}
  \bigg(\prod_{i\in[b]}\frac1{f_i!}\bigg)\sum_{(A_i)_{i\le b} \in_1 \CA_{\mathrm{small}}}&\Prob\big(\CH_{\CG_n}((A_i)_i)\mbox{ is $\mathbf{f}$-connected}\big) \\
  &\le
  n \sum_{\mathbf{f}\in\CF_b}(C_{\ref{claim:spanning}}p\beta^{d\alpha})^{b-1}\prod_{i\in[b]}\frac{\re^{c_\mathrm{pei}m_i}m_i^{f_i}}{f_i!}.
  \end{aligned}
 \end{equation}
 We first analyze a single summand, i.e., the value for a fixed $\mathbf{f}\in \CF_b$. Since $\re^{f_i}=\sum_{j\ge 0}(f_i)^j/j!\ge f_i^{f_i}/f_i!$, it follows that  $f_i!\ge (f_i/\re)^{f_i}$. Thus,
 \begin{equation*}
  \begin{aligned}
   n(C_{\ref{claim:spanning}}p\beta^{d\alpha})^{b-1}\hspace{-2pt}\prod_{i\in[b]}\hspace{-1.5pt}\frac{\re^{c_\mathrm{pei}m_i}m_i^{f_i}}{f_i!}
   \le n(C_{\ref{claim:spanning}}p\beta^{d\alpha})^{b-1}
   \exp\Big(c_\mathrm{pei}\sum_{i\in[b]}m_i\Big)
   \prod_{i\in[b]}\Big(\frac{m_i\cdot \re }{f_i}\Big)^{f_i}.
  \end{aligned}
 \end{equation*}
 It follows from  standard differentiation techniques that for any $a,x\ge 1$, the function $g_a(x)=(a \re/x)^x$ is maximized at $x=a$. Maximizing all factors $(m_i\re /f_i)^{f_i} $ at $f_i=m_i$ yields that $(m_i\re /f_i)^{f_i}\le \re^{m_i}$ for all $i\le b$.
 Since by  definition of $\CF_b$ in Definition \ref{def:f-tree} we have $f_1+\ldots+f_b=b-1$ for all $\mathbf{f}\in\CF_b$, it follows
 \[
  \begin{aligned}
   n(C_{\ref{claim:spanning}}p\beta^{d\alpha})^{b-1}\hspace{-2pt}\prod_{i\in[b]}\frac{\re^{c_\mathrm{pei}m_i}m_i^{f_i}}{f_i!} & \le n (C_{\ref{claim:spanning}}p\beta^{d\alpha})^{b-1} \exp\Big((c_\mathrm{pei}+1)\sum_{i\in[b]}m_i\Big)  \\
  & \le n(Cp\beta^{d\alpha})^{\sum_{i\in[b]}m_i}
  \end{aligned}\]
 for  some $C(d, \alpha)=C(c, C_{\ref{claim:spanning}}(d, \alpha))$, where to obtain the last row  we used that  $b-1\le b\le m_1+\ldots+m_b$ as  each block has at least one interior boundary vertex.
 Substituting this back into \eqref{eq:proof-stat-1}, we obtain with  $c''(d,\alpha):=2+c_{\mathrm{pei}}+c'$
 \begin{equation}
  n\sum_{\mathbf{f}\in\CF_b}(C_{\ref{claim:spanning}}p\beta^{d\alpha})^{b-1}\prod_{i\in[b]}\frac{\re^{c_\mathrm{pei}m_i}m_i^{f_i}}{f_i!}
  \le n(Cp\beta^{d\alpha})^{\sum_{i\in[b]}m_i}\sum_{\mathbf{f}\in\CF_b}1.\nonumber
 \end{equation}
 Using again that $f_1+\ldots+f_b=b-1$ for all $\mathbf{f}\in\CF_b$, and using the same combinatorial bounds as for $\mathbf{m}$ above \eqref{eq:after-comb-trick}, we obtain
 $
  |\CF_b|\le \binom{2b}{b}\le 2^{2b}\le \exp\big(2\sum_{i\in[b]}m_i\big),
 $
 finishing the proof for some $C_{\ref{stat:spanning}}(d,\alpha)=C_{\ref{stat:spanning}}(d,\alpha)>0$.
\end{proof}

\subsection*{Proof of Statement \ref{stat:disconn}}
We start with a geometric claim. Recall $\CA_{\mathrm{small}}$ from \eqref{eq:a-sets}, and  holes from Def.~\ref{def:blocks}.
\begin{claim}\label{claim:complement-star-conn}
 Let $(A_1, \dots, A_b)\in_1 \CA$ be a $1$-disconnected sequence of blocks without holes (i.e., $A_i=\bar A_i$ for all $i\le b$), with $A:=\cup_{i\le b}A_i\subseteq \Lambda_n$. 
 Then, $(\Lambda_n \setminus A) \cup (\cup_{i\le b}\widetilde\partial_\mathrm{int}A_i)\supseteq\widetilde\partial_\mathrm{int}\Lambda_n$.
 Moreover,
 $(\Lambda_n \setminus A) \cup (\cup_{i\le b}\widetilde\partial_\mathrm{int}A_i)$  is $\ast$-connected. 
\end{claim}
\begin{proof}
We first show that $(\Lambda_n \setminus A) \cup (\cup_{i\le b}\widetilde\partial_\mathrm{int}A_i)\supseteq\widetilde\partial_\mathrm{int}\Lambda_n$.
When $x \in (\widetilde\partial_\mathrm{int}\Lambda_n \setminus A)$, then $x\in\Lambda_n\setminus A$. The (only) other case is when $x\in \widetilde\partial_\mathrm{int}\Lambda_n\cap A$. Then there must exist $A_i\subseteq\Lambda_n$ such that $x\in A_i$. Since $x\in\widetilde\partial_\mathrm{int}\Lambda_n$ and $A\subseteq\Lambda_n$, it follows from \eqref{eq:ext-boundary} in  Definition \ref{def:boundary} that there is a vertex $y$ in $\Z^d\setminus \Lambda_n$ neighboring $x$. Since $x\in A_i\subseteq \Lambda_n$, we evidently have $y\notin A_i$, hence  $y\in \widetilde \partial_{\mathrm{ext}}A_i$. 
As a result of \eqref{eq:int-boundary}, $x\in \widetilde\partial_\mathrm{int}A_i$, establishing the statement.

We turn to prove $\ast$-connectedness of $(\Lambda_n \setminus A) \cup (\cup_{i\le b}\widetilde\partial_\mathrm{int}A_i)$.
  Using Claim \ref{claim:unique}, we decompose the set $\Lambda_n\!\setminus\! A$ into a $1$-disconnected sequence of $1$-connected blocks $(B_j)_{j\le b'}$ for some $b'\ge 1$:
\begin{equation}\label{eq:b-j-decomp}
\cup_{j\le b'}B_j= \Lambda_n \setminus A.
\end{equation}
We define now an auxiliary graph. For each set $\widetilde\partial_\mathrm{int}A_i$  we associate a vertex $a_i$ for $i\le b$, and also for each block $B_j$  a vertex $t_j$, for $j\le b'$. 
 We define an auxiliary graph $\CH$ on the vertex set $V=\{a_1\dots, a_b, t_1, \dots, t_{b'}\}$. We say that $a_i\sim_\CH t_j$ if there is a pair of vertices $x\in \widetilde \partial_{\mathrm{int}} A_i, y\in B_j$ that are $\ast$-connected, i.e., $\|x-y\|_\infty=1$. Similarly, $a_i\sim_\CH a_j$ if there is a pair of vertices $x\in \widetilde \partial_{\mathrm{int}} A_i, y\in \widetilde \partial_{\mathrm{int}} A_j$ with $\|x-y\|_\infty=1$. We will show that $\CH$ consists of a single connected component. This then implies that $(\Lambda_n\setminus A) \cup (\cup_{i\le b}\partial_{\mathrm{int}} A_i)$ is $\ast$-connected. 

To show that $\CH$ consists of a single connected component, we argue as follows. 
Let $f: (\Lambda_n\setminus A) \cup (\cup_{i\le b}\widetilde \partial_{\mathrm{int}} A_i) \to\{0,1\}$ be a function that is constant on $\ast$-connected subsets of its domain. 
Since we assumed that $A_i=\bar A_i$ for all $i\le b$, Claim \ref{claim:containment}(iv) is applicable which states that $\widetilde\partial_\mathrm{int}A_i$ is $\ast$-connected for each $i\le b$. Further, $(B_j)_{j\le b'}$ are blocks, i.e., $1$-connected and also $\ast$-connected.
So, $f(x)=f(y)=: f_\CH(a_i)$ for all $x,y\in \widetilde \partial_{\mathrm{int}} A_i$, and also $f(x)=f(y)=: f_\CH(t_j)$ for all $x,y\in B_j$. This defines a function $f_\CH:V\to \{0,1\}$. Since $f$ is constant on $\ast$-connected subsets of its domain, $f_\CH$ is constant on each of the connected components of $\CH$.  Using again that each $\widetilde \partial_{\mathrm{int}} A_i$ is $\ast$-connected and $\ast$-connected to all vertices in $A_i$, $f(x)=f_\CH(a_i)$ for all $x\in \widetilde \partial_{\mathrm{int}} A_i$, and $f$ can be uniquely extended to a function $g$ that takes the value $g(x):= f_\CH(a_i)$ on \emph{all} vertices of $A_i$. If we then set $g(y):=f_\CH(t_j)$ for all $y\in B_j$, then $g: \Lambda_n \to \{0,1\}$ is a function that is constant on $\ast$-connected components of $(\Lambda_n \setminus A)\cup (\cup_{i\le b}A_i)=(\Lambda_n\setminus A)\cup A=\Lambda_n$. Now we may notice that $\Lambda_n$ consists of a single $\ast$-connected component, hence $g$ must be constant everywhere. This implies that $f$ and hence $f_\CH $ must have been also a constant everywhere. This implies that $\CH$ consists of a single connected component and so $(\Lambda_n\setminus A) \cup (\cup_{i\le b}\widetilde \partial_{\mathrm{int}} A_i)$  is $\ast$-connected.
\end{proof}
\begin{proof}[Proof of Statement \ref{stat:disconn}]
Let $A_1,\ldots, A_b\in_1\CA_\mathrm{small}$, and denote $A:=\cup_{i\le b}A_i$.   We first assume $\beta\ge 1$.  We define the set of potential edges between the interior boundary of $A$ with respect to $\Z^d$ and the set of vertices outside $A$ within distance $\beta$ as
 \begin{equation}
  \Delta(A) := \big\{\{x,y\} \mid  x\in \cup_{i\le b}\widetilde\partial_\mathrm{int} A_i, y\in(\Lambda_n\setminus A): \|y-x\|\in[1, \beta]\big\}.\nonumber
 \end{equation}
 Considering the event on the left-hand side of \eqref{eq:stat-isolation} in Statement~\ref{stat:disconn}, there should be no edges in $\CG_n$ between $(\cup_{i\le b}\widetilde\partial_\mathrm{int}A_i)$ and $\Lambda_n\setminus A$. In particular, all edges in $\Delta(A)$ must be absent. The distance $\beta$ is chosen so that all edges of this length are present in $\CG_n$ with probability $p$ by \eqref{eq:connection-prob-gen}. Hence,
 \begin{equation}\label{eq:1-p-deltaA}
  \Prob\Big((\cup_{i\le b}\widetilde\partial_\mathrm{int}A_i)\not\sim_{\CG_n}(\Lambda_n\setminus\cup_{i\le b} A_i)\Big)\le (1-p)^{|\Delta(A)|}.
 \end{equation}
 Our goal is to show that for some constant $c=c(d)>0$,
 \begin{equation}\label{eq:delta-goal}
  |\Delta(A)| \ge c\beta\sum_{i\in [b]}|\widetilde\partial_{\mathrm{int}}A_i|=c\beta\sum_{i\in [b]}m_i,
 \end{equation}
 which then immediately yields \eqref{eq:stat-isolation} in combination with \eqref{eq:1-p-deltaA} when $\beta\ge 1$.
 In what follows we estimate $|\Delta(A)|$. In order to do so, we will make use of the boundary $\partial_{\mathrm{int}}A_i$, i.e., the interior boundary with respect  to the box $\Lambda_n$.
 Using that all blocks in $\CA_{\mathrm{small}}$ have size at most $3n/4$ by definition in \eqref{eq:a-sets}, the conditions of the isoperimetric inequality in Claim \ref{claim:iso} are satisfied, and hence $\delta m_i \le  |\partial_{\mathrm{int}}A_i|\le m_i$ for all $i\le b$.
 Hence, \eqref{eq:delta-goal} is equivalent to showing that
 that there exists $c'=c'(d)>0$ such that for any $1$-disconnected blocks $A_1,\ldots, A_b\in_1\CA_\mathrm{small}$, with $A=\cup_{i\le b}A_i$
 \begin{equation}
  |\Delta(A)| \ge c'\beta\sum_{i\in[b]}|\partial_\mathrm{int} A_i|,
  \label{eq:bound-delta}
 \end{equation}
 since then \eqref{eq:delta-goal} holds with $c=c'\delta$.
In order to show \eqref{eq:bound-delta}, our first goal is to find enough pairs of vertices in $\Delta(A)$ around a linear fraction of vertices in  $\cup_{i\le b}\partial_\mathrm{int} A_i$.
 For this, we claim that a set  $\CT:=\{(x_\ell, y_\ell)\}_{\ell\ge 1}$ with the following properties exists:
 \begin{itemize}
  \item[(i)] $x_\ell\in\cup_i\partial_\mathrm{int}A_i$, $y_\ell\in\cup_i\partial_\mathrm{ext} A_i$, and $\|x_\ell-y_\ell\|=1$ for all $\ell\ge 1$;
  \item[(ii)] each vertex $z\in\Lambda_n$ appears at most once in a pair in $\CT$;
  \item[(iii)] $|\CT| \ge\sum_{i\in[b]}|\partial_\mathrm{int} A_i|/(2d)$.
 \end{itemize}
 Note that requirement (i) implies that all $(x_\ell, y_\ell) \in \CT$ are elements of $\Lambda_n\times \Lambda_n$.
 We now show that a set $\CT$ exists. Consider the following greedy algorithm: order the vertices in $\cup_{i} \partial_\mathrm{int} A_i$ in an arbitrary order, to obtain the list $(v_1, v_2, \dots, v_{M})$ with $M=\sum_{i\in[b]}|\partial_\mathrm{int} A_i|$. Since each $v_j$ is in $ \cup_i\partial_\mathrm{int}A_i $, for each $v_j$ there is at least one vertex $y_j\in \cup_i\partial_\mathrm{ext} A_i\subseteq (\Lambda_n\setminus A)$ with $\|v_j-y_j\|=1$ by Definition \ref{def:boundary} (recall that the sets $A_1, \dots, A_b$ are $1$-disconnected). Starting with $\CT_1:=\{(v_1, y_1)\}$,
 going through the ordering of $(v_j)_j$ one-by-one, append the pair $(v_j, y_j)$ to the list $\CT_{j-1}$, if and only if $y_j$ has not been contained in any pair of $\CT_{j-1}$ yet and so obtain $\CT_j$. Set then $\CT:=\CT_{M}$.
 Since any $y\in \cup_i\partial_\mathrm{ext} A_i$ neighbors at most $2d$ many interior boundary vertices, adding a certain pair $(v_j, y_j)$ only affects at most $2d-1$ other indices where a pair may not be added later.
 Hence,
 \begin{equation}\label{eq:T-size}
  |\CT|\ge \frac{1}{2d}\sum_{i\in[b]}|\partial_\mathrm{int} A_i|.
 \end{equation}
 Next, assume that $\beta \ge 2\sqrt{d}+2$ and set $R:=\lfloor (\beta-1)/\sqrt{d} \rfloor\ge 1$.
 Take any pair $(x_\ell, y_\ell)\in \CT$. Since $\cup_{i\le b}\widetilde \partial_{\mathrm{int}} A_i \cup (\Lambda_n\setminus A)$ is $\ast$-connected by Claim \ref{claim:complement-star-conn}, we may fix for each $x_\ell$  a self-avoiding path \begin{equation}\label{eq:selfavoiding}\pi(x_\ell)=(x_\ell, z_1^{\sss{(x_\ell)}},\dots, z_R^{\sss{(x_\ell)}})\subseteq \cup_{i\le b}\widetilde \partial_{\mathrm{int}} A_i \cup (\Lambda_n\setminus A)\end{equation} (since the set on the right-hand side contains $\widetilde\partial_\mathrm{int}\Lambda_n$, which has cardinality $\Theta(n^{(d-1)/d})$, the set on the right-hand side has size at least $R$ for $n$ sufficiently large, so such a self-avoiding path of length $R$ then exists).  
 By the triangle inequality,
 \begin{equation}\label{eq:triangle}
  \|x_\ell-z_j^\sss{(x_\ell)}\|\le \sqrt{d}R\le \beta, \quad \mbox{and} \quad \|y_\ell-z_j^\sss{(x_\ell)}\|\le \sqrt{d}R+1\le \beta.
 \end{equation}
 Define now the type $\mathrm{typ}(z_j^\sss{(x_\ell)}):=x_\ell$ if $z_j^\sss{(x_\ell)}\in (\Lambda_n\setminus A)$ and set $\mathrm{typ}(z_j^\sss{(x_\ell)}):=y_\ell$ if $z_j^\sss{(x_\ell)}\in \cup_{i\le b}\widetilde \partial_{\mathrm{int}} A_i$.
 Then define the set of (unordered) pairs representing potential edges in $\CG_n$
 \begin{equation}\label{eq:delta-x-y} \Delta(x_\ell,y_\ell):=\Big\{ \{z_1^\sss{(x_\ell)}, \mathrm{typ}(z_1^\sss{(x_\ell)})\}, \dots, \{z_R^\sss{(x_\ell)}, \mathrm{typ}(z_R^\sss{(x_\ell)})\}\Big\}\subseteq \Delta(A).\end{equation}
 The inclusion holds since for each of these pairs, exactly one element is in $\Lambda_n\setminus A$ and the other one is in $\cup_{i\le b}\widetilde \partial_{\mathrm{int}} A_i$, and the distance between the two vertices of each pair is at most $\beta$ by \eqref{eq:triangle}.
 We claim that
 \begin{equation}\label{eq:delta-a-x-y}
  |\Delta(A)|\ge \Big|\bigcup_{(x_\ell, y_\ell)\in \CT} \Delta(x_\ell,y_\ell)\Big| \ {\buildrel \diamond \over \ge }\ (1/2)\sum_{\ell=1}^{|\CT|} |\Delta(x_\ell,y_\ell)|= |\CT| \cdot R/2. 
  \end{equation}
 To see the inequality with $\diamond$, we show that each potential edge $\{z,z'\}\in \Lambda_n\times \Lambda_n$ appears at most twice in a set in the union in the middle. 
 Consider $\{z,z'\}\in\Lambda_n\times\Lambda_n$. First, assume that there exists $\ell$ such that $(z,z')\!=\!(x_\ell, y_\ell) \in \CT$ or $(z',z)\!=\!(x_\ell, y_\ell)\in\CT$. Without loss of generality, we assume that the pair is ordered such that $(z,z')\!=\!(x_\ell, y_\ell)\in\CT$. Then there is no $(x_j, y_j)\in\CT$ different from $(x_\ell, y_\ell)$ such that $\{z,z'\}\in\Delta(x_j, y_j)$, since each element in $\Delta(x_j, y_j)$ contains either $x_j$ or $y_j$, which are different from $x_\ell$ and from $y_\ell$ by requirement (ii) in the construction of $\CT$.
 Moreover, the element $\{z,z'\}\!=\!\{x_\ell, y_\ell\}$ is contained at most once in the set $\Delta(x_\ell, y_\ell)$, since the first coordinates in \eqref{eq:delta-x-y} are all different as they form a self-avoiding path, and the first coordinates do no not contain $x_\ell=z$ by \eqref{eq:selfavoiding}. So, if $(z, z')\in \CT$, then this pair of vertices only appears once in 
$\big|\cup_{(x_\ell, y_\ell)\in \CT} \Delta(x_\ell,y_\ell)\big|$. 

 Next, assume that $(z,z'),(z',z)\notin\CT$, but $\{z,z'\}$ is contained in some $\Delta(x_\ell, y_\ell)$. Then, either $z$ or $z'$ must be equal to either  $x_\ell$ or to $y_\ell$ by \eqref{eq:delta-x-y}. Assume without loss of generality that $z \in \{x_\ell, y_\ell\}$, and therefore $z' \notin \{x_\ell, y_\ell\}$. Thus, $\{z,z'\}$ is contained exactly once in $\Delta(x_\ell, y_\ell)$.
 The only way that $\{z, z'\}$ could be in a set $\Delta(x_{\ell'}, y_{\ell'})$ for  some $(x_{\ell'}, y_{\ell'})\neq (x_\ell, y_\ell)$, is when $z' \in \{x_{\ell'}, y_{\ell'}\}$ and $(x_{\ell'}, y_{\ell'})\in\CT$ for some $\ell' \neq \ell$. (This can only happen if the self-avoiding path from $x_\ell$ passes through either $x_{\ell'}$ or through $y_{\ell'}$ \emph{and} the self-avoiding path from $x_{\ell'}$ passes through either $x_{\ell}$ or $y_{\ell}$). This argument implies that the element $\{z,z'\}$ can be contained at most twice in a set in the union in \eqref{eq:delta-a-x-y}, namely in $\Delta(x_\ell, y_\ell)$ and $\Delta(x_{\ell'}, y_{\ell'})$, and the inequality $\diamond$ in \eqref{eq:delta-a-x-y} holds.
  
 Combining \eqref{eq:delta-a-x-y} with \eqref{eq:T-size},  $R=\lfloor (\beta-1)/\sqrt{d} \rfloor$, and the assumption that $\beta\ge 2\sqrt{d}+2$, (see before \eqref{eq:triangle}), we arrive at
 \[
  |\Delta(A)| \ge \frac{R}{4d}\sum_{i\in[b]} |\partial_\mathrm{int} A_i|=\frac{1}{4d}\left\lfloor\frac{\beta-1}{\sqrt{d}}\right\rfloor\sum_{i\in[b]} |\partial_\mathrm{int} A_i|\ge\frac{\beta}{8d\sqrt{d}}\sum_{i\in[b]} |\partial_\mathrm{int} A_i|,
 \]
 since whenever $x\ge 2\sqrt{d}+2$, then $\lfloor (x-1)/\sqrt{d} \rfloor \ge   x/2\sqrt{d}$.

 This proves \eqref{eq:bound-delta} whenever $\beta \ge 2\sqrt{d}+2$. For the case $1\le\beta \le  2\sqrt{d}+2$, we use that each vertex on the interior boundary is within distance one from a vertex on the exterior boundary, hence
 \[
  |\Delta(A)| \ge \sum_{i\in[b]}|\partial_\mathrm{int} A_i|\ge \frac{\beta}{2\sqrt{d}+2}\sum_{i\in[b]}|\partial_\mathrm{int} A_i|,
 \]
 and so~\eqref{eq:bound-delta} holds for both cases with  $c'(d):=1/\max\{8d\sqrt{d}, 2\sqrt{d}+2\}$. This finishes the proof of Statement \ref{stat:disconn} when $\beta\ge 1$ for $c_{\ref{stat:disconn}}(d):=\delta c'(d)$.

 Assume now $\beta<1$. Each vertex on the interior boundary of $A=\cup_{i\le b}A_i$ is at distance $1$ from at least one vertex in $\Lambda_n\setminus A$ by definition. We again use that the conditions of the isoperimetric inequality in Claim~\ref{claim:iso} are satisfied, so that $|\delta_\mathrm{int}A_i|\ge \delta m_i$ for all $i\le b$. Each individual edge of length $1$ is absent with probability $(1-p\beta^{d\alpha})$. Thus, only excluding such edges, we obtain the second case of Statement \ref{stat:disconn} since $c_{\ref{stat:disconn}}(d)\le\delta$.
\end{proof}

\section{Counting holes}
We turn to the proof of Lemma \ref{lemma:holes}. We set up a few preliminaries about holes. Recall that $\CA$ denotes the family of $1$-connected blocks in $\Lambda_n$ from \eqref{def:calA}, and that
$\CA_\mathrm{large}=\{A\in\CA: |\bar A|>3n/4, \ |A|\le n/2\}$ from \eqref{eq:a-sets}. 
Recall also from Definition \ref{def:blocks} that the holes $\mathfrak{H}_A$ of a $1$-connected set $A\in\CA$ are the $1$-connected subsets of $\bar{A}\setminus A$. By Definition \ref{def:blocks}, each hole $H\in\mathfrak{H}_A$ is surrounded by $A$. This implies that $H$ does not intersect the boundary of the box $\widetilde\partial_\mathrm{int}\Lambda_n$. 
Hence, it follows
by Definition \ref{def:boundary} of the boundaries that for all  $H\in\mathfrak{H}_A$, 
\begin{equation}
\partial_\mathrm{int}H\,  =\, \widetilde\partial_\mathrm{int}H\, \subseteq\, \partial_\mathrm{ext} A.\label{eq:boundary-in-box-identity}
\end{equation}
Hence, comparing this to 
$\CE_2:=\CE_2(\CG_n)=\{\exists A\in\CA_\mathrm{large}: A\not\sim_{\CG_n}\partial_\mathrm{ext}A\}$ from \eqref{eq:def-event-2-holes}, we obtain that 
\begin{equation}\label{eq:e2-bound-1}
    \Prob\big(\CE_2\big)\le \Prob\big(\exists A\in\CA_\mathrm{large}: A\not\sim_{\CG_n} (\cup_{H\in\mathfrak{H}_A}\partial_\mathrm{int}H)\big).
\end{equation}
The following definition (and claim) of principal holes will ensure that the total size of the boundaries of the holes of $A$ in \eqref{eq:e2-bound-1} is sufficiently large compared to the combinatorial factor arising from the number of possible sets $A$ there.
\begin{definition}[Principal holes]\label{def:holes}
 Let $A\subseteq\Lambda_n$. A hole $H\in\mathfrak{H}_A$ has type $i\in\N$ if $|H|\in(2^{i-1}, 2^i]$. We write $\mathfrak{H}_A(i)\subseteq\mathfrak{H}_A$ for the set of holes of $A$ of type $i$. A hole-type $i$ is called \emph{principal} for a set $A$ if
 \begin{equation}
  |\mathfrak{H}_A(i)|=|\{H: H\in \mathfrak{H}_A(i)\}|\ge 2^{-i-3}i^{-2} n=:h_n(i). \label{eq-def:hni}
 \end{equation}
\end{definition}
Since $|\mathfrak{H}_A(i)|$ is an integer for all $i$, the inequality $|\mathfrak{H}_A(i)|\ge \lceil h_n(i)\rceil$ also holds  whenever the inequality in \eqref{eq-def:hni} holds. 
We define for $i\in\N$
\begin{equation}\label{eq:alarge-i}
  \CA_\mathrm{large}(i)  := \{A\in\CA_\mathrm{large}: |\mathfrak{H}_A(i)|\ge \lceil h_n(i)\rceil \big\},                                       \end{equation}
and observe that $A$ might appear in both $\CA_\mathrm{large}(i)$ and $\CA_\mathrm{large}(j)$ if both type $i$ and type $j$ are principal for $A$.   
We  define the following $\beta$-dependent constants: 
\begin{equation}
    R_2= R_2(\beta):=\big\lfloor \beta/\sqrt{d}\big\rfloor\vee 1,
    \qquad 
    i_\star=i_\star(\beta):=1+\left\lceil \tfrac{d}{d-1}\log_2 R_2\right\rceil.
    \label{eq:r-holes}
\end{equation}
Here, if two vertices are within $\|\cdot\|_\infty$-distance $R_2$ then they are within $\|\cdot\|_2$-distance $\beta$, hence they are connected by an edge with probability $p(1\wedge\beta)^{d\alpha}$ in $\CG_n$ by \eqref{eq:connection-prob-gen}. The definition of $i_\star$ ensures that any hole of type $i\ge i_\star$ has an exterior boundary of size at least $R_2$ (the exterior boundary of a hole is a subset of  $\partial_\mathrm{int}A$) to which the vertices on the interior boundary  of the hole (a subset of $\partial_\mathrm{ext}A$) should not be connected in $\CG_n$.
\begin{claim}[Large blocks have a principal hole-type]\label{claim:type-holes}
 For all $A\in\CA_\mathrm{large}$, there exists $i_A\le\lceil\log_2 n\rceil$ such that hole-type $i_A$ is principal for the set $A$, i.e.,
 \begin{equation} \label{eq:alarge-covered} \CA_\mathrm{large}\subseteq \bigcup_{i\le \lceil \log_2 n\rceil} \CA_{\mathrm{large}}(i). \end{equation}
 There exists a constant $c_{\ref{claim:type-holes}}=c_{\ref{claim:type-holes}}(d)>0$ such that for all $A\subseteq\Lambda_n$, with any $i_A$ being any principal hole-type for $A$,
 \begin{equation}
\sum_{H\in\mathfrak{H}_A(i_A)}|\partial_\mathrm{int} H_i|\ge c_{\ref{claim:type-holes}}\,i_A^{-2}\, 2^{-i_A/d}\, n.\label{eq-claim:type-holes-2}
 \end{equation}
 Moreover, for each hole $H$ with type $i\ge i_\star$ in \eqref{eq:r-holes},
 \begin{equation}
    |\partial_\mathrm{ext}H| \ge R_2.\label{eq:min-size-hole-boundary}
\end{equation}
\end{claim}
 \begin{proof}
  We argue by contradiction for the first part.
  By definition of $\CA_{\mathrm{large}}$ in \eqref{eq:a-sets}, $|\bar A| \ge 3n/4$, and also  $|A|\le n/2$. Hence, the total size of the holes is at least $n/4$, i.e., $|\cup_{H\in \mathfrak{H}_A} H|=\sum_{H\in \mathfrak{H}_A} |H|\ge n/4$ hold.
  Suppose \eqref{eq-def:hni} holds in the opposite direction for all $i\ge 1$. Since the holes are $1$-disconnected, it follows from the size requirement in Definition \ref{def:holes} that
  \begin{align}
   \Big|\bigcup_{H\in\mathfrak{H}_A}H\Big| = \sum_{H\in\mathfrak{H}_A}|H| \le \sum_{i\ge 1}|\mathfrak{H}_A(i)|2^i \le \frac{n}{8}\sum_{i\ge 1}i^{-2} <n/4,\nonumber
  \end{align}
  since the sum converges to $\pi^2/6=1.64...<2$. This contradicts the assumption that the total size is at least $n/4$, so there must be at least one principal hole-type, say $i_A$. The restriction $i\le\lceil\log_2 n\rceil$ follows since the number of vertices in $\Lambda_n$ is $n$, and $H>2^{(\lceil\log_2 n\rceil+1)-1}$ thus can never be satisfied. This shows \eqref{eq:alarge-covered}.
  
  We turn to \eqref{eq-claim:type-holes-2}.
   As argued before \eqref{eq:boundary-in-box-identity}, holes do not intersect the boundary of the box $\partial_\mathrm{int}\Lambda_n$. So, $\partial_{\mathrm{ext}}H=\widetilde \partial_{\mathrm{ext}}H$, and  $|\partial_\mathrm{ext}H|\ge |H|^{(d-1)/d}$ by Claim \ref{claim:iso} for each hole. Combined with  $|H|> 2^{i_A-1}$ for all $H\in\mathfrak{H}_A(i_A)$ and the lower bound $|\mathfrak{H}_A(i_A)|\ge 2^{-i_A-3} i_A^{-2} n$ in \eqref{eq-def:hni}, 
    this yields that
  \[
  \begin{aligned}
   \sum_{H\in\mathfrak{H}_A(i_A)}|\partial_\mathrm{int} H|&\ge \sum_{H\in\mathfrak{H}_A(i_A)} |H|^{(d-1)/d} \ge   2^{-(d-1)/d}2^{i_A(d-1)/d}  |\mathfrak{H}_A(i_A)| \\ &\ge 
    2^{-(d-1)/d}2^{i_A(d-1)/d}\cdot 2^{-i_A-3}i_A^{-2}n
   \ge c_{\ref{claim:type-holes}}\,i_A^{-2}\,2^{-i_A/d}\,n,
  \end{aligned}\]
  for some constant $c_{\ref{claim:type-holes}}=c_{\ref{claim:type-holes}}(d)>0$. 
  Lastly, we prove \eqref{eq:min-size-hole-boundary}. Using again that $\partial_{\mathrm{ext}}H\ge |H|^{(d-1)/d}$ for each hole, we obtain for any hole $H$ with type $i\ge i_\star$,
\begin{equation}
    |\partial_\mathrm{ext}H| \ge |H|^{(d-1)/d} >  2^{(i_\star-1)(d-1)/d}
    \ge \big(2^{(\log_2R_2)d/(d-1)}\big)^{(d-1)/d}
    =  R_2.\nonumber
\end{equation}
This finishes the proof of the claim.
\end{proof}

We use the inclusion in \eqref{eq:alarge-covered} to bound the event on right-hand side in \eqref{eq:e2-bound-1} 
(with the convention that the empty sum from $1$ to $i_\star$ is $0$). After a union bound we arrive at
 \begin{align}
  \Prob\big(\CE_2\big) & \le \sum_{i=1}^{i_\star-1}\Prob\big(\exists A\in\CA_\mathrm{large}(i): A\not\sim_{\CG_n} \cup_{H\in\mathfrak{H}_A} \partial_\mathrm{int}H\big) \label{eq:event-2-splitted-1}
   \\&\hspace{15pt}+\sum_{i=i_\star}^{\lceil\log_2 n\rceil}\Prob\Big(\exists A\in\CA_\mathrm{large}(i): A\not\sim_{\CG_n} \cup_{H\in\mathfrak{H}_A} \partial_\mathrm{int}H\Big).\label{eq:event-2-splitted-2}
 \end{align}
We  now bound these two sums, the first one corresponding to \emph{small principal hole types}, the second one corresponding to \emph{large principal hole types}.
\subsection*{Excluding small principal hole types}
\begin{claim}[Sets with small principal hole-types are unlikely components]\label{stat:small-holes}
Let $\CG_n$ be long-range percolation on $\Lambda_n$ as in Definition \ref{def:lrp} with $d\ge 2$, $\alpha> 1$, and $i_\star(\beta)$ from \eqref{eq:r-holes}. 
 Then there exists $c_{\ref{stat:small-holes}}=c_{\ref{stat:small-holes}}(d)>0$ such that for all $\beta>0$ and $n$ sufficiently large,
 \begin{equation}
 \begin{aligned}
\mathrm{Err}_{\mathrm{small}}:= \sum_{i=1}^{i_\star-1}\Prob & \big(\exists A\in\CA_\mathrm{large}(i): A\not\sim_{\CG_n} \cup_{H\in\mathfrak{H}_A}\partial_\mathrm{int}H\big) \\
                            & \le
 \ind{\beta\ge2\sqrt{d}} 2^n\cdot\big((1-p)^{c_{\ref{stat:small-holes}}(1+\log_2\beta)^{-2}\beta^{(d-2)/(d-1)}}\big)^n
  .
 \end{aligned}\label{eq:small-prominent-last}
\end{equation}
\end{claim}
The claim shows that the probability that large sets with small principal hole types appear as a  component of $\CG_n$, decays exponentially in $n$ whenever the base of the second exponential factor is strictly smaller than $1/2$. Note in particular that this error term is not present when $\beta<1$.
\begin{proof}
We may assume that $i_\star(\beta)\ge2$ in \eqref{eq:r-holes}, since otherwise the sum would be empty and the bound holds trivially. By definition of $i_\star$ and $R_2(\beta)$ in~\eqref{eq:r-holes}, we may therefore assume that $\log_2R_2>0$, which is equivalent to $R_2(\beta)=\lfloor \beta/\sqrt{d}\rfloor>1$ and $\beta\ge 2\sqrt{d}$. We assume $R_2(\beta)>1$ throughout the remainder of the proof. We start estimating a single summand on the left-hand side of \eqref{eq:small-prominent-last}.
Consider some $A\in \CA_{\mathrm{large}}(i)$. By Definition \ref{def:holes},
\begin{equation}\nonumber%\label{eq:hA}
 h_A:=|\mathfrak{H}_A|\ge |\mathfrak{H}_A(i)|\ge h_n(i) = 2^{-i -3}i^{-2}n.
\end{equation}
We now find enough potential edges that all must be absent in order for the event $\{A\not\sim_{\CG_n} \cup_{H\in\mathfrak{H}_A}\partial_\mathrm{int}H\} $ in \eqref{eq:small-prominent-last} to occur. By \eqref{eq-claim:type-holes-2} in Claim \ref{claim:type-holes}, since $i$ is a principal hole-type of $A\in\CA_\mathrm{large}(i)$
\begin{equation}\label{eq:ell-i}
|\cup_{H\in\mathfrak{H}_A}\partial_\mathrm{int}H|=\sum_{H\in \mathfrak{H}_A(i)}|\partial_\mathrm{int}H|\ge c_{\ref{claim:type-holes}}i^{-2}2^{-i/d}n=:\ell_i.
\end{equation} 
We now obtain a lower bound on $|A|$ using the isoperimetric inequality of $\Z^d$ in Claim \ref{claim:iso}.
By definition of  $\widetilde\partial_\mathrm{int}A$ in Definition \ref{def:boundary}, and since $\widetilde\partial_\mathrm{int}A\supseteq \widetilde\partial_\mathrm{int}\bar A$ by Claim \ref{claim:containment}(i), it follows from Claim \ref{claim:iso} applied to $\bar A $ that for all $A\in\CA_\mathrm{large}$, 
\begin{equation}
    |A|\ge |\widetilde\partial_\mathrm{int}A|\ge |\widetilde\partial_\mathrm{int}\bar{A}|\ge |\bar{A}|^{(d-1)/d}\ge (3/4)^{(d-1)/d}n^{(d-1)/d}.
   \nonumber% \label{eq:min-size-bad}
\end{equation}
Take now a vertex $x\in \partial_\mathrm{int}H\subseteq \partial_\mathrm{ext} A$. Then, since $|A|$ diverges with $n$ and $A$ is $1$-connected, whenever $n$ is sufficiently large compared to $\beta$,
\[|\{y\in A: \|y-x\|_1\le \lfloor\beta\rfloor\}| \ge \lfloor\beta\rfloor, \qquad \forall x\in \cup_{H\in \mathfrak{H}_A} \partial_{\mathrm{int}}H.\]
So, 
by \eqref{eq:ell-i}, 
\[ \big|\big\{\{x,y\}: x\in \cup_{H\in \mathfrak{H}_A} \partial_{\mathrm{int}}H, y\in A: \|y-x\|_1\le \lfloor\beta\rfloor\big\}\big|\ge \lfloor\beta\rfloor\cdot |\cup_{H\in \mathfrak{H}_A} \partial_{\mathrm{int}}H| \ge \lfloor\beta\rfloor\ell_{i}.\] 
These edges must be all absent in order for $\{A\not\sim_{\CG_n} \cup_{H\in\mathfrak{H}_A}\partial_\mathrm{int}H\}$ to occur for $A\in\CA_\mathrm{large}(i)$ in \eqref{eq:small-prominent-last}. The connection probability in \eqref{eq:connection-prob-gen} ensures that two vertices within $\ell_1$-distance $\lfloor\beta\rfloor$ are connected with probability $p$. Hence using a union bound and then the independence of edges, we obtain 
\begin{equation}
\begin{aligned}
 \Prob\big(\exists A\in\CA_\mathrm{large}(i): A\not\sim_{\CG_n} \cup_{H\in\mathfrak{H}_A}\partial_\mathrm{int}H\big) &\le \sum_{A\in \CA_\mathrm{large}(i)} \Prob\big(A\not\sim_{\CG_n} \cup_{H\in\mathfrak{H}_A}\partial_\mathrm{int}H\big) \\
 &\le \sum_{A\in \CA_\mathrm{large}(i)}  (1-p)^{\lfloor\beta\rfloor \ell_i}\le 2^n (1-p)^{\lfloor\beta\rfloor \ell_i},
 \end{aligned}\nonumber
\end{equation}
where we used that $\CA_\mathrm{large}(i)$ counts subsets of $\Lambda_n$, and the number of subsets of $\Lambda_n$ is at most $2^n$. This bounds a single summand in \eqref{eq:small-prominent-last}. To evaluate the sum, recalling that $\ell_i=c_{\ref{claim:type-holes}}i^{-2}2^{-i/d}n$ from \eqref{eq:ell-i}, and that $i_\star(\beta)=1+\lceil \tfrac{d}{d-1} \log_2R_2\rceil$ in \eqref{eq:r-holes} we have
\begin{equation}
\begin{aligned}
\mathrm{Err}_{\mathrm{small}} &\le \sum_{i=1}^{i_\star(\beta)-1} 2^n (1-p)^{\lfloor\beta\rfloor\ell_i}=2^n\sum_{i=1}^{i_\star(\beta)-1} (1-p)^{\lfloor\beta\rfloor c_{\ref{claim:type-holes}}i^{-2}2^{-i/d}}\\
&\le i_\star2^n
(1-p)^{n    c_{\ref{claim:type-holes}}2^{-1/d}\lfloor\beta\rfloor R_2^{-1/(d-1)}((\log_2R_2)d/(d-1)+1)^{-2}},
\end{aligned}\nonumber
\end{equation}
where for the last inequality we used that for all $i\le  i_\star-1=\lceil(\log_2R_2)d/(d-1)\rceil$, we have that $2^{-i/d}i^{-2}\ge 2^{-1/d}R_2^{-1/(d-1)}/(1+\log_2R_2)^2$. We recall that  we may assume $R_2(\beta)=\lfloor \beta/\sqrt{d}\rfloor>1$ by the reasoning at the beginning of the proof. So, the exponent of $\beta$ in the numerator becomes $1-1/(d-1)=(d-2)/(d-1)$ in \eqref{eq:small-prominent-last}. The prefactor $i_\star$ is a ($\beta$-dependent) constant by~\eqref{eq:r-holes}. Therefore, the statement in \eqref{eq:small-prominent-last} follows by adapting the constant in the exponent.
\end{proof}

\subsection*{Excluding large principal holes}
We turn to the sum in \eqref{eq:event-2-splitted-2}.
\begin{claim}[Sets with large principal hole-types are unlikely components]\label{stat:large-holes}
Let $\CG_n$ be long-range percolation on $\Lambda_n$ as in Definition \ref{def:lrp} with $d\ge 2$, $\alpha> 1$, and $i_\star(\beta)$ from \eqref{eq:r-holes}. 
 There exists $c_{\ref{stat:large-holes}}=c_{\ref{stat:large-holes}}(d)>0$ such that for all $n$ sufficiently large 
\begin{equation}\label{eq:large-holes-goal}
\begin{aligned}
\mathrm{Err}_{\mathrm{large}}&:=\sum_{i=i_\star}^{\lceil\log_2 n\rceil}\Prob\Big(\exists A\in\CA_\mathrm{large}(i):\, A\not\sim_{\CG_n} \cup_{H\in\mathfrak{H}_A} \partial_\mathrm{int}H\Big)\\
&\le \big(1-p(1\wedge\beta)^{d\alpha}\big)^{\lfloor1\vee(\beta/\sqrt d)\rfloor \cdot c_{\ref{stat:large-holes}}(\log n)^{-2}n^{(d-1)/d}}.
\end{aligned}
\end{equation}
whenever $p$ and $\beta$ guarantee that the base of $c_{\ref{stat:large-holes}}(\log n)^{-2}n^{(d-1)/d}$ is sufficiently small..
\end{claim}
\begin{proof}
Similarly to the small principal hole-types, we will find enough potential edges that all must be absent in order for the events on the left-hand-side in \eqref{eq:large-holes-goal} to occur. 
Fix an ordering $L$ of vertices in $\Lambda_n$ so that $x_1<_Lx_2<_L\cdots<_Lx_n$ with respect to this ordering (e.g., the lexicographic ordering). For a block $A\in\CA_\mathrm{large}(i)$ (which has at least $\lceil h_n(i)\rceil$ holes of type $i$ by \eqref{eq-def:hni}), we order its holes $\mathfrak{H}_A$ in such a way that the holes of type $i$ are $H_A^\sss{(1)},\ldots, H_A^\sss{|\mathfrak{H}_A(i)|}$,  and that for all $r<s\le |\mathfrak{H}_A(i)|$ the vertices smallest in the ordering within $H_A^\sss{(r)}$ and $H_A^\sss{(s)}$ -- say $x_r\in H_A^\sss{(r)}$ and $x_s\in H_A^\sss{(s)}$ -- satisfy $x_r<_Lx_s$. 
We obtain an upper bound when we exclude edges from $A$ towards only its first $\lceil h_n(i)\rceil$ holes of type $i$:
\begin{equation}
\begin{aligned}
\Prob  \big(\exists A\in\CA_\mathrm{large}(i)&:  A\not\sim_{\CG_n} \cup_{H\in\mathfrak{H}_A} \partial_\mathrm{int}H\big)\\
&\le 
\Prob  \big(\exists A\in\CA_\mathrm{large}(i): A\not\sim_{\CG_n} \cup_{j\le \lceil h_{n}(i)\rceil} \partial_\mathrm{int}H_A^\sss{(j)}\big)\\
&\le \Prob  \Big(\exists A\in\CA_\mathrm{large}(i): \cup_{j\le \lceil h_{n}(i)\rceil}\partial_\mathrm{ext}H_A^\sss{(j)}\not\sim_{\CG_n} \cup_{j\le \lceil h_{n}(i)\rceil} \partial_\mathrm{int}H_A^\sss{(j)}\Big),\label{eq:A-to-A'-2}
\end{aligned}
\end{equation}
where to get the last row we only look at edges emanating from $A$ that are on the exterior boundaries of the holes. This is an upper bound since $\partial_\mathrm{ext}H_A^\sss{(j)}\subseteq A$ by \eqref{eq:boundary-in-box-identity} for all $j\le \lceil h_n(i)\rceil$.
If for two blocks $A, A'\in\CA_\mathrm{large}(i)$, the first $\lceil h_n(i)\rceil$ holes coincide, also the exterior boundaries of these first $\lceil h_n(i)\rceil$ holes coincide, and the event in \eqref{eq:A-to-A'-2} excludes the exact same edges.
So, a simple union bound over $A$ in \eqref{eq:A-to-A'-2} would overcount the non-presence of those edges too many times. Instead, we carry out a union bound over all possible lists of the first $\lceil h_n(i)\rceil$ holes. To this end, we consider for all $A\in \CA_\mathrm{large}(i)$ the following:
\begin{equation}\label{eq:DA}
 D_i(A):=\Lambda_n\setminus\{ \cup_{j\le \lceil h_n(i)\rceil}H_{A}^{\sss{(j)}}\}, \qquad
\CD(i):=\{ D:   \exists A\in \CA_{\mathrm{large}}(i), D=D_i(A)\}.
 \end{equation}
In words, the set $D_i(A)$ has exactly $\lceil h_n(i)\rceil$ many holes, all of which are type $i$, and its holes coincide with the first $\lceil h_n(i)\rceil$ type-$i$ holes of $A$. $\CD(i)$ collects all sets $D$ that arise in this way for some $A \in \CA_{\mathrm{large}}(i)$.
Since $D_i(A)$ shares the first $\lceil h_n(i)\rceil$ type-$i$ holes with $A$, also the exterior boundaries of those holes agree between $A$ and $D_i(A)$. So,
\begin{equation} 
\begin{aligned}
  \big\{\cup_{j\le \lceil h_{n}(i)\rceil}\partial_\mathrm{ext}H_{D_i(A)}^\sss{(j)}\not\sim_{\CG_n}& \cup_{j\le \lceil h_{n}(i)\rceil} \partial_\mathrm{int}H_{D_i(A)}^\sss{(j)}\big\}\\
  &= \big\{\cup_{j\le \lceil h_{n}(i)\rceil}\partial_\mathrm{ext}H_A^\sss{(j)}\not\sim_{\CG_n} \cup_{j\le \lceil h_{n}(i)\rceil} \partial_\mathrm{int}H_A^\sss{(j)}\big\}.\nonumber%
  \end{aligned}
\end{equation}
Hence, in \eqref{eq:A-to-A'-2} we can group the blocks in $\CA_\mathrm{large}(i)$ that all map to the same $D\in \CD(i)$, and  obtain that
\begin{align}
\Prob  \Big(\exists A\in\CA_\mathrm{large}(i)&: \cup_{j\le \lceil h_{n}(i)\rceil}\partial_\mathrm{ext}H_A^\sss{(j)}\not\sim_{\CG_n} \cup_{j\le \lceil h_{n}(i)\rceil} \partial_\mathrm{int}H_A^\sss{(j)}\Big), \nonumber\\
&=  \Prob  \Big(\exists D\in \CD(i): \cup_{j\le \lceil h_{n}(i)\rceil}\partial_\mathrm{ext}H_{D}^\sss{(j)}\not\sim_{\CG_n}  \cup_{j\le \lceil h_{n}(i)\rceil} \partial_\mathrm{int}H_{D}^\sss{(j)}\Big)\nonumber\\
&\le \sum_{D\in\CD(i)}\Prob\Big(\cup_{j\le \lceil h_{n}(i)\rceil}\partial_\mathrm{ext}H_{D}^\sss{(j)}\not\sim_{\CG_n} \cup_{j\le \lceil h_{n}(i)\rceil} \partial_\mathrm{int}H_{D}^\sss{(j)}\Big). \label{eq:aprimelarge}
\end{align}
 We combine the following three observations to bound a single summand in the last row. First,
each vertex $x\in \partial_\mathrm{int}H_{D}^\sss{(j)}$ is at distance one from at least one vertex $y_x\in \partial_\mathrm{ext}H_{D}^\sss{(j)}\subseteq D\subseteq \Lambda_n$ (by definition, a hole $H_{D}^\sss{(j)}$ does not intersect $\widetilde \partial_{\mathrm{int}}\Lambda_n$). Second, for all $j\le \lceil h_n(i)\rceil$, $|\partial_\mathrm{ext}H_{D}^\sss{(j)}|\ge R_2$ by \eqref{eq:min-size-hole-boundary} and the fact that $i\ge i_\star$. Third, the exterior boundary of a hole $\partial_\mathrm{ext}H_{D}^\sss{(j)}$ is $\ast$-connected by Claim \ref{claim:containment}(v). 

Recall $R_2=\lfloor \beta/\sqrt{d}\rfloor \vee 1\ge 1$ from \eqref{eq:r-holes}.
Hence, for each vertex $x\in \partial_\mathrm{int}H_{D}^\sss{(j)}$, starting from $y_x\in \partial_\mathrm{ext}H_{D}^\sss{(j)} $, one can find a $\ast$-connected set of vertices $\CB_x\subseteq \partial_\mathrm{ext}H_{D}^\sss{(j)} $ that satisfies
\[ |\CB_x| \ge R_2, \quad\mbox{and}\quad\forall z\in \CB_x:\  \|x-z\|_2\le \beta\vee 1. \]
Here we used that $\sqrt{d}$ is the maximal $\|\cdot\|_2$-distance between $\ast$-adjacent vertices, which gives $\ell_2$-distance at most $\beta$ when $\beta\ge 1$. When $\beta<1$, $R_2=1$ and the $\ell_2$-distance between $y_x$ and $x$ is $1$ by Definition \ref{def:boundary}.
Then, the edges $\{ \{x,z\}: x\in \cup_{j\le \lceil h_{n}(i)\rceil}\partial_\mathrm{int}H_{D}^\sss{(j)}, z\in \CB_x \}$  all need to be absent for 
the event in \eqref{eq:aprimelarge} to occur.  When $\beta\ge 1$, the distance bound $\|x-z\|\le \beta$ ensures that all these edges are present with probability $p$ by \eqref{eq:connection-prob-gen}. When $\beta<1$, the edge $\{x,y_x\}$ is not present with probability $1-p\beta^{d\alpha}$.
Combining the two cases with   \eqref{eq:A-to-A'-2} and \eqref{eq:aprimelarge}, it follows by the independence of the edges in $\CG_n$ that
\begin{equation}\label{eq:both-betas}
\begin{aligned}
\Prob  \big(\exists A\in\CA_\mathrm{large}(i): A\not\sim_{\CG_n} &\cup_{H\in\mathfrak{H}_A}\partial_\mathrm{int}H\big)
\\
&\le \sum_{D\in\CD(i)}\prod_{j\le \lceil h_n(i)\rceil}\big(1-p(1\wedge\beta)^{d\alpha}\big)^{|\partial_\mathrm{int} H_{D}^\sss{(j)}|\cdot\lfloor1\vee(\beta/\sqrt d)\rfloor}. 
\end{aligned}
\end{equation}
We  now encode the holes of $D\in\CD(i)$  similar to the encoding of the blocks in Section \ref{sec:spanning}.
We write $\mathbf{x}_{D}:=(x_1,\ldots, x_{\lceil h_n(i)\rceil})$ for the vertices with the smallest label in the $L$-ordering within the respective holes
 $H^{\sss{(1)}}_{D},\ldots, H^{\sss{\lceil h_n(i)\rceil}}_{D}$. 
  Let us write $\Lambda_{n,<}^{\lceil h_n(i)\rceil}$ for the vectors $\mathbf{x}\in\Lambda_n^{\lceil h_n(i)\rceil}$ with $x_r<_L x_s$ for all $r<s$. By the initial ordering of the holes above \eqref{eq:A-to-A'-2}, $\mathbf{x}_{D}\in \Lambda_{n,<}^{\lceil h_n(i)\rceil}$. 
  
 Let then $m_j:=|\partial_\mathrm{int}H_D^{\sss{(j)}}|$ for all $j\le \lceil h_n(i)\rceil$, and  write $\mathbf{m}_D:=(m_1, \dots, m_{\lceil h_n(i)\rceil})$. 
Define then for all $\mathbf{x}\in\Lambda_{n,<}^{\lceil h_n(i)\rceil}$, and $\mathbf{m}\in\N^{\lceil h_n(i)\rceil}$:
 \[ \CD(i, \mathbf{x},\mathbf{m}):=\{D\in \CD(i): \mathbf{x}_{D}=\mathbf{x}, \mathbf{m}_{D}=\mathbf{m} \}. \]
  The set $D\in \CD(i)$ has the first $\lceil h_n(i)\rceil$ holes of some $A\in \CA_{\mathrm{large}}(i)$ where $D(A)=D$, and for that $A$, hole-type $i$ was principal in terms of Definition \ref{def:holes} and \eqref{eq:alarge-i}. Definition \ref{def:holes} readily implies that the inequality \eqref{eq-claim:type-holes-2} in Claim \ref{claim:type-holes} holds for the first $\lceil h_n(i)\rceil$ many holes of $A$, and in turn of $D$ with $i_A$ replaced by $i$ in \eqref{eq-claim:type-holes-2}.  Hence,  the total interior boundary size $m:=\sum_{j=1}^{\lceil h_n(i)\rceil} m_j$ satisfies that  $m \ge c_{\ref{claim:type-holes}}i^{-2}2^{-i/d}n$. So for $m\ge  c_{\ref{claim:type-holes}}i^{-2}2^{-i/d}n$ we introduce the possible boundary-length vectors with total size $m$:
  \begin{equation}\label{eq:sum-m}
 \CM_i(m):=\big\{\mathbf{m}\in\N^{\lceil h_n(i)\rceil}:  m_1+\ldots+m_{\lceil h_n(i)\rceil}=m\big\}.
  \end{equation}
 Returning to \eqref{eq:aprimelarge}, we decompose the summation on the right-hand side as follows:
\begin{align}
\Prob  &\big(\exists A\in\CA_\mathrm{large}(i): A\not\sim_{\CG_n} \cup_{H\in\mathfrak{H}_A}\partial_\mathrm{int}H\big)\nonumber\\
&\le \!\!\!\!
\sum_{m\ge c_{\ref{claim:type-holes}}i^{-2}2^{-i/d}n}\sum_{\mathbf{m}\in\CM_i(m)}
\sum_{\mathbf{x}\in\Lambda_{n,<}^{\lceil h_n(i)\rceil}}
\sum_{D\in\CD(i,\mathbf{x}, \mathbf{m})}\prod_{j\le \lceil h_n(i)\rceil}\hspace{-5pt}\big(1-p(1\wedge\beta)^{d\alpha}\big)^{|\partial_\mathrm{int} H_{D}^\sss{(j)}|\cdot \lfloor1\vee(\beta/\sqrt d)\rfloor} \nonumber\\
&=
\sum_{m\ge c_{\ref{claim:type-holes}}i^{-2}2^{-i/d}n}\big(1-p(1\wedge\beta)^{d\alpha}\big)^{m\lfloor1\vee(\beta/\sqrt d)\rfloor} 
\sum_{\mathbf{m}\in\CM_i(m)}\sum_{\mathbf{x}\in\Lambda_n^{\lceil h_n(i)\rceil,<}}
\sum_{D\in\CD(i,\mathbf{x}, \mathbf{m})}1
.\label{eq:di-counting-1}
\end{align}
Now we evaluate the number of terms of the last three summations in~\eqref{eq:di-counting-1}. 
Each block $D\in\CD(i, \mathbf{x}, \mathbf{m})$ is uniquely characterized by its $\lceil h_n(i)\rceil$ holes by \eqref{eq:DA} (since these are the only holes of $D$). Having fixed the vectors $\mathbf{x}$ and $\mathbf{m}$, we apply  Lemma \ref{lemma:peierl} ---Peierls' argument--- to each hole $H_{D}^{\sss{(j)}}$ with $H_{D}^{\sss{(j)}}\ni x_j$ and $|\partial_{\mathrm{int}}H_{D}^{\sss{(j)}}|=m_j$ to count the size of $\CD(i, \mathbf{x}, \mathbf{m})$. 
The lemma can be applied since for each hole $H$, we have $H=\bar H$ by Claim \ref{claim:containment}(v).
Hence,  there are at most $\exp(c_\mathrm{pei}m_j)$ possible holes of interior boundary size $m_j$ containing $x_j$. So, for all $\mathbf{x}\in\Lambda_n^{\lceil h_n(i)\rceil}$ and $\mathbf{m}\in \CM_i(m)$, 
\[ 
|\CD(i,\mathbf{x}, \mathbf{m})| \le \prod_{j\le \lceil h_n(i)\rceil }\exp\big(c_\mathrm{pei}m_j\big)=\exp(c_\mathrm{pei}m).
\]
Moreover, by~\eqref{eq:sum-m}, $m=m_1+\ldots+m_{\lceil h_n(i)\rceil}\ge \lceil h_n(i)\rceil$, and so $|\CM_i(m)| \le \binom{m+\lceil h_n(i)\rceil}{m}\le 2^{2m}\le \re^{2m}$.  Next, since the vertices in $\mathbf x$ are ordered, there are at most 
$\binom{n}{\lceil h_n(i)\rceil}$ many choices for the vector $\mathbf x\in \Lambda_n^{\lceil h_n(i)\rceil, <}$. 
Using these bounds in~\eqref{eq:di-counting-1}, and that $R_2=\lfloor \beta/\sqrt{d}\rfloor\vee 1$, we obtain
\begin{equation}\nonumber
\begin{aligned}
\Prob  \big(\exists A\in\CA_\mathrm{large}(i)&: A\not\sim_{\CG_n} \cup_{H\in\mathfrak{H}_A}\partial_\mathrm{int}H\big)\\
&\le
\binom{n}{\lceil h_n(i)\rceil}\sum_{m\ge c_{\ref{claim:type-holes}}i^{-2}2^{-i/d}n}\bigg(\re^{2+c_{\mathrm{pei}}}\cdot\big(1-p(1\wedge\beta)^{d\alpha}\big)^{\lfloor1\vee(\beta/\sqrt d)\rfloor}\bigg)^m,
\end{aligned}
\end{equation}
The constant $c_\mathrm{pei}$ depends only on $d$. In what follows, we assume that $p$ and $\beta$ are such that the first factor $\re^{(c_\mathrm{pei}+2)}$ is at most the second factor to the power $-1/2$ (equivalently, the second factor is at most $\re^{-2(c_\mathrm{pei}+2)}$).
Then, the summands decay in $m$ and the sum is dominated by its first term. This gives for some $C>0$ and $n$ sufficiently large that 
\begin{align}
\Prob  \big(\exists A\in\CA_\mathrm{large}(i)&: A\not\sim_{\CG_n} \cup_{H\in\mathfrak{H}_A}\partial_\mathrm{int}H\big)\nonumber\\
&\le C\binom{n}{\lceil h_n(i)\rceil}\big(1-p(1\wedge\beta)^{d\alpha}\big)^{\lfloor1\vee(\beta/\sqrt d)\rfloor c_{\ref{claim:type-holes}}i^{-2}2^{-i/d-1}n}\label{eq:aprime-two-factors}.
\end{align}
 We now bound  the binomial coefficient as  $\binom{n}{h}\le n^h/h!$, and then $\re^h =\sum_{i=0}^\infty h^i/i! \ge h^h/h!$ which implies $h!\ge(h/\re)^h$. This estimate substituted into the bound on the binomial coefficient gives $\binom{n}{h}\le (\re \cdot n/h)^h$. Using that $h=\lceil h_n(i)\rceil =\lceil 2^{-i-3}i^{-2}n\rceil$ in \eqref{eq-def:hni} gives
 \begin{align}\label{eq:binom-coeff}
 \binom{n}{\lceil h_n(i)\rceil}\le (\re2^{i+3}i^2)^{i^{-2}2^{-i-3}n+1}
 \le
  \exp\Big((i+4+2\log i)\big(i^{-2}2^{-i-3}n+1\big)\Big),
\end{align}
where we also used that $2<\re$ to obtain the right-hand-side. Using this bound in the right-hand side of \eqref{eq:aprime-two-factors}, we may compare the exponents.
Let $i_\circ=i_\circ(d)$ be the smallest $i\in\N$ such that for \emph{all} $n\ge 1$ and all $i\ge i_\circ$,
\[ (i+4+2\log i)(i^{-2}2^{-i-3}n+1)<c_{\ref{claim:type-holes}}i^{-2}2^{-i/d}n.\]
Then for $i\ge i_\circ$, 
\begin{align}
 \Prob \big(\exists A\in\CA_\mathrm{large}(i)&: A\not\sim_{\CG_n} \cup_{H\in\mathfrak{H}_A}\partial_\mathrm{int}H\big)\nonumber\\
 & \le C \Big(\re \cdot \big(1-p(1\wedge\beta)^{d\alpha}\big)^{\lfloor1\vee(\beta/\sqrt d)\rfloor}\Big)^{c_{\ref{claim:type-holes}}i^{-2}2^{-i/d-1}n}\nonumber \\
 &\le C\big(1-p(1\wedge\beta)^{d\alpha}\big)^{\lfloor1\vee(\beta/\sqrt d)\rfloor c_{\ref{claim:type-holes}}i^{-2}2^{-i/d-2}n},\label{eq:111}
 \end{align}
 where the last bound follows from the assumption on the second factor between brackets in the second line being smaller than $\re^{-2(c_\mathrm{pei}+2)}$ above~\eqref{eq:aprime-two-factors}.

Now we treat the case $i<i_\circ$. Using that $i_\circ$ is a constant that only depends on $d$, (comparing the coefficients of $c_{\ref{claim:type-holes}}n$ in  \eqref{eq:aprime-two-factors} to \eqref{eq:binom-coeff}) we require that $p,\beta$ are such that for all $i< i_\circ$ and all $n\ge 1$,
\begin{equation}
\exp\big((i+4+2\log i)(i^{-2}2^{-i-3}n+1)\big)\le \big(1-p(1\wedge\beta)^{d\alpha}\big)^{-\lfloor1\vee(\beta/\sqrt d)\rfloor c_{\ref{claim:type-holes}}i^{-2}2^{-i/d-3}n}.
\end{equation}
In this case we obtain that for all $i< i_\circ$
\begin{equation}\label{eq:222}
\begin{aligned}
\Prob  \big(\exists A\in\CA_\mathrm{large}(i)&: A\not\sim_{\CG_n} \cup_{H\in\mathfrak{H}_A}\partial_\mathrm{int}H\big) \\
&\le C\big(1-p(1\wedge\beta)^{d\alpha}\big)^{\lfloor1\vee(\beta/\sqrt d)\rfloor c_{\ref{claim:type-holes}}i^{-2}2^{-i/d-3}n}. 
\end{aligned}
\end{equation}
We  combine \eqref{eq:111} and \eqref{eq:222} and obtain for all $i\ge 1$ that
\[
 \Prob\big(\exists A\!\in\!\CA_\mathrm{large}(i)\!:\! \partial_\mathrm{int}A\not\sim_{\CG_n} \cup_{H\in\mathfrak{H}_A(i)} \partial_\mathrm{int}H\big)
 \le C\big(1-p(1\wedge\beta)^{d\alpha}\big)^{\lfloor1\vee(\beta/\sqrt d)\rfloor c_{\ref{claim:type-holes}}i^{-2}2^{-i/d-3}n}.
\]
We apply this bound to all summands in \eqref{eq:large-holes-goal}, use that the terms of the upper bound are increasing in $i$, that  the last term is $i=\lceil \log_2 n\rceil$ (so $2^{-i/d}i^{-2}n\ge 2^{-1/d}(1+\log_2 n)^{-2}n^{(d-1)/d}$ for all $i$), and obtain for some $c=c(d)>0$ and $n$ sufficiently large
\begin{align*}
 \sum_{i=i_\star+1}^{\lceil\log_2 n\rceil}
 \Prob\big(\exists A\in\CA_\mathrm{large}(i) & : \partial_\mathrm{int}A\not\sim_{\CG_n} \cup_{H\in\mathfrak{H}_A(i)} \partial_\mathrm{int}H\big) \\
           & \le
 \lceil\log_2 n\rceil\big(1-p(1\wedge\beta)^{d\alpha}\big)^{\lfloor1\vee(\beta/\sqrt d)\rfloor c(\log n)^{-2}n^{(d-1)/d}}.
 \end{align*}
 The logarithmic prefactor is of smaller order as $n\to\infty$.
 This finishes the proof of Claim \ref{stat:large-holes} by adapting $c$ to some slightly smaller $c_{\ref{stat:large-holes}}=c_{\ref{stat:large-holes}}(d)$.\qedhere
\end{proof}

We are ready to give the final proofs of the section.
\begin{proof}[Proof of Lemma \ref{lemma:holes}]
We recall the bound \eqref{eq:event-2-splitted-1} -- \eqref{eq:event-2-splitted-2} on the probability of the event $\CE_2(\CG_n)=\{\exists A\in\CA_\mathrm{large}: A\not\sim_{\CG_n}(\cup_{H\in\mathfrak{H}_A} \partial_\mathrm{int}H)\}$, splitting it into two sums: one sum \eqref{eq:event-2-splitted-1} for small principal holes, and one sum \eqref{eq:event-2-splitted-2} for large principal holes. First assume that $\beta<2\sqrt{d}$. Then the error term from Claim~\ref{stat:small-holes} is $0$ regardless of the exact values of $p$ and $\beta$. Recall $f(p,\beta)=1-p(1\wedge\beta)^{d\alpha}$ from~\eqref{eq:fpbeta}, which we assumed to be sufficiently small in the statement of Lemma~\ref{lemma:holes}. To apply Claim~\ref{stat:large-holes}, we must assume $\big(1-p(1\wedge\beta)^{d\alpha}\big)^{\lfloor1\vee(\beta/\sqrt{d})\rfloor}$ to be sufficiently small, which follows if $f(p, \beta)$ is small. Thus, when $\beta<2\sqrt{d}$ we obtain by Claims~\ref{stat:small-holes} and~\ref{stat:large-holes},
\[
\Prob\big(\CE_2(\CG_n)\big)\le\big(1-p(1\wedge\beta)^{d\alpha}\big)^{\lfloor1\vee(\beta/\sqrt d)\rfloor \cdot c_{\ref{stat:large-holes}}(\log n)^{-2}n^{(d-1)/d}}.
\]
When $\beta\ge 2\sqrt{d}$, we obtain an extra error term from Claim~\ref{stat:small-holes}. Now, our assumption on $p$ and $\beta$ for $\beta \ge 2\sqrt{d}$ in Lemma~\ref{lemma:holes} is that $f(p,\beta)=(1-p)^{(\log_2\beta)^{-2}\beta^{(d-2)/(d-1)}}$ is sufficiently small. This criterion directly guarantees the criterion in Claim~\ref{stat:small-holes}, and also implies the necessary criterion that $(1-p)^{\lfloor \beta/\sqrt{d}\rfloor}$ is small when $\beta\ge2\sqrt{d}$ in Claim~\ref{stat:large-holes}. Thus, the bounds in Claims~\ref{stat:small-holes} and~\ref{stat:large-holes} hold, and
\[
\begin{aligned}
\Prob\big(\CE_2(\CG_n)\big)\le \big(2&\cdot(1-p)^{c_{\ref{stat:small-holes}}(1+\log_2 \beta)^{-2}\beta^{(d-2)/(d-1)}}\big)^n \\
&\hspace{15pt}+
\big(1-p\big)^{\lfloor1\vee(\beta/\sqrt d)\rfloor \cdot c_{\ref{stat:large-holes}}(\log n)^{-2}n^{(d-1)/d}}.
\end{aligned}
\]
The constant $c_{\ref{stat:small-holes}}$ depends only on $d$. Thus, by ensuring that $f(p,\beta)$ is sufficiently small, the first term can be made smaller than $2^{-n}$. For such $p$ and $\beta$, the second term dominates the right-hand side. This finishes the proof. 

\end{proof}

\begin{proof}[Proof of Proposition \ref{prop:second-largest}]
The proof is immediate from Claim \ref{claim:second-events} and Lemmas \ref{lemma:spanning-tree} and \ref{lemma:holes}. The condition $n(\log n)^{-2d/(d-1)}\ge k$ implies that $k^{(d-1)/d}\le n^{(d-1)/d}/(\log n)^2$, so that the error bound from \eqref{eq:unlikely-block-graphs} dominates the error bound from \eqref{eq:no-large-holes}. The coefficient of $-k^{(d-1)/d}$ in the exponent can be made at least $1$ by choosing either $\beta$ sufficiently large or $p(1\wedge\beta)^{d \alpha}$ sufficiently close to $1$.
\end{proof}
\section{Proof of Theorem \ref{thm:longrange}}
We will verify the statements in Theorem \ref{thm:longrange}, based on Proposition \ref{prop:second-largest}.

\subsection*{Upper bounds}
The upper bound on the second-largest component in \eqref{eq:main-second-largest} follows immediately from Proposition \ref{prop:second-largest}, by substituting $k=A(\log n)^{d/(d-1)}$ for some large constant $A=A(\delta)$. For the upper bound on the cluster-size decay in~\eqref{eq:main-cluster-decay} and the lower bounds in~(\ref{eq:main-second-largest}--\ref{eq:main-cluster-decay}), we cite two statements from our paper \cite{clusterI} which considers a more general class of percolation models.
The paper \cite{clusterI}  considers models where vertices have associated vertex marks, and the connection probability \eqref{eq:connection-prob-gen} contains an additional factor to $\beta$ in the numerator on the right-hand side of \eqref{eq:connection-prob-gen} that depends on these vertex-marks.  The model long-range percolation in Definition \ref{def:lrp} hence forms a subclass of the model in  \cite{clusterI} in which the vertex set is $\Z^d$ and all the vertex marks are identical to $1$. Due to this, conditions that regard vertices with high vertex-marks in \cite{clusterI} are automatically satisfied for the long-range percolation model in Definition \ref{def:lrp}.
We rephrase the results of \cite{clusterI} to the setting of long-range percolation of Definition \ref{def:lrp} by setting all vertex marks identical to $1$ in \cite{clusterI}.
\begin{proposition}[Prerequisites for the upper bound {\cite[Proposition 6.1]{clusterI}}]\label{prop:prerequisites-upper}
 Consider supercritical long-range percolation with parameters $\alpha>1$, and $d\in\N$.
 Assume that there exist $\zeta, c, c'>0$ and a function $g(k)=O(k^{1+c'})$ such that for all $n$, $k$ sufficiently large, whenever $n\ge g(k)$, it holds that
 \begin{align}
  \Prob\big(|\CC_{n}^\sss{(2)}| \ge k\big)
  &\le
  n^{c'}\exp\big(-c\, k^{\zeta}\big),\label{eq:cond-1}\\
  \Prob\big(|\CC_{n}^\sss{(1)}|\le n^{c}\big)                  & \le  n^{-1-c}.\label{eq:prop-outgiant}
 \end{align}
 Then there exists a constant $A>0$ such that for all $n, k$, sufficiently large such that $g(k)\le n\le\infty$,
 \begin{equation}\nonumber%\label{eq:exp-decay-71}
  \Prob\big(|\CC_n(0)|\ge k, 0\notin\CC_n^\sss{(1)}\big)\le \exp\big(-(1/A)k^{\zeta}\big),
 \end{equation}
 and 
 \begin{equation}
  \frac{|\CC_n^\sss{(1)}|}{n}\overset\Prob\longrightarrow \Prob\big(|\CC(0)|=\infty\big),\qquad \mbox{as $n\to\infty$.}\nonumber
 \end{equation}
\end{proposition}
We have just proved the prerequisite \eqref{eq:cond-1} for our case in Proposition \ref{prop:second-largest} with $\zeta=(d-1)/d$, $c=1$ and $c'=2$. The other prerequisite  \eqref{eq:prop-outgiant} is a consequence of the following lemma (\eqref{eq:giant-ldp} below in particular). The second statement of the lemma, \eqref{eq:origin-giant} will be needed for the lower bound of Theorem \ref{thm:longrange} shortly. 
\begin{lemma}\label{lem:Coupling}
 Consider long-range percolation in Definition \ref{def:lrp} with $\alpha>1$, and $d\ge 2$. For all $p\in(0,1)$, there exists $\beta_{\ref{lem:Coupling}}=\beta_{\ref{lem:Coupling}}(p,d,\alpha)>0$ such that for all $\beta\ge \beta_{\ref{lem:Coupling}}$ there exists $\rho>0$ such that for $n$ sufficiently large,
 \begin{align}
  \Prob\big(|\CC_n^\sss{(1)}|\ge \rho n\big)&\ge 1-\exp\big(-\rho n^{(d-1)/d}\big),\label{eq:giant-ldp}\\
\Prob\big(|\CC_n(0)|\ge \rho n\big)&\ge \rho.\label{eq:origin-giant}
 \end{align}
  The same bounds hold whenever $p(1\wedge\beta)^{d\alpha}$ is sufficiently close to 1.
\end{lemma}
\begin{proof}
 We start with showing \eqref{eq:giant-ldp}.
 If $p(1\wedge\beta)^{d\alpha}$ is sufficiently close to $1$, then we can do the following.  Taking $G_n$ as a realization of LRP in $\Lambda_n$, and retaining only the edges of $G_n$ that are between nearest-neighbor vertices in $\Z^d_1$, we obtain the classical iid nearest-neighbor Bernoulli percolation in $\Lambda_n$. Denote this graph by $G_n^{\sss{(\mathrm{nn})}}$ and its largest component by $\CC_n^{\sss{(1)}}(G_n^{\sss{(\mathrm{nn})}})$. 
 Then since $\CE(G_n^{\sss{(\mathrm{nn})}})\subseteq \CE(G_n)$, for the sizes of the largest components it holds that $|\CC_n^{\sss{(1)}}|\ge |\CC_n^{\sss{(1)}}(G_n^{\sss{(\mathrm{nn})}})|$. 
Since $p(1\wedge\beta)^{d\alpha}$ is sufficiently close to $1$,  the surface-order large devation result of \cite[Theorem 1.1]{deuschel1996surface} applies to $\CC_n^{\sss{(1)}}(G_n^{\sss{(nn)}})$, and \eqref{eq:giant-ldp} immediately follows. 

Assume now that $p(1\wedge\beta)^{d\alpha}$ is not sufficiently close to $1$ for the nearest-neighbor subgraph to ensure the required result. 
Let 
\begin{equation}\label{eq:box-number}
m(n):=\left\lceil \frac{n^{1/d}}{\beta/(2\sqrt{d})}\right\rceil^d.
\end{equation}
 Partition $\Lambda_n$ into $m(n)$ identical boxes $Q_1, Q_2, Q_{m(n)}$, each having sidelength  $r(n)=n^{1/d}/m(n)^{1/d}$ $\le \beta/(2\sqrt{d})$. 
 Denote $n_i:=|\Z^d\cap Q_i|$ the number of vertices in box $Q_i$.
  Then, for $n$ large enough, for all $i\le m(n)$ it holds that  $n_i\in [C(\beta), 4C(\beta)]$, where for all $n$ sufficiently large
 \begin{equation}\label{eq:C-beta}
 C(\beta):=\mathrm{Vol}(Q_1)/2= r(n)^d/2\in \left[\frac{\beta^d}{4(2\sqrt{d})^d},\frac{\beta^d}{2(2\sqrt{d})^d}\right].\end{equation} 
 Let us say that $Q_i, Q_j$ are adjacent boxes if they share a $(d-1)$-dimensional face. 
 Since the diameter of each box is at most $\beta/2$, by \eqref{eq:connection-prob-gen}, if $Q_i, Q_j$ are adjacent boxes, 
\begin{equation}\label{eq:across-boxes}
    \Prob\left(\{x,y\}\in \CE(\CG_n)\right)=p \qquad \mbox{for all }  x\in Q_i\cap \Z^d, y\in Q_j\cap \Z^d,
\end{equation}
 independently of other edges. The same is true when both $x,y\in Q_i\cap \Z^d$.
 
 Let us denote the subgraph of $\CG_n$ induced by the vertices in the box $Q_i$ by $\CG_n(Q_i)$.  $\CG_n(Q_i)$ stochastically dominates an Erd\H{o}s--R\'enyi random graph with $n_i\in [C(\beta),4C(\beta)]$ vertices and edge probability $p$. The probability that the graph diameter of $\CG_n(Q_i)$ is at most 2, is by a union bound at most $\binom{n_i}2(1-p)^{n_i-2}\to0$ as $n_i\to\infty$. If the graph diameter is at most $2$, then $\CG_n(Q_i)$ must be connected. Thus, for any fixed $\varepsilon>0$,  by choosing $\beta$ (and hence also $C(\beta)$) large enough depending on $\varepsilon$, the probability that $\CG_n(Q_i)$ is connected is at least $1-\varepsilon$. Further, the graphs $(\CG_n(Q_i))_{i\le m(n)}$ are \emph{independent} since they are induced subgraphs of long-range percolation on vertices in disjoint boxes, and edges are present independently in $\CG_n$ by Definition \ref{def:lrp}.
 
 We define a deterministic auxiliary graph $G$. Every box $Q_i$ corresponds to a vertex $v_i$, for each $i\le m(n)$, and two vertices $v_i,v_j$ in $G$ are adjacent if the corresponding boxes $Q_i, Q_j$ are adjacent, i.e., they share a $(d-1)$-dimensional face. (Similarly, one can define the $1$-distance between any two vertices $v_i, v_j$ by the length of the shortest path between $v_i, v_j$ via adjacent vertices.)
Hence, the vertices of $G$ then form a box $\widetilde \Lambda_{m(n)}$ of volume $m(n)$ of $\Z^d_1$. This we call the re-normalized lattice.

 Now we define a random subgraph $H$ of $G$.
 We declare a vertex $v_i$ of $G$
 \emph{active} when $\CG_n(Q_i)$ is connected. Edges of $H$ will be only present between active and adjacent vertices in $G$. Assuming that two vertices $v_i$, $v_j \in H$ are adjacent in $G$ and both active, we declare the edge between $v_1$ and $v_2$ \emph{open}, equivalently, present in $\CE(H)$, if there exist vertices $x \in Q_i\cap \Z^d, y \in Q_j\cap \Z^d$ with the edge $\{u, v\}\in \CE(\CG_n)$.   Conditional on $v_i, v_j$ being active, by \eqref{eq:across-boxes}, the nearest-neighbor edge between $v_1$ and $v_2$ is open with probability  $1-(1-p)^{n_i n_j}\ge 1-(1-p)^{C(\beta)^2}\ge 1-\varepsilon$, where the last inequality holds for arbitrarily small $\varepsilon > 0$ by making $C=C(\beta)$ in \eqref{eq:C-beta} large enough. Different edges of $\CE(G)$ are present conditionally independently in $\CE(H)$ given that the end-vertices are active.
 
Finally, let $H$ be the induced graph obtained from $H$ on active vertices and open edges $\CE(H)$. By the observation above, the vertices $(v_i)_{i\le m(n)}$ form a box $\widetilde \Lambda_{m(n)}$ of volume  $m(n)$ of $\Z^d_1$. Then $H$ stochastically dominates a site-bond percolation of $\Z_d^1$ in $\widetilde \Lambda_{m(n)}$.
 
 More precisely, since vertices of $G$ are active independently with probability at least $1-\varepsilon$, and edges of $G$ between adjacent vertices are present conditionally independently again with probability at least $1-\varepsilon$ in $H$,  each edge in the renormalised lattice $\widetilde \Lambda_{m(n)}$ is open with probability at least $(1-\varepsilon)^3$.  The model is $1$-dependent,  since the state of any edge $\{v_i, v_j\}$ of $H$ depends only on edges sharing at least one vertex with $\{v_i, v_j\}$. 
 
 Since $\varepsilon$ can be chosen arbitrarily small, by \cite[Theorem 0.0]{LSS97} or \cite[Remark 6.2]{lyons2011indistinguishability}, the graph $H$ therefore stochastically dominates iid nearest-neighbor bond percolation $G^\star$ on $\widetilde \Lambda_{m(n)}$ with parameter $p^\star$ that can also be made arbitrarily close to $1$. Hence for the sizes of the largest connected components, there is a coupling such that 
 $|\CC_{m(n)}^{\sss{(1)}}(H)|\ge |\CC_{m(n)}^{\sss{(1)}}(G^\star)|$ holds.
 Thus,  \cite[Theorem 1.1]{deuschel1996surface} applies to $|\CC_{m(n)}^{\sss{(1)}}(G^\star)|$, 
 and so for some $c(\beta)>0$ we obtain that using \eqref{eq:box-number}
 \begin{equation}\label{eq:ldp-renorm}
\Prob\left( |\CC_{m(n)}^{\sss{(1)}}(H)|\ge \rho m(n)\right) \ge \re^{-c\,m(n)^{(d-1)/d}}\ge 1-\re^{-c(\beta)n^{(d-1)/d}}. \end{equation}
    An active $v_i\in G$ corresponds to a box $Q_i$ that an  contains at least $C(\beta)$ vertices and the graphs $\CG(Q_i)$ are connected, so it holds deterministically that
    $|\CC_n^{\sss{(1)}}(\CG_n)|\ge C(\beta) |\CC_{m(n)}^{\sss{(1)}}(H)|$. This, combined with 
    \eqref{eq:ldp-renorm} implies \eqref{eq:giant-ldp}.

 We turn now to prove \eqref{eq:origin-giant}.  Consider a smaller box $\Lambda_{2^{-d}n}$.
 Define then 
 \[Z_{\ell} :=\sum_{x\in\Lambda_{2^{-d}n}} \ind{|\CC_{2^{-d}n}(x)|\ge \ell}.\]
 We argue that $\{\CC_n^{\sss{(1)}}\ge \ell\}\subseteq \{Z_\ell \ge \ell\}$. Indeed, if the largest component is at least of size $\ell$ then in $Z_\ell$ at least $\ell$ many indicators are $1$.
 Then, using \eqref{eq:giant-ldp} with $\rho2^{-d}$ for a lower bound, and  applying a Markov's inequality with $\ell=\rho 2^{-d}n$ followed by a union bound yields that 
 \[ 
 \begin{aligned}
 1-\exp(-\rho 2^{-d} n^{(d-1)/d})\le \Prob( \CC_n^{\sss{(1)}}\ge\rho 2^{-d}n)&\le \Prob(Z_{\rho 2^{-d}n} \ge \rho 2^{-d}n )\\
& \le \frac{\E[Z_{\rho 2^{-d}n}]}{\rho 2^{-d}n}\\& \le \frac{1}{\rho 2^{-d}n}\sum_{x\in\Lambda_{2^{-d}n}} \Prob\big(|\CC_{2^{-d}n}(x)|\ge \rho 2^{-d}n\big). 
 \end{aligned}
 \]
 If for all $x\in\Lambda_{2^{-d}n}$ it would hold that $\Prob\big(|\CC_{2^{-d}n}(x)|\ge \rho 2^{-d}n\big)\le \rho/2$, then the right-hand side would be at most $1/2$, while the  left-hand side tends to $1$. 
 
 Hence, there must exist $x\in\Lambda_{2^{-d}n}$ such that $\Prob\big(|\CC_{2^{-d}n}(x)|\ge \rho 2^{-d}n\big)\ge \rho/2$. 
 Let $x\in\Lambda_{2^{-d}n}$ be such a vertex. 
Then, by the translation invariance of the infinite model $\CG_\infty$, looking at the component of the origin $\CC_{2^{-d}n}^{\sss{(-x)}}(0)$ inside the box $\Lambda_{2^{-d}n}(-x)$, it holds that 
 \[ \Prob\big(|\CC_{2^{-d}n}^{\sss{(-x)}}(0)|\ge \rho 2^{-d}n\big)=\Prob\big(|\CC_{2^{-d}n}(x)|\ge \rho 2^{-d}n\big).\]
However, the shifted box $\Lambda_{2^{-d}n}(-x)\subseteq\Lambda_n$ for any $x\in\Lambda_{2^{-d}n}$, and hence $\CC_{2^{-d}n}^{\sss{(-x)}}(0)\subseteq \CC_{n}(0)$. Hence,
we obtain $$\Prob\big(|\CC_n(0)|\ge \rho 2^{-d}n\big)\ge \Prob\big(|\CC_{2^{-d}n}(x)|\ge \rho 2^{-d}n\big)\ge\rho/2.$$ Hence, \eqref{eq:origin-giant} follows by adapting the constant $\rho$.
\end{proof}
Since both prerequisites of Proposition \ref{prop:prerequisites-upper} are satisfied, this finishes the proof of the upper bounds of Theorem \ref{thm:longrange}.

\subsection*{Lower bounds}

For the lower bound we adapt the lower bound from \cite{clusterI}, rephrased to the model of long-range percolation of Definition \ref{def:lrp}, by setting the vertex set to $\Z^d$ and all vertex marks to $1$ in \cite{clusterI}. The lower bound of cluster-size decay and second-largest component that we are about the cite --  \cite[Proposition 7.1]{clusterI} --
requires that  in a box of volume $\ell$ a linear sized (at least $\rho \ell$) giant component on vertices with marks in the interval $[1, \mathrm{polylog}(\ell)]$ exists, with probability at least $\rho>0$. Since in long-range percolation all vertex marks are identical to $1$, this requirement of \cite[Proposition 7.1]{clusterI} turns into the requirement \eqref{eq:condition-for-lower} below for LRP. Moreover, the number $\mathfrak{m}_\CZ$ from \cite[Proposition 7.1]{clusterI} equals $1$ when we restrict to long-range percolation with $\alpha>1+1/d$: in the setting without vertex marks, $\mathfrak{m}_\CZ$ counts the number of maximizers in the set $\{2-\alpha, (d-1)/d\}$.
\begin{proposition}[Lower bound {\cite[Proposition 7.1]{clusterI}}]\label{prop:lower}
 Consider supercritical long-range percolation with parameters $\alpha>1+1/d$,  $d\ge 2$, and assume that $p\wedge \beta\in(0, 1)$.
 Assume that there exists a constant $\rho>0$ such that for all $n$ sufficiently large,
 \begin{equation}\label{eq:condition-for-lower}
 \Prob\big(|\CC_n(0)|\ge \rho n\big)
  \ge
  \rho.
  \end{equation}
 Then there exists $A>0$ such that for all $n\in [Ak, \infty]$,
 \begin{equation}\label{eq:cluster-lower}
  \Prob\big(|\CC_n(0)|\ge k, 0\notin\CC_n^\sss{(1)}\big) \ge
  \exp\big(-Ak^{(d-1)/d}\big).
 \end{equation}
Moreover, there exists $ \delta, \varepsilon>0$, such that for all (finite) $n$ sufficiently large 
 \begin{equation}\label{eq:cn2-lower}
  \Prob\big(|\CC_{n}^\sss{(2)}|\le \varepsilon(\log n)^{d/(d-1)}\big)\le n^{-\delta}.
 \end{equation}
\end{proposition}
In Lemma \ref{lem:Coupling} we has just proved in \eqref{eq:origin-giant} the requirement \eqref{eq:condition-for-lower}. The requirements in Theorem \ref{thm:longrange} are more restrictive then the assumption  $p\wedge \beta<1$ here, so the lower bounds in Theorem \ref{thm:longrange} follow from Proposition \ref{prop:lower}. In particular, \eqref{eq:cn2-lower} implies the lower bound in \eqref{eq:main-second-largest}, and after taking logarithm of both sides, \eqref{eq:cluster-lower} implies the lower bound in \eqref{eq:main-cluster-decay}.
\vskip1em

\begin{acks}
The work of JJ and JK has been partly supported through grant NWO 613.009.122.  The work of DM has been partially supported by grant Fondecyt grant 1220174 and by grant GrHyDy ANR-20-CE40-0002. We thank the careful reviewer for many detailed suggestions, and in particular for the elegant new proof of Claim \ref{claim:complement-star-conn}.
\end{acks}
% \subsection{Symbols}\label{sec:symbols}
% \vspace{-13pt}
% \begin{table}[H]
% \begin{tabular}{rllrl}
% $a\wedge b$            & minimum of $a,h\in\R$               &  & \hyperlink{closure}{$\bar A$}                                 & closure of (\hyperlink{1conn}{$1$-conn.}) block $A$                 \\
% $a\vee b$              & maximum of $a,h\in\R$               &  & $\hyperlink{1el}{\in_1}$                                  & $1$-disconnected elements                            \\
% $\|\cdot\|$            & Euclidean 2-norm                    &  & $\hyperlink{extL}{\partial_\mathrm{ext}}(\cdot)$           & exterior boundary w.r.t.\ $\Lambda_n$ \\
% $u\sim_G v$            & edge-existence between $u,v$ in $G$ &  & $\hyperlink{extZ}{\widetilde\partial_\mathrm{ext}}(\cdot)$ & exterior boundary w.r.t.\ $\Z^d$      \\
% $u\leftrightarrow_G v$ & path-existence between $u,v$ in $G$ &  & $\hyperlink{intL}{\partial_\mathrm{in}}(\cdot)$           & interior boundary w.r.t.\ $\Lambda_n$ \\
% $\Lambda_n$            & $[-n^{1/d}/2, n^{1/d}/2]^d\cap\Z^d$ &  & $\hyperlink{intZ}{\widetilde\partial_\mathrm{in}}(\cdot)$ & interior boundary w.r.t.\ $\Z^d$     
% \end{tabular}
% \end{table}

\appendix
\section{Proofs of preliminary claims}
\begin{proof}[Proof of Claim \ref{claim:unique}]\phantomsection\label{sec:unique-decomp}
 Identify a path on $\Z^d_1$ with its vertex set.  Given the set $A$, we define an equivalence class $\leftrightarrow_{A,1}$ on the vertices of $A$, where $x\leftrightarrow_{A,1} y$ if and only if there is a $1$-connected path $\pi$ (of $\Z_1^d$) consisting of vertices of $A$ that connects $x$ and $y$. We then define the blocks $A_1, A_2, \dots, A_b$ as the equivalence classes of $\leftrightarrow_{A,1}$. 
 In other words, start from any vertex $x\in A$ and define its block as all vertices that $x$ is $1$-connected to using paths of only vertices of $A$ (and the edges of $\Z^d_1$), and then we iterate this over all $x\in A$, yielding the (different) blocks $A_1, A_2, \dots, A_b$.
 
Each $A_i$ is $1$-connected since every pair of vertices in $A_i$ is connected by a $1$-connected path by the definition of $\leftrightarrow_{A,1}$, i.e., $A_i$ is a block. Further, if $i\neq j$ then $\|A_i-A_j\|_1>1$ must hold, since otherwise there would be a $1$-connected path from some $x\in A_i$ to some $y\in A_j$, and that would contradict $x\not\leftrightarrow_{a,1}y$. Uniqueness of this decomposition follows because $\leftrightarrow_{A,1}$ is an equivalence relation.
 
 When the set $A$ is the vertex set of a component of $\CC$ of $\CG$, we show that the block graph $\CH_\CG((A_i)_{i\le b})$ is connected. Suppose otherwise. This means that there is a proper subset of blocks whose union is not connected to the union of all the other blocks. However, this contradicts that $\CC$ is a component of $\CG$.
\end{proof}
\begin{proof}[Proof of Claim \ref{claim:containment}]\phantomsection\label{proof:containment}
 We show that $\widetilde\partial_\mathrm{int}\bar B\subseteq \widetilde\partial_\mathrm{int}B$. We argue by contradiction. Assume that there exists $x\in\widetilde\partial_\mathrm{int}\bar B\setminus\widetilde\partial_\mathrm{int}B$. 
 Since $\widetilde\partial_\mathrm{int}\bar B\subseteq \bar B= B \cup (\cup_{H\in\mathfrak{H}_B}H)$ and $B$ is disjoint from $\bar B\setminus B$, there are two cases. Either $x\in B\setminus \widetilde\partial_\mathrm{int}B$ or $x\in \mathfrak (\cup_{H\in\mathfrak{H}_B}H)\setminus  \widetilde\partial_\mathrm{int}B$. 
 
 For the first case assume that  $x\in B\setminus \widetilde\partial_\mathrm{int}B$. Then all its $\Z^d_1$-neighboring vertices were also in $B$ by Definition \ref{def:boundary} of the interior boundary. Thus $x$ is surrounded by $B$, and hence $x\in \bar B$. Similarly,  the neighboring vertices are also in $\bar B$, contradicting that $x\in\widetilde\partial_\mathrm{int}\bar B$.
 
For the second case assume that  $x\in (\cup_{H\in\mathfrak{H}_B}H)$. Then $x$ was surrounded by $B$, but then also all its $\Z^d_1$-neighboring vertices were either a member of $B$ or surrounded by $B$, contradicting again that $x\in\widetilde\partial_\mathrm{int}\bar B$.

 We move on to part (ii).
 Assume that $\bar B_1\cap\bar B_2\neq \emptyset$, then there exists $x\in \bar B_1\cap\bar B_2$. Since $B_1$ and $B_2$ are $1$-disconnected, it is excluded that $x\in B_1\cap B_2$. Assume there exists some $x\in (\bar B_2\setminus B_2)\cap B_1$. Then, since  $x\in (\bar B_2\setminus B_2)$, $x$ is surrounded by $B_2$. Further, any vertex $y\in B_1$ must be surrounded by $B_2$, since $\|B_1-B_2\|_1\ge 2$ and $B_1$ itself is $1$-connected. Hence, $B_1\subseteq \bar B_2$. Further, vertices surrounded by $B_1$ are then also surrounded by $B_2$, so we obtain $\bar B_1\subseteq \bar B_2$. The argument when there exists some $x\in (\bar B_1\setminus B_1)\cap B_2$ follows analogously yielding that in that case $\bar B_2\subseteq \bar B_1$.
 Lastly, assume that there exists $x\in (\bar B_1\setminus B_1)\cap (\bar B_2\setminus B_2)$, i.e., $x$ is in the intersection of a hole of $B_2$ and a hole of $B_1$, in particular it is surrounded by both $B_1$ and $B_2$. Then there exists, for some $j\ge 1$, a (self-avoiding) path $\pi=(x,x_1,\ldots, x_j)$ on $\Z^d_1$ such that $x_{\ell}\in (\bar B_1\setminus B_1)\cap (\bar B_2\setminus B_2)$ for all $\ell\le j-1$ and then
  $x_j\in\widetilde\partial_\mathrm{int} B_1\cup\widetilde\partial_\mathrm{int}B_2$. That is, from $x$ we start a path $\pi$ of `hole' vertices until we hit one of the interior boundaries of sets $B_1$ or $B_2$. Assume w.l.o.g.\ that $x_j\in\widetilde \partial_\mathrm{int}B_1$. Since there were no vertices from $B_2$ on the path and $x$ is surrounded by $B_2$, it must follow that also $x_j$ is surrounded by $B_2$. Similar to the previous case, it follows that $\bar B_1\subseteq \bar B_2$.

 Part (iii) claims that  when $\bar B_1\cap\bar B_2 = \emptyset$, and initially $B_1,B_2$ are $1$-disconnected, then $\|\bar B_1-\bar B_2\|_1\ge 2$.
 By definition of  $\|\cdot\|_1$ between sets, we have
 \begin{equation}
  \|\bar B_1-\bar B_2\|_1 = \min_{x_1\in\bar B_1\setminus B_1, x_2\in\bar B_2\setminus B_2}\big\{\|B_1-B_2\|_1, \|\{x_1\}-B_2\|_1, \|B_1-\{x_2\}\|_1, \|x_1-x_2\|_1\big\}.\label{eq:distance-closures}
 \end{equation}
 Each $1$-connected path from $x_i\in \bar B_i\setminus B_i$ to any $y\notin \bar B_i$ must cross a vertex in $\widetilde\partial_\mathrm{int}\bar B_i$. Consequently,
 \[
  \|\{x_1\}-B_2\|_1\ge \|\{x_1\}-\widetilde\partial_\mathrm{int}\bar B_1\|_1 + \|\widetilde\partial_\mathrm{int}\bar B_1-B_2\|_1 \ge \|\{x_1\}-\widetilde\partial_\mathrm{int}\bar B_1\|_1 + \|\widetilde\partial_\mathrm{int} B_1-B_2\|_1 \ge 2,
 \]
 where the second inequality follows since $\widetilde\partial_\mathrm{int} \bar B_1\subseteq \widetilde\partial_\mathrm{int} B_1$ by part (i), and the third inequality since by 1-disconnectedness of $B_1$ and $B_2$ the second term on the right-hand side is at least two. The third and fourth term in the minimum in \eqref{eq:distance-closures} can be bounded similarly. It follows that $\bar B_1$ and $\bar B_2$ are $1$-disconnected.

 The fourth statement is immediate from \cite[Lemma 2.1]{deuschel1996surface} which states that the interior and exterior boundaries with respect to $\Z_1^d$ of any $\ast$-connected set are $\ast$-connected.
 
 We turn to the last statement, and show that for any hole $H$ of a block $B$, $H=\bar H$, i.e., that $H$ does not contain holes. This is true since $H$ was formed as a maximal $1$-connected subset of 
the vertices in $\Lambda_n\setminus B$ surrounded by $B$, see below \eqref{eq:closure} in Definition \ref{def:blocks}. So if there were a hole $J$ inside $H$, then $J$ must intersect $B$, which would then contradict the $1$-connectedness of $B$, since $\|J-\partial_{\mathrm{ext}}\bar H\|_1\ge 2$ as $H$ fully  surrounds $J$. The fact that the interior and exterior boundaries of a hole are $\ast$-connected follows now from Part~(iv).
\end{proof}
\begin{proof}[Proof of Claim \ref{claim:iso}]\phantomsection\label{proof:isoperimetry}
 The proof is inspired  by an argument by Deuschel and Pisztora \cite[Proof of (A.3)]{deuschel1996surface}. The inequalities with $(\star)$ in \eqref{eq:iso} follow by standard isoperimetric inequalities, but we will also derive them below.

 We will first show the bounds for $\partial_\mathrm{int}$ and $\widetilde\partial_\mathrm{int}$. At the end of the proof we adjust it to $\partial_\mathrm{ext}$ and $\widetilde\partial_\mathrm{ext}$. 
 We start by showing that there exists $\delta'>0$ such that $\partial_\mathrm{int}A\ge \delta'\widetilde\partial_\mathrm{int}A$ for all $A\subseteq \Lambda_n$ of size at most $3n/4$. 
 Fix such a set $A \subseteq\Lambda_n$.
 We recall an inequality related to the isoperimetric inequality by Loomis and Whitney \cite[Theorem 2]{loomis1943}.  For a set $A\subseteq\Lambda_n$, let $S_i:=\pi_i(A)$ denote the projection $\pi_i$ of $A$ onto the $i$-th coordinate hyperplane.
 That is, for a vertex with coordinates $x=(x_1, \dots, x_{i-1}, x_i, x_{i+1}, \dots, x_d)$ we define $\pi_i x := (x_1, \dots, x_{i-1}, 0, x_{i+1}, \dots, x_d)$.
 Then \cite[Theorem 2]{loomis1943} proves that
 \begin{equation}
  |A|^{d-1}\le \prod_{i\in[d]}|S_i|.\label{eq:loomis}
 \end{equation}
 Let $i_\star$ be the coordinate dimension that contains the largest projected set $S_{i_\star}$ (ties broken arbitrarily), so that as a result of \eqref{eq:loomis},
 \begin{equation}
  |S_{i_\star}|\ge |A|^{(d-1)/d}.\label{eq:large-proj}
 \end{equation} We abbreviate $S_\star=S_{i_\star}$, and write $\pi_{\star}$ for the the $i_\star$-th projection.
 For $s\in S_{\star}$ we define the pre-image of $s$ as
 \begin{equation}
  \pi_{\star}^{\uparrow}:=\{y \in A: \pi_{\star}(y)=s\}\nonumber%
 \end{equation}
  We now describe `fibers' of $A$, informally, where $A$ has a full $n^{1/d}$-length straight line segment on the $i_\star$th coordinate connecting two opposite faces of the box $\Lambda_n$.
 Formally, we call a vertex $s\in S_\star$ a \emph{projection of a fiber} or shortly a \emph{fiber} if there is no vertex in $\partial_{\mathrm{int}}A$ that projects to $s$ via $\pi_{\star}$, and define the set 
 \begin{equation}\label{eq:fibers}
  F:=\{s\in S_{\star}: \nexists y\in\partial_{\mathrm{int}} A\mbox{ with } \pi_{\star}(y)=s\}.
 \end{equation}
 The pre-image of any fiber  does not contain any vertex of $\partial_{\mathrm{int}} A$ within $\Lambda_n$, hence, it contains a full length-$n^{1/d}$ line $\CL_s$ connecting the two opposite faces of $\Lambda_n$, with $\pi_\star(\CL_s)= s$.
 This is because all vertices that share all coordinates with $s$ except the $i_{\star}$th coordinate, project to $s$ via $\pi_{\star}$, so $A$ must contain all of them (the possibility of $A$ containing none of $\CL_s$ is excluded by assuming $s\in S_{\star}$), otherwise there would be a boundary vertex of $A$ among them.
 Then the pre-image of any such fiber intersects the box-boundary $\widetilde\partial_{\mathrm{int}}\Lambda_n$ in exactly $2$ vertices, which then must be also boundary vertices of $A$ with respect to $\Z^d$:
 \[ |\pi_{\star}^{\uparrow}(s)\cap \widetilde\partial_{\mathrm{int}}\Lambda_n| = |\pi_{\star}^{\uparrow}(s) \cap \widetilde \partial_{\mathrm{int}} A| = 2, \qquad \forall s\in F.\]
By the definition in \eqref{eq:fibers}, the pre-image of vertices in $F$ are disjoint of $\partial_{\mathrm{int}} A$ (the interior boundary of $A$ with respect to $\Lambda_n$) by the definition in \eqref{eq:int-boundary}, see also the text below \eqref{eq:int-boundary}. By the definition \eqref{eq:fibers}, the pre-image of each vertex  $z\in S_\star\setminus F$ contains at least one vertex in $\partial_{\mathrm{int}}A$.
 A similar argument as the one for fibers shows that the pre-image of each vertex  $z\in S_\star\setminus F$ contains at least two vertices in $\widetilde \partial_{\mathrm{int}}A$. Namely, if $z$ is not a projection of a fiber, then the line segment $\CL_z\cap A$ must not equal $\CL_z$ and hence it contains at least one vertex pair $x, y$ so that $x\in A, y\in \Lambda_n\setminus A$. In this case $y\in \partial_{\mathrm{ext}}A$ and so $x\in \partial_{\mathrm{int}}A$. 
 Further, either there is a second such vertex pair, or if there is no second such vertex pair then $\CL_s\cap A$ must contain one vertex of $\CL_s\cap \widetilde \partial_{\mathrm{int}}\Lambda_n$, yielding that the pre-image of  $z$ contains at least two vertices of $\widetilde \partial_{\mathrm{int}}A$.
 Formally
 \begin{equation}\label{eq:min-boundary}
  |\pi_{\star}^{\uparrow}(z) \cap \partial_{\mathrm{int}} A| \ge 1 \quad \mbox{and} \quad |\pi_{\star}^{\uparrow}(z) \cap \widetilde \partial_{\mathrm{int}} A|\ge 2 \qquad \forall z\in S_\star\setminus F.
 \end{equation}
 Further, for any configuration of $A\cap \CL_z$, we see that the difference between the above two intersections is the number of vertices that  $A\cap \CL_s\cap \widetilde\partial_{\mathrm{int}}\Lambda_n$ contains. We obtain
 \begin{equation}
  |\pi_{\star}^{\uparrow}(z) \cap \widetilde \partial_{\mathrm{int}} A|-  |\pi_{\star}^{\uparrow}(z) \cap \partial_{\mathrm{int}} A| \le 2 \qquad \forall z\in S_\star\setminus F.\nonumber%
 \end{equation}
 We characterize $z\in S\setminus F$ according to this difference. For $j\in\{0,1,2\}$ we define 
 \begin{equation}
  (S_\star\setminus F)_j:=\big\{ z\in S_\star\setminus F:  |\pi_{\star}^{\uparrow}(z) \cap \widetilde \partial_{\mathrm{int}} A|-  |\pi_{\star}^{\uparrow}(z) \cap \partial_{\mathrm{int}} A| =j \big\}.\nonumber%
 \end{equation}
 Then we can count all vertices in $\widetilde \partial_\mathrm{int}A$ according to their projection via $\pi_\ast$, and obtain that
 \begin{equation}
  \begin{aligned}
   |\widetilde \partial_{\mathrm{int}} A| & = \sum_{s\in F}  |\pi_{\star}^{\uparrow}(s) \cap \widetilde \partial_{\mathrm{int}} A| +
   \sum_{j\in\{0,1,2\}}\sum_{z\in (S_\star\setminus F)_j}  |\pi_{\star}^{\uparrow}(z) \cap \widetilde \partial_{\mathrm{int}} A|
   \\
& = 2|F| +  \sum_{j\in\{0,1,2\}}\sum_{z\in (S\setminus F)_j}  |\pi_{\star}^{\uparrow}(z) \cap \widetilde \partial_{\mathrm{int}} A|
 \\
 &  = 2|F| +  \sum_{j\in\{0,1,2\}}\sum_{z\in (S_\star\setminus F)_j}  (|\pi_{\star}^{\uparrow}(z) \cap \partial_{\mathrm{int}} A|+j).
  \end{aligned}\nonumber%
 \end{equation}
 Now we consider the ratio of the boundaries, i.e.,
 \begin{equation}
  \begin{aligned}
   \frac{|\partial_{\mathrm{int}}A|}{|\widetilde\partial_{\mathrm{int}}A|}
    & =\frac{\sum_{j\in\{0,1,2\}}\sum_{z\in (S_\star\setminus F)_j} |\pi_{\star}^{\uparrow}(z) \cap  \partial_{\mathrm{int}} A|}{2|F| + |(S_\star\setminus F)_1| + 2|(S_\star\setminus F)_2| + \sum_{j\in\{0,1,2\}}\sum_{z\in (S_\star\setminus F)_j}|\pi_{\star}^{\uparrow}(z) \cap   \partial_{\mathrm{int}} A|}.
  \end{aligned}\nonumber%
 \end{equation}
 Taking the set sizes $|(S_\star\setminus F)_j|, |F|$ fixed, it is elementary to see that the ratio is increasing in the summands of the double sum, i.e., its minimal value is attained when all summands are minimal. Now we use that each of the summands is at least $1$ by \eqref{eq:min-boundary}, and obtain
 \begin{equation}\nonumber%
  \frac{|\partial_{\mathrm{int}}A|}{|\widetilde\partial_{\mathrm{int}}A|} \ge \frac{|S_\star\setminus F|}{2|F| + |(S_\star\setminus F)_1| + 2|(S_\star\setminus F)_2| + |S_\star\setminus F|}.
 \end{equation}
 We bound the denominator from above, i.e.,
 \begin{equation}\nonumber%
  \Big(|F|+|S_\star\setminus F|\Big) + \Big(|F| + |(S_\star\setminus F)_1| + |(S_\star\setminus F)_2|\Big) + |(S_\star\setminus F)_2| \le 3|S_\star|,
 \end{equation}
 to obtain
 \begin{equation}\label{eq:app-fiber-ratio}
\frac{|\partial_{\mathrm{int}}A|}{|\widetilde\partial_{\mathrm{int}}A|} \ge \frac{|S_\star\setminus F|}{3|S_\star|}.
 \end{equation}
 We focus on the ratio on the right-hand side. The idea here is that our two assumptions that $|A|<3n/4$ and that $S^\star$ is the largest projection set together prevents the relative ratio of fibers in $S_\star$ to be large. For each vertex $s\in F$, there are $n^{1/d}$-many vertices in $A$ that are projected onto it (namely $\CL_s$).
 Using also $|S_\star|\ge |A|^{(d-1)/d}$ by \eqref{eq:large-proj} and $n\ge(4/3)|A|$, we get
 \begin{equation*}
  |A|\ge |F|n^{1/d} = \frac{|F|}{|S_\star|}|S_\star|n^{1/d}\ge \frac{|F|}{|S_\star|}|A|^{(d-1)/d}n^{1/d}\ge (4/3)^{1/d}\frac{|F|}{|S_\star|}|A|.
 \end{equation*}
 After rearranging we obtain  $|F|\le (3/4)^{1/d}|S_\star|$, and using this in \eqref{eq:app-fiber-ratio}, we obtain that $$
  |\partial_\mathrm{int} A|\ge \big(1-(3/4)^{1/d}\big)|\widetilde\partial_\mathrm{int} A|/3,$$
  which finishes the proof of the first inequality in \eqref{eq:iso} with $\delta=(1-(3/4)^{1/d})/3$. The left-hand inequality with $(\star)$ in \eqref{eq:iso} is not even sharp, and follows immediately from \eqref{eq:large-proj}, since each projected vertex corresponds to at least two interior boundary vertices with respect to $\Z^d$, i.e., $|\widetilde\partial_\mathrm{int} A| \ge 2 |S_\star| \ge 2 |A|^{(d-1)/d}$.

 We turn to the inequality concerning  $\partial_\mathrm{ext}A$ and $\widetilde\partial_\mathrm{ext}A$ in \eqref{eq:iso}. The inequality with $(\star)$ in \eqref{eq:iso} holds for the same reason as for $\widetilde \partial_\mathrm{int}A$. Namely, each projected vertex corresponds to at least two exterior boundary vertices with respect to $\Z^d$, i.e., $|\widetilde\partial_\mathrm{ext} A| \ge 2 |S_\star| \ge 2 |A|^{(d-1)/d}$.

 To obtain a lower bound on the ratio $|\partial_\mathrm{ext}A|/|\widetilde\partial_\mathrm{ext}A|$, we use that each exterior boundary vertex is within distance one from an interior boundary vertex, which holds for both for $\partial$ and $\widetilde\partial$. Since each vertex has at most $2d$ vertices within distance one, it follows that
 \[
  |\partial_\mathrm{ext}A|\ge \frac{1}{2d}\cdot|\partial_\mathrm{int}A| \ge \frac{1}{2d}\cdot \frac{1-(3/4)^{1/d}}{3}\cdot|\widetilde\partial_\mathrm{int}A| \ge \frac{1}{2d}\cdot \frac{1-(3/4)^{1/d}}{3}\cdot \frac{1}{2d}\cdot|\widetilde\partial_\mathrm{ext}A|,
 \]
 and the proof is finished for $\delta=(2d)^{-2}\big(1-(3/4)^{1/d}\big)/3$.
\end{proof}

\begin{proof}[Proof of Lemma \ref{lemma:peierl}]\phantomsection\label{proof:peierl}
 We first show that there exists $c_\mathrm{pei}'>0$ such that for all $x\in\Z^d$ and $m\in\N$
 \begin{equation}
  |\{A\subseteq\Lambda_n:  A \ni x, |A|=m, A\mbox{ is $\ast$-connected}\}|                                   \le \exp(c_\mathrm{pei}'m).\label{eq:peierl-1}
 \end{equation}
 Let $A$ be in the set on the left-hand side. Since $A$ is $\ast$-connected, the induced subgraph $\Z^d_\infty[A]$ contains a spanning tree containing $x$, which can be associated to a walk on the spanning tree (for example, walking through the tree in depth-first order), visiting each vertex in $A$ at most twice. Since the degree of any vertex is $3^d-1$ in $\Z_\infty^d$, the walk has at most $3^d-1$ options at each step for its next vertex, and has length at most $2m$. This shows \eqref{eq:peierl-1}.
 
 For \eqref{eq:peierl-2}, we observe that each set $A$ without holes ($A=\bar A$) can be uniquely reconstructed from its  interior boundary $\widetilde\partial_{\mathrm{int}} A$ (a  vertex is in $\bar A\setminus \widetilde\partial_\mathrm{int}A$ iff it is surrounded by $\widetilde\partial_\mathrm{int}A$), which is $\ast$-connected \cite[Lemma 2.1]{deuschel1996surface}. 
 Since we assume $|\widetilde\partial_{\mathrm{int}} A|=m$, the isoperimetric inequality \eqref{eq:iso} ensures that $|A|\le C_1 m^{d/(d-1)}$ for some $C_1>0$ for all $m\in \N$.
 This interior boundary must either contain $x$ or surround $x$ as defined in Definition \ref{def:boundary}.

 We claim that there is a constant $C>0$ such that $\|x-\widetilde\partial_\mathrm{int} A\|_2\le Cm^{1/(d-1)}$ for all $x, A$ with $A\ni x$ and $A=\bar A$. Indeed, suppose otherwise. Then, on $\Z^d$, vertices in the Euclidean ball of radius $ Cm^{1/(d-1)} $ around $x$ would be contained fully in $A$ (without containing a vertex of $\widetilde\partial_\mathrm{int} A$). This would mean, for some dimension-dependent constant $c_d$,
 that $|A|\ge c_d (Cm^{1/(d-1)})^d$, which contradicts that $|A|\le  m^{d/(d-1)}$ by Claim \ref{claim:iso} when $C$ is chosen sufficiently large.

 Hence, we may find a vertex $y\in \widetilde\partial_{\mathrm{int}}A\cap \mathrm{Ball}(Cm^{1/(d-1)}, x)$  where the latter set denotes the Euclidean ball of radius $Cm^{1/(d-1)}$ around $x$. Then, since $\widetilde\partial_{\mathrm{int}}A$ is a $\ast$-connected set of size $m$, \eqref{eq:peierl-1} ensures that
 the number of possible sets $S$ that may form $\widetilde\partial_{\mathrm{int}}A$ is $\exp(c_{\mathrm{pei}}m)$.
 Summing over the possible choices of $y\in \mathrm{Ball}(Cm^{1/(d-1)},x)$, we arrive at
 \begin{equation}
  |\{A\in\CA:  A\ni x, A=\bar{A}, |\widetilde\partial_\mathrm{int}A|=m\}|\le c_d C^d m^{d/(d-1)}\exp(c_\mathrm{pei}' m).\nonumber%
 \end{equation}
 The result follows by absorbing the factor $c_dCm^{d/(d-1)}$ into the constant $c_\mathrm{pei}$.
\end{proof}

% \setglossarystyle{alttree}
% \setglossarystyle{mcolalttree}
% \glssetwidest{$u\sim_G v$}
% \printunsrtglossary[type=symbols,style=mcolindex]
% \bibliographystyle{abbrv}
% \bibliography{LRP-paper/lrp-bib}

\end{document}